\documentclass[10pt,twoside, a4paper, english, reqno]{amsart}
\usepackage[dvips]{epsfig}
\usepackage{a4wide}
\usepackage{amscd}
\usepackage{amssymb}
\usepackage{amsthm}
\usepackage{amsmath}
\usepackage{latexsym}
\usepackage{graphicx,enumitem,dsfont}

\usepackage{upref}
\usepackage[colorlinks,backref]{hyperref}
\usepackage{color}
\usepackage[usenames,dvipsnames]{xcolor}
\usepackage{pdfsync}
\theoremstyle{plain}
\newtheorem{thm}{Theorem}[section]
\theoremstyle{plain}
\newtheorem{lem}[thm]{Lemma}

\theoremstyle{definition}
\newtheorem{defi}{Definition}[section]
\newtheorem{rem}{Remark}[section]

\newtheorem*{maintheorem*}{Main Theorem}
\newtheorem*{maincorollary*}{Main Corollary}

\newcommand*{\mysqrt}[4]{\sqrt[\leftroot{#1}\uproot{#2}#3]{#4}}

\newenvironment{Assumptions}
{
\setcounter{enumi}{0}

\begin{enumerate}}
{\end{enumerate} }

\newcommand{\norm}[1]{\left\|#1\right\|}
\newcommand{\abs}[1]{\left|#1\right|}

\newcommand{\Dx}{{\Delta x}}
\newcommand{\seq}[1]{\left\{#1\right\}}
\newcommand{\R}{\ensuremath{\mathbb{R}}}
\newcommand{\N}{\ensuremath{\mathbb{N}}}
\newcommand{\E}{\ensuremath{\mathbb{E}}}

\newcommand{\supp}{\ensuremath{\mathrm{supp}\,}}
\newcommand{\goto}{\ensuremath{\rightarrow}}

\newcommand{\eps}{\ensuremath{\varepsilon}}
\newcommand{\Z}{\mathbb{Z}}

\def\N{{I\!\!N}}

\numberwithin{equation}{section} \allowdisplaybreaks

\title[rate of convergence]
{A Finite difference scheme for conservation laws \\driven by L\'{e}vy noise}

\date{\today}


\thanks{Supported in part by Indo-French Centre for Applied Mathematics (IFCAM) and Institute for the Sustainable Engineering of Fossil Resources (ISIFoR)}

\author[Ujjwal Koley]{Ujjwal Koley}
\address[Ujjwal Koley]{\newline
	Centre for Applicable Mathematics,
	Tata Instiute of Fundamental Research,
	P.O.\ Box 6503, GKVK Post Office,
	Bangalore 560065, India}
\email[]{ujjwal@math.tifrbng.res.in}

\author[Ananta K. Majee]{Ananta K. Majee}
\address[Ananta K. Majee]{\newline
	Centre for Applicable Mathematics,
	Tata Instiute of Fundamental Research,
	P.O.\ Box 6503, GKVK Post Office,
	Bangalore 560065, India}
\email[]{ananta@math.tifrbng.res.in}

\author[Guy Vallet]{Guy Vallet}
\address[Guy Vallet]{\newline 
	LMAP UMR- CNRS 5142, IPRA BP 1155, 64013 Pau Cedex, France}
\email[]{guy.vallet@univ-pau.fr}

\begin{document}
\begin{abstract}
In this paper, we analyze a semi-discrete finite difference scheme for a conservation laws driven by
a homogeneous multiplicative L\'{e}vy noise. Thanks to BV estimates, we show a compact sequence of approximate solutions, generated by the finite difference scheme, converges to the unique entropy solution of the underlying problem, as the spatial mesh size $\Delta x \mapsto 0$. Moreover, we show that the expected value of the $L^1$-difference between the approximate solution and the unique entropy solution converges at a rate $\mathcal{O}(\sqrt{\Delta x})$.  
\end{abstract}

\maketitle
\tableofcontents
\section{Introduction}
 Let $\big(\Omega, \mathbb{P}, \mathcal{F}, \{\mathcal{F}_t\}_{t\ge 0} \big)$ be a filtered probability space satisfying 
 the usual hypothesis i.e.,  $\{\mathcal{F}_t\}_{t\ge 0}$ is a right-continuous filtration such that $\mathcal{F}_0$ 
 contains all the $\mathbb{P}$-null subsets of $(\Omega, \mathcal{F})$. In this paper, we are interested in numerical approximations of an $L^2(\R)$-valued predictable process $u(t, \cdot)$ which satisfies the 
Cauchy problem
  \begin{equation}
 \label{eq:stoc_con_laws}
 \begin{cases} 
 du(t,x) + \partial_x f(u(t,x))\,dt
 =\sigma(u(t,x)) \,dW(t) + \int_{|z|>0} \eta(u(t,x);z)\widetilde{N}(dz,dt), & \quad x \in \Pi_T, \\
 u(0,x) = u_0(x), & \quad  x\in \R,
 \end{cases}
 \end{equation}
 where $\Pi_T=\R \times (0,T)$ with $T>0$ fixed, $u_0(x)$ is the given initial
 function, and $f:\R \mapsto \R$ is a given (sufficiently smooth) scalar valued flux function 
(see Section~\ref{sec:pre} for the complete list of assumptions). Moreover, $W(t)$ is a real valued Brownian noise and $ \widetilde{N}(dz,dt)= N(dz,dt)-m(dz)\,dt $, 
where $N$ is a Poisson random measure on $\R\times (0,\infty)$ with intensity measure $m(dz)$, a Radon measure on $\R \setminus \lbrace0\rbrace$
with a possible singularity at $z=0$ satisfying $\int_{|z|>0} (1\wedge |z|^2)\,m(dz) <  + \infty$.\footnote{Here we denote $ x \wedge y :=\min{\lbrace x,y \rbrace}$}
Furthermore, $ u \mapsto \sigma (u)$ and $(u,z)\mapsto \eta(u,z)$ are given real valued functions signifying
the multiplicative nature of the noise.

In the case $\sigma=\eta=0$, the equation \eqref{eq:stoc_con_laws} becomes 
a standard conservation law in $\R$.
For the conservation laws, well-posedness analysis has a very long tradition 
and it goes back to the $1950$s. 
However, we will not be able to discuss the whole literature here, 
but only mention a few landmark results.    
The question of existence and uniqueness of solutions of conservation laws was first settled in the 
pioneer papers of  Kru\v{z}kov \cite{kruzkov} and Vol'pert \cite{Volpert}. For a completely satisfactory 
well-posedness theory of conservation laws, 
we refer to the monograph of Dafermos \cite{dafermos}. See also \cite{godu} and references therein.

\subsection{Stochastic Balance Laws}   
In recent years, there has been a growing interest in the study of conservation laws driven by stochastic noise. It has been attracted the attention of many authors, and
resulted in significant momentum in the theoretical development of such problems. For some first results in that direction, see Holden and Risebro \cite{risebroholden1997}. E, Khanin, Mazel, and
Sinai \cite{khanin} described the statistical properties of the Burgers equation with Brownian noise. Kim \cite{Kim2003} extended the Kru\v{z}kov well-posedness theory to one dimensional balance laws that are driven by additive Brownian noise. However, when the noise is of multiplicative nature, one could not apply a straightforward Kru\v{z}kov's doubling method to get uniqueness. This issue was settled by Feng $\&$ Nualart \cite{nualart:2008}, who established uniqueness of entropy solution by recovering additional information from the vanishing viscosity method. The existence was proven using stochastic version of compensated compactness method and it was valid for \emph{one} spatial dimension.
To overcome this problem, Debussche $\&$ Vovelle \cite{Vovelle2010} introduced
kinetic formulation of such problems and as a result they were able to establish the well-posedness
of multidimensional stochastic balance law via \emph{kinetic} approach. 
A number of authors have contributed since then, and we mention the works of Bauzet et al. \cite{BaVaWit},
Biswas et al. \cite{BisMaj,BisKarlMaj}.
We also mention works by Chen et al. \cite{Chen:2012fk}, and Biswas et al. \cite{BisKoleyMaj}, where well-posedness of entropy solution is established in $L^p \cap BV$, via BV framework. 
Moreover, they were able to develop continuous dependence theory for multidimensional balance laws and as a by product they derived an explicit \emph{convergence rate} of the approximate solutions to the underlying problem.

Independently of the smoothness of the initial data $u_0(x)$, due to the presence of nonlinear flux term and a nonlocal term in equation \eqref{eq:stoc_con_laws}, solutions to \eqref{eq:stoc_con_laws} are not necessarily smooth and weak solutions must be sought. Before introducing the concept of weak solutions, we first recall the notion of predictable $\sigma$-field. By a predictable $\sigma$-field on $[0,T]\times \Omega$, denoted
by $\mathcal{P}_T$, we mean that the $\sigma$-field generated by the sets of the form: $ \{0\}\times A$ and $(s,t]\times B$  for any $A \in \mathcal{F}_0; B \in \mathcal{F}_s,\,\, 0<s,t \le T$.
The notion of stochastic weak solution is defined as follows:
\begin{defi} [Stochastic Weak Solution]
\label{defi:stochweaksol}
A square integrable $ L^2(\R)$-valued $\{\mathcal{F}_t: t\geq 0 \}$-predictable stochastic process $u(t)= u(t,x)$ is
called a stochastic weak solution of \eqref{eq:stoc_con_laws} if for all non-negative test functions $\psi \in C_0^{\infty}([0,T) \times \R)$,
\begin{align}
&\int_{\R} \psi(0,x) u_0(x)\,dx  + \int_{\R} \int_{0}^T \Big\{ \partial_t \psi(t,x) u(t,x)
+  f(u(t,x)) \partial_x \psi(t,x) \Big\}\,dx\,dt \notag \\ 
&\quad + \int_{\R} \int_{0}^T  \sigma(u(t,x)) \psi(t,x) \,dW(t) \,dx\,dt + \int_{0}^T \int_{|z|>0}\int_{\R}  \eta( u(t,x);z) \psi(t,x) \,dx\,\widetilde{N}(dz,dt) = 0, \quad \mathbb{P}-\text{a.s}.
\label{3w1}
\end{align}
\end{defi} 
However, it is well known that weak solutions may be discontinuous and they
are not uniquely determined by their initial data. Consequently, an entropy condition must be
imposed to single out the physically correct solution. Since the notion of entropy solution is built 
around the so called entropy-entropy flux pairs, we begin with the definition of  entropy-entropy flux pairs.
\begin{defi}[Entropy-Entropy Flux Pair]
An ordered pair $(\beta,\zeta) $ is called an entropy-entropy flux pair if $ \beta \in C^2(\R) $ with $\beta \ge0$,
	and $\zeta :\R \mapsto \R $ such that
	\begin{align*}
	\zeta'(r) = \beta'(r)f'(r), \,\,\text{for all r}.
	\end{align*}
Moreover, an entropy-entropy flux pair $(\beta,\zeta)$ is called convex if $ \beta^{\prime\prime}(\cdot) \ge 0$.  
\end{defi}
With the help of a convex entropy-entropy flux pair $(\beta,\zeta)$, the notion of stochastic entropy solution is defined as follows:
\begin{defi} [Stochastic Entropy Solution]
\label{defi:stochentropsol}
A square integrable $ L^2(\R)$-valued $\{\mathcal{F}_t: t\geq 0 \}$-predictable stochastic process $u(t)= u(t,x)$ is called a stochastic entropy solution of \eqref{eq:stoc_con_laws} provided
	\begin{itemize}
	 \item[(i)] for each $T>0$,
	 \begin{align}
	  \sup_{0\le t\le T} \E\Big[ \|u(t,\cdot)\|_2^2 \Big] < + \infty.
	 \end{align}
 \item[(ii)] For all test functions $0\leq \psi\in C_{c}^{1,2}([0,\infty )\times\R) $, and each convex entropy pair 
	$(\beta,\zeta) $, 
	\begin{align}
	&  \int_{\R} \psi(0,x) \beta(u_0(x))\,dx + \int_{\Pi_T} \Big\{ \partial_r \psi(r,x) \beta(u(r,x))
	+  \zeta(u(r,x))\partial_x \psi(r,x) \Big\}\,dx\,dr \notag \\ 
	& + \int_{\Pi_T} \sigma(u(r,x))\beta^\prime(u(r,x))\psi(r,x)\,dx\,dW(r)
+ \frac{1}{2} \int_{\Pi_T} \sigma^2(u(r,x))\beta^{\prime\prime}(u(r,x))\psi(r,x)\,dx\,dr  \notag \\
	& + \int_{0}^T\int_{|z|>0}\int_{\R} \Big(\beta\big(u(r,x) +\eta(u(r,x);z)\big)- \beta(u(r,x)\Big)\psi(r,x)\,dx\,\widetilde{N}(dz,dr) \notag \\
	 & +  \int_0^T \int_{|z|>0} \int_{\R} \int_{0}^1 (1-\lambda) 
 \eta^2(u(r,x);z)\beta^{\prime\prime}\big(u(r,x) + \lambda\, \eta(u(r,x);z)\big) \psi(r,x)\,d\lambda\,dx\,m(dz)\,dr 
 \notag \\
	 & \ge 0, \quad \mathbb{P}-\text{a.s.}\notag
	\end{align}
	\end{itemize}
\end{defi} 
Due to nonlocal nature of the  It\^{o}-L\'{e}vy  formula  and the noise-noise interaction, the Definion~\ref{defi:stochentropsol} alone does not give the $L^1$-contraction principle in the sense of average and hence the uniqueness. In fact, classical ``doubling of variable'' technique in time variable does not work when one tries to compare directly two entropy solutions defined in the sense of Definion~\ref{defi:stochentropsol}. To overcome this problem, the authors in \cite{BaVaWit, BisKarlMaj} used a more direct approach
by comparing one entropy solution against the solution of the regularized problem and subsequently sending the regularized parameter to zero, relying on ``weak compactness'' of the regularized approximations.

In order to successfully implement the direct approach, one needs to weaken the 
notion of stochastic entropy solution, and subsequently install the notion of so called generalized entropy solution (cf. \cite{BaVaWit, BisKarlMaj}).
\begin{defi} [Generalized Entropy Solution]
\label{defi: young_stochentropsol}
A square integrable $ L^2\big(\R \times (0,1)\big)$-valued $\{\mathcal{F}_t: t\geq 0 \}$-predictable stochastic process $v(t)= v(t,x,\alpha)$ is
	called a generalized stochastic entropy solution of \eqref{eq:stoc_con_laws} if 
	\begin{itemize}
	 \item[(i)] for each $T>0$,
	 \begin{align}
	  \sup_{0\le t\le T} \E\Big[ \|v(t,\cdot,\cdot)\|_2^2 \Big] < + \infty.
	 \end{align}
	 \item[(ii)] For $0\leq \psi\in C_{c}^{1,2}([0,\infty )\times\R)$, and each convex entropy pair $(\beta,\zeta) $, it holds that
	\begin{align}
	&  \int_{\R} \psi(0,x)\beta(v_0(x))\,dx + \int_{\Pi_T} \int_{0}^1  \Big( \partial_r \psi(r,x)\beta(v(r,x,\alpha)) + 
	\zeta(v(r,x,\alpha))\partial_x \psi(r,x) \Big) \,d\alpha\,dx\,dr \notag \\ 
	&  \quad +  \int_{\Pi_T} \int_{0}^1 \sigma(v(r,x,\alpha)) \beta^{\prime}(v(r,x,\alpha)) \psi(r,x)\,d\alpha \,dW(r)\,dx \notag\\
	&  \qquad + \frac{1}{2}  \int_{\Pi_T} \int_{0}^1 \sigma^2(v(r,x,\alpha)) \beta^{\prime\prime}(v(r,x,\alpha)) \psi(r,x)\,d\alpha \,dr\,dx \notag \\
	&  +  \int_{0}^T\int_{|z|>0}  \int_{\R} \int_{0}^1 \Big( \beta \big(v(r,x,\alpha) +\eta(v(r,x,\alpha);z)\big)- \beta(v(r,x,\alpha))\Big)
 \psi(r,x)\,d\alpha \,\widetilde{N}(dz,dr)\,dx \notag\\
  & +  \int_0^T \int_{|z|>0} \int_{\R} \int_{0}^1 \int_{0}^1 (1-\lambda) 
 \eta^2(v(r,x,\alpha);z)\beta^{\prime\prime}\big(v(r,x,\alpha) + \lambda\, \eta(v(r,x,\alpha);z)\big) \psi(r,x)\,d\alpha \,d\lambda\,dx\,m(dz)\,dr  \notag \\
  & \ge  0,  \quad \mathbb{P}-\text{a.s.} \notag 
\end{align}
\end{itemize}
\end{defi}

\subsection{Numerical Schemes}
Given the nonlinear nature of the stochastic balance laws, explicit solution formulas are hard to obtain. Consequently, robust numerical schemes for approximating such equations are very important in the study of stochastic balance laws. For hyperbolic equations ($\sigma=\eta=0$), the convergence analysis of difference
schemes has a long tradition, we mention only a \emph{few}
references. Finite difference schemes have been studied by Ole\u{i}nik \cite{Oleinik}, Harten \textit{et al}.~\cite{Hartenetal}, Kuznetsov \cite{Kuznetsov}, Crandall and Majda \cite{CrandallMajda} as well as many others. See also \cite{gallouet00}
and references therein.

However, only few papers are available on the study of numerical schemes for stochastic balance laws.
Within the existing literature we can refer to the paper by Holden et al. \cite{risebroholden1997},  
where the authors successfully implemented operator-splitting method to prove the existence of pathwise weak solutions for such Cauchy problems driven by Brownian noise in one space dimension.

Recently, Bauzet in \cite{bauzet2015} generalized the operator-splitting method for the same problem in multi-dimensions. Operator-splitting method for stochastic balance laws driven by multiplicative L\'{e}vy noise was studied by Koley et al. \cite{KoleyMaj}. In \cite{kroker}, Kr\"{o}ker and Rodhe established pathwise convergence of a strongly monotone semi-discrete finite volume solution towards a stochastic entropy solution for one dimensional stochastic balance laws driven by Brownian noise. The main tool was a stochastic version of the compensated compactness approach. It avoids the use of a maximum principle and total-variation estimates but restricts the study to the one-dimensional case and to the use of genuinely nonlinear flux functions. Finally, we mention that Bauzet et al. in \cite{bauzet-fluxsplitting,bauzet-finvolume} established convergence of the \emph{fully discrete} schemes for stochastic balance laws driven by Brownian noise in multi space dimension.

Despite all the above mentioned results, to the best of our knowledge, there does not exist a result giving a \emph{convergence rate} of the approximate solutions to stochastic balance laws. We remark that, for conservation laws with $u_0 \in BV$ and locally Lipschitz $f$, the convergence rate for monotone methods has long been known to be $\sqrt{\Dx}$ \cite{Kuznetsov}, and this is also optimal for discontinuous solutions.
To show the same rate for a finite difference scheme approximating \eqref{eq:stoc_con_laws} seems very difficult, and the main aim of this paper is to precisely address that issue.

\subsection{Scope and Outline of the Paper}
As we mentioned earlier, the aim of this paper is to fill the gap left by the previous authors by introducing 
a \emph{rate of convergence} result for a nonlinear scalar conservation laws forced by a L\'{e}vy noise.
In this paper, drawing preliminary motivation from \cite{BisKoleyMaj,Chen:2012fk}, we intend to prove that the expected value of the $L^1$ difference between the approximate solution and the unique entropy solution converges at a rate $\mathcal{O}(\sqrt{\Delta x})$, $\Dx$ being the spatial discretization parameter.  
Moreover, we also prove the convergence of approximate numerical solutions, generated by the semi-discrete finite difference scheme, to the unique stochastic BV-entropy solution of the underlying equation \eqref{eq:stoc_con_laws}.

The remaining part of this paper is organized as follows. We state the assumptions, detail the technical framework, introduce the semi-discrete finite difference scheme, and state the main result in Section~\ref{sec:pre}. Section~\ref{sec:apriori&convergence} deals with a priori estimates for the approximate solutions. In particular, we list a number of useful properties of the (unique) entropy solution and of the approximate solution. Section~\ref{sec:main-thm} is devoted to the proof of the main theorem, and 
the convergence of approximate solutions to the unique entropy solution of stochastic conservation laws \eqref{eq:stoc_con_laws}, with the help of Young measure theory, is presented in Appendix~\ref{sec:app}.

%%%%%%%%%%%%%%%%%%%%%%%%%%%%%%%%%%%%%%%%%%%%%%%%%%%
 \section{Technical Framework and Statement of the Main Result}
 \label{sec:pre}
Throughout this paper, we use the letter $C$ to denote various generic constants. There are situations
where constants may change from line to line, but the notation is kept unchanged so long as it does not
impact the central idea. The Euclidean norm on $\R$ is denoted by $||\cdot||$ and BV semi-norm is by $|\cdot|_{BV(\R)}$. Note that BV$(\R)$ is the set of integrable functions with bounded variation on $\R$ endowed with the norm $|u|_{BV(\R)}= \norm{u}_{1} + TV_{x}(u)$, where $TV_{x}(u)$ is the total variation of $u$ defined on $\R$. 

%Moreover, we denote by $N^2_\omega(0,T,L^2(\R))$ the space of $L^2(\Omega\times(0,T)\times\R)$ predictable $L^2(\R)$-valued processes. 
Moreover, we assume the following set of assumptions:
 \begin{Assumptions}
\item \label{A1} The initial function $u_0(x)$ is a $\mathcal{F}_0$-measurable random variable such that
\begin{align*}
 \E \Big[ ||u_0||_{2} + ||u_0||_{\infty} + |u_0|_{BV(\R)}\Big] < + \infty.
\end{align*}
\item \label{A2}  The flux function $f: \R \mapsto \R$ is a Lipschitz continuous function with $f(0)=0$.
\item \label{A3} There exists a constant  $C>0$  such that, for all $u,v \in \R $,
 \begin{align*}
 	| \sigma(u)-\sigma(v)|  \leq C\, |u-v|.
 \end{align*}
 \item \label{A31}
Moreover, we assume that $\sigma(0)=0$ and there exists $M>0$ such that $\sigma(u)=0$, for all $|u|>M$.
\item \label{A4} The L\'{e}vy measure $m(dz)$ is a Radon measure on $\R\backslash \{0\}$ with a possible singularity at $z=0$, which satisfies
\begin{align*}
    \int_{|z|>0}(1 \wedge |z|^2)\,m(dz) < \infty. 
\end{align*}
\item \label{A5} There exist positive constants $0<\lambda^* <1$ and $ C>0$, such that for all
$ u,v \in \R;~~z\in \R$
 \begin{align*} | \eta(u;z)-\eta(v;z)|  & \leq  \lambda^* |u-v|( |z|\wedge 1), \\
  \text{and}\quad|\eta(u;z)| & \le C(1+|u|)(|z|\wedge 1).
 \end{align*}
 \item \label{A6}
Furthermore, there exists a constant $M >0$ such that $\eta(u;z)=0$, for all $|u|> M$, and we assume that $\eta(0;z)= 0$, for all $z \in \R$. As a consequence, 
\begin{align*}
|\eta(u;z)|  \leq  \lambda^* |u|( |z|\wedge 1),\quad \mbox{and}\quad |u| \leq \frac{1}{1-\lambda^*} |u + \alpha \eta(u;z) |,\ \forall \alpha\in[0,1].
\end{align*}
\end{Assumptions}

\begin{rem}
Note that, in view of \cite{BaVaWit, BisKarlMaj}, for the existence and uniqueness of stochastic entropy solution of \eqref{eq:stoc_con_laws}, 
we only need initial data to be in $L^2(\R)$ along with the assumptions \ref{A2}, \ref{A3}, \ref{A4}, and \ref{A5}. 
In fact, the assumption ~\ref{A4} is necessary to handle the nonlocal term present in the equation.
Furthermore, the assumption $u_0 \in BV(\R)$ is necessary 
to make sure that the solution has finite variation.
Finally, we mention that the assumption \ref{A31}, and \ref{A6} are necessary 
conditions to maintain the boundedness of the solution. 
\end{rem}

\subsection{Finite Difference Scheme} 
We begin by introducing some notation needed to define the
semi-discrete finite difference scheme. Throughout this paper, we
reserve $\Dx$ to denote a small positive number that represents the
spatial discretization parameter of the numerical scheme. Given
$\Dx>0$, we set $x_j=j\Dx$ for $j\in \Z$, to denote the spatial mesh points. 
Moreover, for any function $u =u(x,t)$ admitting point values, we write $u_j(t) = u(x_j, t)$. 
Furthermore, let us introduce the spatial grid cells
\begin{align*}
  I_j = [x_{j-1/2}, x_{j+1/2}),
\end{align*}
where $x_{j\pm1/2} = x_j \pm \Dx/2$. Let $D_{\pm}$ denote the discrete
forward and backward differences in space, i.e.,
\begin{equation*}
  D_{\pm}u_j = \pm \frac{u_{j\pm 1} - u_j}{\Dx}.
\end{equation*}
The discrete summation-by-parts formula is given by
\begin{align*}
\sum \limits_{j\in \Z} u_j D_{\pm} v_j = -  \sum \limits_{j\in \Z} v_j D_{\mp} u_j.
\end{align*}

We propose the following semi-discrete (in time) finite difference scheme 
approximating the solutions generated by the equation \eqref{eq:stoc_con_laws}

\begin{align}
  du_j(t) + \frac{1}{\Delta x} \Big(F_{j+\frac{1}{2}}(t)-F_{j-\frac{1}{2}}(t) \Big)\,dt & = \sigma(u_j(t))\,dW(t)  + \int_{|z|>0} \eta(u_j(t);z)\widetilde{N}(dz,dt),\,\,t>0, \, j \in\Z,
  \label{eq:levy_stochconservation_ discrete_laws_1}  \\
  u_j(0) &=\frac{1}{\Delta x}\int_{x_{j-\frac{1}{2}}}^{x_{j+\frac{1}{2}}} u_0(x)\,dx, \, j \in\Z, \label{eq:scheme_initial}
  \end{align}
where $F_{j+\frac{1}{2}}(t):= F\big(u_j(t),u_{j+1}(t)\big)$ is the Engquist-Osher (EO) flux. More precisely, for any given flux function $f$, the generalized upwind scheme of Engquist and Osher is defined by 
\begin{align*}
F\big(u_j,u_{j+1} \big) = f^{+}(u_j) + f^{-}(u_{j+1}),
\end{align*}
where
\begin{align*}
f^{+}(u) = f(0) + \int_0^u \max(f'(s),0) \,ds, \quad f^{-}(u) = \int_0^u \min(f'(s),0) \,ds. 
\end{align*}
%using the conventional notations that $a^{+} = \max{(a,0)}$, and $a^{-} = \min{(a,0)}$.
With this in mind, we can recast our scheme \eqref{eq:levy_stochconservation_ discrete_laws_1} as
\begin{align*}
  du_j(t) + \Big(D_{-}f^{+}(u_j(t)) + D_{+}f^{-}(u_j(t)) \Big)\,dt & = \sigma(u_j(t))\,dW(t)  + \int_{|z|>0} \eta(u_j(t);z)\widetilde{N}(dz,dt),\,\,t>0, \, j \in\Z, \notag  \\
  u_j(0) &=\frac{1}{\Delta x}\int_{x_{j-\frac{1}{2}}}^{x_{j+\frac{1}{2}}} u_0(x)\,dx, \, j \in\Z.  \label{eq:levy_stochconservation_ discrete_laws} 
  \end{align*}
  
\begin{rem}
\label{rem:EO}
We have chosen to analyse the scheme \eqref{eq:levy_stochconservation_ discrete_laws_1}-\eqref{eq:scheme_initial} with EO flux because of its apparent simplicity. One can, however, adopt the method of proof developed in this paper and obtain similar results for other schemes (e.g., all monotone schemes).
\end{rem}

For a given initial data $u_0$, we define the initial grid function $\{u_j(0)\}_{j\in\Z}$ by \eqref{eq:scheme_initial}. Moreover, for the sequence $\seq{u_j(t)}_{j\in \Z}$, we associate the function $u_\Dx$
defined by
\begin{equation}
\label{eq:app-solutions_1}
\begin{aligned}
u_{\Dx}(t,x)= \sum_{j \in \Z} u_j(t) \, \mathds{1}_{I_j}(x),
\end{aligned}
\end{equation}
where $\mathds{1}_{A}$ denotes the characteristic function of the set
$A$. 

Throughout this paper, we use the notation $u_{\Dx}$ to
denote the functions associated with the sequence $\seq{u_j(t)}_{j\in\Z}$. For later use, recall that the discrete
$\ell^{\infty}(\R)$, $\ell^1(\R)$ and $\ell^p(\R)$ $(1<p<\infty)$ norms, and BV semi-norm for a lattice function $u_{\Dx}$ are defined
respectively as
\begin{equation*}
  \begin{aligned}
    & \norm{u_{\Dx}(\cdot,t)}_{\infty} = \sup_{j \in \Z} \abs{u_j(t)}, \quad
     \norm{u_\Dx(\cdot, t)}_{1} = \Dx \sum_{j\in\Z} \abs{u_j(t)}, \\
    & \norm{u_\Dx(\cdot, t)}_{p} = \mysqrt{-2}{6}{p}{\Dx \sum_{j\in\Z} \abs{u_j(t)}^p}, \quad
    \abs{u_\Dx(\cdot, t)}_{BV} = \sum_{j\in\Z} \abs{u_{j+1}(t)- u_j(t)}.
  \end{aligned}
\end{equation*}

%\textcolor{red}{Observe that, stochastic calculus for ordinary differential equations yields that  for a.s. $\omega \in \Omega$,  and for all $t\in [0,T]$
%  \begin{align*}
%  u_{\Delta x}(\omega, t,\cdot) \in  L^1(\R).
%  \end{align*}
% Should we include a proof of the fact that for any $t$, $\E\left[ \|u_{\Delta x}(t)\|_{L^1(\R)}\right]$ is bounded?? Not classical?}

Now we are in a position to state the main result of this article.
\begin{maintheorem*}
  	Let the assumptions \ref{A1}-\ref{A6} hold and $u_{\Delta x}(t,x)$ be finite difference approximations generated by the scheme \eqref{eq:levy_stochconservation_ discrete_laws_1}. Then there exist an unique BV entropy solution $u(t,x)$ to the problem \eqref{eq:stoc_con_laws} and a constant $C=C(|u_0|_{BV(\R)})$, independent of the small parameter 
  	$\Delta x$, such that  for a.e. $t\in (0,T]$,
\begin{align*}
  	\E \Big[\int_{\R}\big|u_{\Delta x}(t,x)-u(t,x)\big|\,dx\Big] \le C \sqrt{\Delta x},
\end{align*}
provided the initial error is bounded by 
  \begin{align}
  \E \Big[ \| u_{\Delta x}(0,\cdot)-u_0(\cdot)\|_{L^1(\R)}\Big] \le C \sqrt{\Delta x}. \label{error:initial-data}
  \end{align}
\end{maintheorem*}

%%%%%%%%%%%%%%%%%%%%%%%%%%%%%%%%%%%%%%%%%%%%%%%%%%%%%%%%%%%%%%%%%%%%%%%%%%%%%%%%%%%%%%%%%%%%%%%%%%%%%%%%%%%%%%%%%%%%%%%%%

\section{A Priori Estimates}
\label{sec:apriori&convergence}
This section is devoted to the derivation of \emph{a priori} estimates for the approximate solutions $u_{\Delta x}(t,x)$ under the assumptions \ref{A1}-\ref{A6}, which turns out to be useful to prove convergence, as well as rate of convergence, of the approximate solutions to the unique entropy solution of \eqref{eq:stoc_con_laws}.

\begin{rem}  
To keep the presentation fairly short, we shall only present proof of \textit{a priori} estimates (and Main Theorem) for $f' \le 0$. 
Observe that, in that case $F_{j+1/2} = f(u_{j+1})$. However, we note that the same proofs apply \emph{mutatis mutandis} also in the general case, since we can write $F_{j+1/2} = f^{+}(u_{j}) + f^{-}(u_{j+1}) $, where $(f^{+})' \ge 0$, $(f^{-})' \le 0$, and both are Lipschitz continuous. In fact, it is clear that $\norm{f'}_{\infty}$ serves as a Lipschitz constant for Engquist-Osher flux. 
\end{rem}
  
\subsection{Uniform Moment Estimates} 
As we mentioned earlier, to ensure the convergence of approximate solutions, one needs to obtain uniform moment estimates for the approximate solutions. We remark that, we are following the same strategy as depicted in \cite{kroker} (for the Brownian noise set-up) to obtain such estimates. In what follows, we start with the following simple but useful lemma which is essentially a discrete version of the entropy inequality.

\begin{lem}
\label{lem:cellentropyinequality_finitevolumescheme}
Let $(\beta,\zeta)$ be a convex entropy-entropy flux pair. Let $u_j(t)$ be the approximate solution generated by the finite volume scheme \eqref{eq:levy_stochconservation_ discrete_laws_1}. Then $u_j(t)$ satisfies the following cell entropy inequality
\begin{align*}
 d\beta(u_j(t)) &+  \frac{1}{{\Delta x}} \Big(\zeta(u_{j+1}(t))- \zeta(u_{j}(t)) \Big)\,dt \notag \\
 &\le  \sigma(u_j(t))\beta^\prime(u_j(t))\,dW(t) + \frac{1}{2} \sigma^2(u_j(t))\beta^{\prime\prime}(u_j(t))\,dt \notag \\
 & \quad  + \int_{|z|>0} \Big\{\beta\Big(u_j(t)+ \eta(u_j(t);z)\Big)-\beta(u_j(t))\Big\}\, \widetilde{N}(dz,dt) \notag \\
 & \qquad + \int_{|z|>0}\int_{\lambda=0}^1 (1-\lambda) \, \eta^2(u_j(t);z) \, \beta^{\prime\prime} \Big(u_j(t)+\lambda\,\eta(u_j(t);z)\Big)\,d\lambda\,m(dz) \, dt,
\end{align*}
for all $j\in \mathbb{Z}$ and almost all $\omega \in \Omega$.
\end{lem}

\begin{proof}
 Let $(\beta,\zeta)$ be a given convex entropy-entropy flux pair. 
 A simple application of It\^{o}-L\'{e}vy formula applied to $\beta(u_j(t))$, where
$u_j(t)$ satisfies the semi-discrete finite volume scheme \eqref{eq:levy_stochconservation_ discrete_laws_1}, leads to
 \begin{align}
   d\beta(u_j(t)) &+ \frac{1}{{\Delta x}} \beta^\prime(u_j(t)) \Big( f(u_{j+1}(t))-f(u_{j}(t)) \Big)\,dt \notag \\
   &\quad = \sigma(u_j(t))\beta^\prime(u_j(t))\,dW(t) + \frac{1}{2} \sigma^2(u_j(t))\beta^{\prime\prime}(u_j(t))\,dt \notag \\
  & \qquad +\int_{|z|>0} \Big\{\beta \Big(u_j(t)+ \eta(u_j(t);z)\Big)-\beta(u_j(t))\Big\}\, \widetilde{N}(dz,dt) \notag \\
 & \quad \qquad + \int_{|z|>0}\int_{\lambda=0}^1 (1-\lambda)\, \eta^2(u_j(t);z) \, \beta^{\prime\prime}\Big(u_j(t)+\lambda\,\eta(u_j(t);z)\Big)\,d\lambda\,
 m(dz)\,dt.
 \label{ito-discreate}
 \end{align} 
Next, observe that
 \begin{align}
  \beta^\prime(u_j(t)) & \Big( f(u_{j+1}(t)) - f(u_{j}(t)) \Big)  \notag \\
  &= \Big( \beta^\prime(u_j(t)) f(u_{j+1}(t))-\beta^\prime(u_j(t)) f(u_j(t))
  -\zeta(u_{j+1}(t)) + \zeta(u_j(t)) \Big) + \Big( \zeta(u_{j+1}(t)) -\zeta(u_{j}(t)) \Big) \notag \\
 & := \Phi_j(u_{j+1}(t)) +  \Big( \zeta(u_{j+1}(t)) -\zeta(u_{j}(t))\Big)
  \label{expressionfor numericalflux}
 \end{align} 
where $\Phi_j$ is defined as follows
 \begin{align*}
 \Phi_j(w)= \beta^\prime(u_j(t)) f(w)-\beta^\prime(u_j(t)) f(u_j(t))
  -\zeta(w) + \zeta(u_j(t))
  \end{align*}
A straightforward calculation reveals that
\begin{align*}
 \frac{d}{dw}\Phi_j(w)&= \beta^\prime(u_j(t)) f'(w)- \zeta'(w)
   = \beta^\prime(u_j(t)) f'(w) -  \beta'(w) f'(w)\\
  &=- \beta^{\prime\prime}(\xi_2)(w-u_j(t)) f'(w),\quad \xi_2 \in (w,u_j(t)).
\end{align*} 
Since $\beta$ is convex function and $f' \le 0$, we have
\begin{align*}
 \frac{d}{dw}\Phi_j(w) =
\begin{cases} \ge 0 \quad \text{if} ~ u_j(t)\le w,\\
              \le 0 \quad \text{if} ~u_j(t)\ge w.
\end{cases}
\end{align*} 
Also, $\Phi_j(u_j(t))=0$. So, $\Phi_j(w) \ge 0 $, for all $w\in \R$. Hence, we conclude that
\begin{align}
 \Phi_j(u_{j+1}(t))\ge 0.\label{errorfornumericelflux}
\end{align}
Finally, we combine \eqref{expressionfor numericalflux} and \eqref{errorfornumericelflux} in \eqref{ito-discreate} to conclude 
\begin{align*}
 d\beta(u_j(t)) & + \frac{1}{{\Delta x}} \Big( \zeta(u_{j+1}(t)) - \zeta(u_{j}(t)) \Big)\,dt \\
 &\le  \sigma(u_j(t))\beta^\prime(u_j(t))\,dW(t) + \frac{1}{2} \sigma^2(u_j(t))\beta^{\prime\prime}(u_j(t))\,dt \notag \\
  &  \quad +\int_{|z|>0} \Big\{\beta\Big(u_j(t)+ \eta(u_j(t);z)\Big)-\beta(u_j(t))\Big\}\,\widetilde{N}(dz,dt)\\
 & \qquad + \int_{|z|>0}\int_{\lambda=0}^1 (1-\lambda) \, \eta^2(u_j(t);z) \, \beta^{\prime\prime}\Big(u_j(t)+\lambda\,\eta(u_j(t);z)\Big)\,d\lambda\,m(dz)\,dt.
\end{align*}
This completes the proof of the Lemma.
\end{proof}

Now we are ready to prove uniform moment estimates. To that context, we have the following lemma:

 \begin{lem}
Let the assumptions \ref{A1}-\ref{A6} be true.
Let $u_{\Delta x}(t,x)$ be the approximate solutions generated by the finite difference scheme \eqref{eq:levy_stochconservation_ discrete_laws_1}. Then the approximate solutions 
$u_{\Delta x}(t,x)$ satisfies the following uniform moment estimates
\begin{align}
 \sup_{{\Delta x} >0}\sup_{0\le t \le T} \E \Big[||u_{\Delta x}(t,\cdot)||_p^p\Big] < +\infty, \,\, \text{for} ~~ p\in \mathbb{N},~ p\ge 1.\label{uni:moment}
\end{align}
\end{lem}

\begin{proof}
To prove the result for $p\ge 2$, we assume that $(\beta,\zeta)$ be a given convex entropy-entropy flux pair. Then, by the virtue of Lemma \ref{lem:cellentropyinequality_finitevolumescheme}, we conclude
 \begin{align}
 \label{eq:test1}
\beta(u_j(t)) &- \beta(u_j(0))  + \frac{1}{{\Delta x}} \int_0^t \Big( \zeta(u_{j+1}(s)) -\zeta(u_{j}(s)) \Big)\,ds \notag \\
\le & \int_0^t \sigma(u_j(s))\beta^\prime(u_j(s))\,dW(s) + \frac{1}{2} \int_0^t \sigma^2(u_j(s))\beta^{\prime\prime}(u_j(s))\,ds \notag \\
 & \quad +  \int_0^t \int_{|z|>0} \Big\{\beta\big(u_j(s)+ \eta(u_j(s);z)\big)-\beta(u_j(s))\Big\} \,\widetilde{N}(dz,ds) \notag \\
 &  \qquad +  \int_0^t\int_{|z|>0}\int_{\lambda=0}^1 (1-\lambda)\, \eta^2(u_j(s);z) \, \beta^{\prime\prime}\big(u_j(s)+\lambda\,\eta(u_j(s);z)\big)\,d\lambda\,
 m(dz)\,ds.
\end{align} 
Next, we multiply the above inequality \eqref{eq:test1} by $\Delta x$, sum over all $j \in \Z$, and then take expectation. Note that, an application of It\^{o}-L\'{e}vy integral reveals that
\begin{align*}
& \E\Big[\Delta x \sum_{j \in \Z} \int_0^t \sigma(u_j(s))\beta^\prime(u_j(s))\,dW(s)\Big] = 0\\
& \E\Big[\Delta x \sum_{j \in \Z} \int_0^t \int_{|z|>0} \Big\{\beta\big(u_j(s)+ \eta(u_j(s);z)\big)-\beta(u_j(s))\Big\} \,\widetilde{N}(dz,ds)\Big] =0.
\end{align*}
Hence, we are left with
\begin{align}
  \E\Big[ \Delta x & \sum_{j\in \mathbb{Z}}  \big(\beta(u_j(t))- \beta(u_j(0))\big) \Big] \le \frac{1}{2} 
  \E\Big[ \Delta x\sum_{j\in \mathbb{Z}}  \int_0^t \sigma^2(u_j(s))\beta^{\prime\prime}(u_j(s))\,ds \Big] \notag \\
& +  \E\Big[ \Delta x\sum_{j\in \mathbb{Z}} \int_0^t\int_{|z|>0}\int_{0}^1 (1-\lambda) \eta^2(u_j(s);z)\beta^{\prime\prime}\big(u_j(s)+\lambda\,\eta(u_j(s);z)\big)\,d\lambda\, \label{eq:test2}
 m(dz)\,ds\Big].
\end{align} 
Let $p \ge 2$. Choosing $\beta(u)=\frac{|u|^p}{p}$ and using assumptions \ref{A1}-\ref{A6}, we obtain
\begin{align*}
  \E\Big[ \Dx &\sum_{j\in \mathbb{Z}} |u_j(t)|^p \Big] -  \E\Big[ \Dx \sum_{j\in \mathbb{Z}} |u_j(0)|^p \Big] \notag \\
 &\le  C(p)\, \E\Big[\Dx \sum_{j\in \mathbb{Z}}  \int_0^t\int_{|z|>0}\int_{0}^1 (1-\lambda) 
 \eta^2(u_j(s);z) |u_j(s)+\lambda\,\eta(u_j(s);z)|^{p-2}\,d\lambda\,m(dz)\,ds\Big] \notag \\
 & \hspace{3cm} + C(p)\, \E\Big[ \Dx \sum_{j\in \mathbb{Z}}  \int_0^t \sigma^2(u_j(s))|u_j(s)|^{p-2}\,ds\Big] \notag \\
 &\le  C(p)\, \E\Big[ \Dx \sum_{j\in \mathbb{Z}}  \int_0^t\int_{|z|>0}  \eta^2(u_j(s);z)
 \Big( |u_j(s)|^{p-2}+ |\eta(u_j(s);z)|^{p-2}) \Big)\,m(dz)\,ds\Big] \notag \\
  & \hspace{3cm} + C(p)\, \E\Big[ \Dx \sum_{j\in \mathbb{Z}}  \int_0^t \sigma^2(u_j(s))|u_j(s)|^{p-2}\,ds\Big] \notag \\
  &\le  C(p)\, \E\Big[ \Dx \sum_{j\in \mathbb{Z}}  \int_0^t \int_{|z|>0} |u_j(s)|^{p}(1\wedge |z|^2)\,m(dz)\,ds \Big] + C(p)\,
  \E\Big[ \Dx \sum_{j\in \mathbb{Z}}  \int_0^t |u_j(s)|^{p}\,ds \Big] \notag \\
 & \le  C(p)\,\int_0^t \E\Big[ \Dx \sum_{j\in \mathbb{Z}}|u_j(s)|^{p} \Big]\,ds  
\end{align*}
Hence, by Gronwall's inequality, for any $t>0$,
\begin{align*}
  \E\Big[ \Dx \sum_{j\in \mathbb{Z}} |u_j(t)|^p \Big] \le C\, \E\Big[ \Dx \sum_{j\in \mathbb{Z}} |u_j(0)|^p \Big].  
\end{align*}
By Jensen's inequality, we get 
$$  \E\Big[{\Delta x} \sum_{j\in \mathbb{Z}} |u_j(0)|^p \Big]\le \E\Big[||u_0(\cdot)||_p^p \Big].$$
Again, notice that
 \begin{align*}
\E\Big[||u_{\Delta x}(t,\cdot)||_p^p\Big]&= \E\Big[\int_{\R}|u_{\Delta x}(t,\cdot)|^p\,dx \Big] = \E\Big[ \sum_{j\in \mathbb{Z}}  \int_{x_{j-\frac{1}{2}}}^{x_{j+\frac{1}{2}}} |u_{\Delta x}(t,\cdot)|^p\,dx \Big]
= \E\Big[  {\Delta x}\,\sum_{j\in \mathbb{Z}} |u_j(t)|^p \Big].
\end{align*}
Thus, we conclude that
\begin{align*}
 \sup_{{\Delta x}>0}\sup_{0\le t\le T}\E\Big[||u_{\Delta x}(t,\cdot)||_p^p \Big] <  \infty,
\end{align*} 
for any $p \ge 2$. 
Choosing now $\beta= \beta_\delta$, a classical convex and regular approximation of the absolute-value function, one obtains from \eqref{eq:test2} 
\begin{align*}
  \E\Big[ \Delta x & \sum_{j\in \mathbb{Z}}  \big(\beta_\delta(u_j(t))- \beta_\delta(u_j(0))\big) \Big] \le \frac{1}{2} 
  \E\Big[ \Delta x\sum_{j\in \mathbb{Z}}  \int_0^t \sigma^2(u_j(s))1_{|u_j(s)|\leq\delta}\,ds \Big] \notag \\
& +  \E\Big[ \Delta x\sum_{j\in \mathbb{Z}} \int_0^t\int_{|z|>0}\int_{0}^1 (1-\lambda) \eta^2(u_j(s);z)1_{|u_j(s)+\lambda\,\eta(u_j(s);z)|\leq \delta}\,d\lambda\,
 m(dz)\,ds\Big].
\end{align*}
Then, by \ref{A3}, \ref{A31} and \ref{A6}, passing to the limit over $\delta$ to $0$ yields $\E\Big[ \Delta x  \sum_{j\in \mathbb{Z}}  \big(|u_j(t)|- |u_j(0)|\big) \Big] \le 0$.
Moreover, one can use interpolation theorem to conclude the assertion of the lemma for any $p \ge 1$.

This finishes the proof of the lemma.

\end{proof}  
  
%%%%%%%%%%%%%%%%%%%%%%%%%%%%%%%%%%%%%%%%%%%%%%%%%%%%%% 

\subsection{Spatial Bounded Variation} 
Like its deterministic counterpart, we derive spatial BV bound for the approximate solutions under the assumptions \ref{A1}-\ref{A6}. First note that, monotone numerical scheme for deterministic balance laws implies its $L^{\infty}$ stability. But for the stochastic case, this is not true in general. However, under the additional assumptions \ref{A31} and \ref{A6}, it is possible to derive such estimates. Regarding that, we have the following lemma.

\begin{lem}
 \label{imp}
Let the assumptions \ref{A1}-\ref{A6} be true. Let $u_{\Delta x}(t,x)$ be the finite volume approximations prescribed by the scheme \eqref{eq:levy_stochconservation_ discrete_laws_1}. Then
$\|u_{\Delta x}(\omega, t,\cdot)\|_{L^\infty(\R)}\le \max\Big\{ 2M, \|u_0\|_{L^\infty(\R)}\Big\}$, 
for any $t$ and $\omega$ a.s.,  where $M$ is defined as in the assumptions \ref{A31} and \ref{A6}. 
\end{lem}
\begin{proof}
First note that, for any $x \in (x_{j-\frac12},x_{j+\frac12})$, the scheme \eqref{eq:levy_stochconservation_ discrete_laws_1} reads
\begin{align*}
du_{\Delta x}(t,x) + \frac{1}{\Delta x}\Big[f(u_{\Delta x}(t,x+\Delta x))&-f(u_{\Delta x}(t,x))\Big] \\
&= \sigma(u_{\Delta x}(t,x)) \,dW(t) + \int_{|z|>0} \eta(u_{\Delta x}(t,x),z)\,\widetilde{N}(dz,dt).
\end{align*}
Applying It\^{o}-L\'{e}vy formula to $\beta(u_{\Delta x})$ with 
$$\beta(x)=\Big[\big(x+\max\{ 2M, \|u_0\|_{L^\infty(\R)}\}\big)^-+\big(x-\max\{ 2M, \|u_0\|_{L^\infty(\R)}\}\big)^+\Big]^2,$$ integrating with respect to $x,t$ and taking the expectation yield
\begin{align}
&\E \int_\R \beta(u_{\Delta x}(t,x))\,dx + 
\E \int_{0}^t \int_{\R}  \beta^{\prime}(u_{\Delta x}(s,x)) \frac{1}{\Delta x}\Big\{ f(u_{\Delta x}(s,x+\Delta x))-f(u_{\Delta x}(s,x))\Big\}\,dx\,ds \label{eq:121}
\\ =  & \E \int_\R \beta(u_{\Delta x}(0,x))\,dx + 
\frac12 \E\int_0^t \int_\R \sigma^2(u_{\Delta x}(s,x))\beta^{\prime\prime}(u_{\Delta x}(s,x))\,dx\,ds \notag \\
& + \E \int_0^t \int_{\R} \int_{|z|>0}\int_{\lambda=0}^1 (1-\lambda)\, \eta^2(u_{\Dx}(s,x);z) \, \beta^{\prime\prime}\big(u_{\Dx}(s,x)+\lambda\,\eta(u_{\Dx}(s,x);z)\big)\,d\lambda\, m(dz) \,dx\,ds. \notag
\end{align}
Next, observe that $\beta(u_{\Delta x}(0,x))=0$, $\sigma(\tau)\beta^{\prime\prime}(\tau)=0$, and 
$\eta(\tau)\beta^{\prime\prime}(\tau+\lambda \eta(\tau,.))=0$, for all $\tau$.
Moreover, since $\beta^\prime$ is a non-decreasing function and $f$ a non-increasing one, for any $x,y$
\begin{align*}
\beta^\prime(x)[f(y)-f(x)] \geq \int_x^y f^\prime(\tau)\beta^\prime(\tau)\,d\tau.
\end{align*}
Hence, \eqref{eq:121} leads to
\begin{align}
\E \int_\R \beta(u_{\Delta x}(t))\,dx + 
\E \int_{0}^t \int_{\R} \frac{1}{\Delta x}\Bigg\{ \int^{u_{\Delta x}(s,x+\Delta x)}_{u_{\Delta x}(s,x)} \beta^{\prime}(\tau)f^\prime(\tau)\,d\tau \Bigg\}\,dx\,ds
 \leq 0. 
\end{align}
Furthermore, since $\Dx \sum_j |u_j(s)|^2<+\infty$ for any $s$, $|u_j(s)|\to 0$ and
\begin{align*}
\int_{\R} \frac{1}{\Delta x}\Bigg\{ \int^{u_{\Delta x}(s,x+\Delta x)}_{u_{\Delta x}(s,x)} \beta^{\prime}(\tau)f^\prime(\tau)\,d\tau \Bigg\}\,dx = 
\sum_j\Bigg\{ \int^{u_{j+1}(s)}_0\beta^{\prime}(\tau)f^\prime(\tau)\,d\tau - \int^{u_j(s)}_0 \beta^{\prime}(\tau)f^\prime(\tau)\,d\tau \Bigg\}\,dx=0.
\end{align*}
Thus, for any $t$, $|u_{\Delta x}(t,x)|\leq \max\{ 2M, \|u_0\|_{L^\infty(\R)}\}$, $x$ a.e. and $\omega$ a.s.
 \end{proof} 
%%%%%%%%%%%%%%%%%%%%%%%%%%%%%%%%%%%%%%%%%%%%%%%%%%%%%%%%%%%%
 
\begin{lem}\label{lem:bv-estimate}%%%%%%%%%    BV estimate
Let the assumptions \ref{A1}-\ref{A6} be true. Let $u_{\Delta x}(t,x)$ be the finite volume approximations prescribed by the finite difference scheme \eqref{eq:levy_stochconservation_ discrete_laws_1}. Then for any $t>0$
 	 \begin{align}
 	 \E\Big[|u_{\Delta x}(t,\cdot)|_{BV(\R)}\Big] \le C \,\E\Big[|u_0(\cdot)|_{BV(\R)}\Big].\label{prop:bv}
 	 \end{align}	
 \end{lem}
 \begin{proof}
Note that, since the flux function $f$ is non-increasing i.e., $f^\prime \leq 0$, we have 
 $F_{j+\frac{1}{2}}=f(u_{j+1})$ and hence we obtain 
 \begin{align*}
  & d\big(u_{j+1}(t)-u_{j}(t)\big) + \frac{1}{\Delta x} \big( f(u_{j+2}(t))-f(u_{j+1}(t))\big)\,dt -\frac{1}{\Delta x} 
  \big( f(u_{j+1}(t))-f(u_{j}(t))\big)\,dt \\
  & \hspace{3cm}= \big( \sigma(u_{j+1}(t))-\sigma(u_j(t))\big)\,dW(t) + \int_{|z|>0} \big( \eta(u_{j+1}(t);z)-\eta(u_{j}(t);z)\big)\,\widetilde{N}(dz,dt).
 \end{align*}
 Again, since $f$ is non-increasing, we have
 \begin{align*}
  & d\big(u_{j+1}(t)-u_{j}(t)\big) - \frac{1}{\Delta x} \Big|\frac{ f(u_{j+2}(t))-f(u_{j+1}(t))}{u_{j+2}(t)-u_{j+1}(t)}\Big| \big( u_{j+2}(t)-u_{j+1}(t)\big)\,dt\\
  & \hspace{1cm}+ \frac{1}{\Delta x} \Big|\frac{ f(u_{j+1}(t))-f(u_{j}(t))}{u_{j+1}(t)-u_{j}(t)}\Big| \big( u_{j+1}(t)-u_{j}(t)\big)\,dt \\
 & = \big( \sigma(u_{j+1}(t))-\sigma(u_j(t))\big)\,dW(t)  + \int_{|z|>0} \big( \eta(u_{j+1}(t);z)-\eta(u_{j}(t);z)\big)\,\widetilde{N}(dz,dt).
 \end{align*}
Let  $\eps$ be a small positive parameter and $\beta=\beta_{\eps} \in C^2(\R)$ be a convex function approximating the absolute value function such that $\beta^\prime(r)\mapsto \mbox{sign}(r)$, as 
 $\eps \mapsto 0$. We apply It\^{o}-L\'{e}vy formula to $\beta \big(u_{j+1}(t)-u_j(t)\big)$ to conclude
 \begin{align*}
  & \beta \big(u_{j+1}(t)-u_{j}(t)\big) \\
  &=\beta \big(u_{j+1}(0)-u_{j}(0)\big) + \frac{1}{\Delta x}\int_{0}^t \beta^\prime \big(u_{j+1}(s)-u_j(s)\big)
  \Big|\frac{ f(u_{j+2}(s))-f(u_{j+1}(s))}{u_{j+2}(s)-u_{j+1}(s)} \Big| \big( u_{j+2}(s)-u_{j+1}(s)\big)\,ds\\
  & \hspace{1cm} -\frac{1}{\Delta x} \int_{0}^t \beta^\prime \big(u_{j+1}(s)-u_j(s)\big)\Big|\frac{ f(u_{j+1}(s))-f(u_{j}(s))}{u_{j+1}(s)-u_{j}(s)}\Big|
  \big( u_{j+1}(s)-u_{j}(s)\big)\,ds \\
& \hspace{1.5cm} + \int_0^t \beta^\prime \big(u_{j+1}(s)-u_j(s)\big) \big( \sigma(u_{j+1}(s))-\sigma(u_j(s))\big)\,dW(s) \\
& \hspace{2cm}+ \frac{1}{2} \int_0^t \beta^{\prime\prime} \big(u_{j+1}(s)-u_j(s)\big) \big(\sigma(u_{j+1}(s))-\sigma(u_j(s))\big)^2\,ds \\
&   + \int_0^t \int_{|z|>0} \int_0^1  \beta^\prime \Big(\big(u_{j+1}(s)-u_j(s)\big) +
\lambda \big( \eta(u_{j+1}(s);z)-\eta(u_{j}(s);z)\big)\Big) \notag \\
& \hspace{8.5cm} \times   \big( \eta(u_{j+1}(s);z)-\eta(u_{j}(s);z)\big)\,d\lambda \,\widetilde{N}(dz,ds) \notag \\
&   + \int_0^t \int_{|z|>0} \int_0^1 (1-\lambda) \beta^{\prime\prime} \Big(\big(u_{j+1}(s)-u_j(s)\big) +
\lambda \big( \eta(u_{j+1}(s);z)-\eta(u_{j}(s);z)\big)\Big)  \notag \\
& \hspace{8.5cm} \times \big( \eta(u_{j+1}(s);z)-\eta(u_{j}(s);z)\big)^2\,d\lambda \,m(dz)\,ds.
 \end{align*}
 Summing over $j$ from $-N$ to $N$ where $N\in \mathbb{N}$ and then taking expectation, we obtain
 \begin{align}
  & \E\Big[ \sum_{j=-N}^{N}\beta \big(u_{j+1}(t)-u_{j}(t)\big)\Big]-\E\Big[ \sum_{j=-N}^{N} \beta \big(u_{j+1}(0)-u_{j}(0)\big)\Big] \notag \\
  &= \frac{1}{\Delta x}  \E\Big[ \sum_{j=-N}^{N} \int_{0}^t \beta^\prime \big(u_{j+1}(s)-u_j(s)\big)
  \Big|\frac{ f(u_{j+2}(s))-f(u_{j+1}(s))}{u_{j+2}(s)-u_{j+1}(s)} \Big| \big( u_{j+2}(s)-u_{j+1}(s)\big)\,ds\Big]\notag \\
  & \hspace{1cm} -\frac{1}{\Delta x}  \E\Big[ \sum_{j=-N}^{N} \int_{0}^t \beta^\prime \big(u_{j+1}(s)-u_j(s)\big)\Big|\frac{ f(u_{j+1}(s))-f(u_{j}(s))}{u_{j+1}(s)-u_{j}(s)}\Big|
  \big( u_{j+1}(s)-u_{j}(s)\big)\,ds\Big] \notag \\
& \hspace{2cm}+ \frac{1}{2}  \E\Big[ \sum_{j=-N}^{N} \int_0^t \beta^{\prime\prime} \big(u_{j+1}(s)-u_j(s)\big) \big(\sigma(u_{j+1}(s))-\sigma(u_j(s))\big)^2\,ds\Big] \notag  \\
& +  \E\Big[ \sum_{j=-N}^{N} \int_0^t \int_{|z|>0} \int_0^1 (1-\lambda) \beta^{\prime\prime} \Big(\big(u_{j+1}(s)-u_j(s)\big) +
\lambda \big( \eta(u_{j+1}(s);z)-\eta(u_{j}(s);z)\big)\Big)  \notag \\
&   \hspace{4cm} \times \big( \eta(u_{j+1}(s);z)-\eta(u_{j}(s);z)\big)^2\,d\lambda \,m(dz)\,ds\Big] \notag \\
&\le \frac{1}{\Delta x}  \E\Big[ \sum_{j=-N}^{N} \int_{0}^t \beta^\prime \big(u_{j+1}(s)-u_j(s)\big)
\Big|\frac{ f(u_{j+2}(s))-f(u_{j+1}(s))}{u_{j+2}(s)-u_{j+1}(s)} \Big| \big( u_{j+2}(s)-u_{j+1}(s)\big)\,ds\Big] \notag \\
& \hspace{1cm} -\frac{1}{\Delta x}  \E\Big[ \sum_{j=-N}^{N} \int_{0}^t \beta^\prime \big(u_{j+1}(s)-u_j(s)\big)\Big|\frac{ f(u_{j+1}(s))-f(u_{j}(s))}{u_{j+1}(s)-u_{j}(s)}\Big|
\big( u_{j+1}(s)-u_{j}(s)\big)\,ds\Big] \notag  \\
& \hspace{2cm} + C  \E\Big[ \sum_{j=-N}^{N} \int_0^t \beta^{\prime\prime} \big(u_{j+1}(s)-u_j(s)\big) \big(u_{j+1}(s)-u_j(s)\big)^2\,ds\Big] \notag  \\
&   + C \E\Big[ \sum_{j=-N}^{N} \int_0^t \int_{|z|>0} \int_0^1 (1-\lambda) \beta^{\prime\prime} \Big(\big(u_{j+1}(s)-u_j(s)\big) +
\lambda \big( \eta(u_{j+1}(s);z)-\eta(u_{j}(s);z)\big)\Big)  \notag \\
&  \hspace{4cm} \times \big(u_{j+1}(s)-u_{j}(s)\big)^2 (1\wedge |z|^2)\,d\lambda \,m(dz)\,ds\Big] \notag \\
& \le \frac{1}{\Delta x}  \E\Big[ \sum_{j=-N}^{N} \int_{0}^t \beta^\prime \big(u_{j+1}(s)-u_j(s)\big)
  \Big|\frac{ f(u_{j+2}(s))-f(u_{j+1}(s))}{u_{j+2}(s)-u_{j+1}(s)} \Big| \big( u_{j+2}(s)-u_{j+1}(s)\big)\,ds\Big] \notag \\
  & \hspace{.1cm} -\frac{1}{\Delta x}  \E\Big[ \sum_{j=-N}^{N} \int_{0}^t \beta^\prime \big(u_{j+1}(s)-u_j(s)\big)\Big|\frac{ f(u_{j+1}(s))-f(u_{j}(s))}{u_{j+1}(s)-u_{j}(s)}\Big|
  \big( u_{j+1}(s)-u_{j}(s)\big)\,ds\Big] + C\, \eps\,t\,N \notag \\
  &  + C \E\Big[ \sum_{j=-N}^{N} \int_0^t \int_{|z|>0} \int_0^1  \beta^{\prime\prime} \Big(\big(u_{j+1}(s)-u_j(s)\big) +
\lambda \big( \eta(u_{j+1}(s);z)-\eta(u_{j}(s);z)\big)\Big)  \notag \\
&  \hspace{4cm} \times \big(u_{j+1}(s)-u_{j}(s)\big)^2(1\wedge |z|^2)\,d\lambda \,m(dz)\,ds\Big] \notag \\
&  := \frac{1}{\Delta x}  \E\Big[ \sum_{j=-N}^{N} \int_{0}^t \beta^\prime \big(u_{j+1}(s)-u_j(s)\big)
  \Big|\frac{ f(u_{j+2}(s))-f(u_{j+1}(s))}{u_{j+2}(s)-u_{j+1}(s)} \Big| \big( u_{j+2}(s)-u_{j+1}(s)\big)\,ds\Big] \notag \\
  & \hspace{1cm} -\frac{1}{\Delta x}  \E\Big[ \sum_{j=-N}^{N} \int_{0}^t \beta^\prime \big(u_{j+1}(s)-u_j(s)\big)\Big|\frac{ f(u_{j+1}(s))-f(u_{j}(s))}{u_{j+1}(s)-u_{j}(s)}\Big|
  \big( u_{j+1}(s)-u_{j}(s)\big)\,ds\Big]  \notag \\
  &  \hspace{4cm}+ C\, \eps\,t\,N + \mathcal{A}, \label{esti:discrete-bv-1}
 \end{align}
where we have used the assumption \ref{A3}, along with the fact that for $\beta=\beta_{\eps}$, $r^2 \beta^{\prime\prime}(r)\le C\,\eps$.
Let us consider the term $\mathcal{A}$. By our assumption \ref{A5}, since $\lambda \in [0,1]$, one has that
\begin{align*}
\Big|\big(u_{j+1}(s)-u_j(s)\big) + \lambda \big( \eta(u_{j+1}(s);z)-\eta(u_{j}(s);z)\big)\Big| 
\geq 
\Big|\big(u_{j+1}(s)-u_j(s)\big)\Big| \left(1 - \lambda^*\right).
\end{align*}
 % \begin{align*}
%  \big|\eta(u_{j+1}(s);z)-\eta(u_{j}(s);z)\big| \le \lambda^* |u_{j+1}(s)-u_{j}(s)| = \lambda^* a. 
% \end{align*}
%  Thus, we get \ananta{ $ -\lambda^*\, a \le \eta(u_{j+1}(s);z)-\eta(u_{j}(s);z)$} and hence 
% \begin{align*}
%  0\le a \le (1-\lambda^*)^{-1} \big( a + \lambda \big( \eta(u_{j+1}(s);z)-\eta(u_{j}(s);z)\big)\big), \,\, \text{for all}\,\, \lambda \in [0,1]. 
% \end{align*}
 Since $r^2 \beta^{\prime\prime}(r) \le C \eps$, we see that, by assumption \ref{A4}
 \begin{align}
  \mathcal{A} &= C \E\Big[ \sum_{j=-N}^{N} \int_0^t \int_{|z|>0} \int_0^1  \beta^{\prime\prime} \Big(\big(u_{j+1}(s)-u_j(s)\big) +
\lambda \big( \eta(u_{j+1}(s);z)-\eta(u_{j}(s);z)\big)\Big) \notag \\
& \hspace{4cm} \times \big(u_{j+1}(s)-u_{j}(s)\big)^2 (1\wedge |z|^2)\,d\lambda \,m(dz)\,ds\Big] \notag \\
& \le C (1-\lambda^*)^{-2} \eps\,  \E\Big[ \sum_{j=-N}^{N} \int_0^t \int_{|z|>0} (1\wedge |z|^2)\,d\lambda \,m(dz)\,ds\Big] 
 \le  C (1-\lambda^*)^{-2} \eps\,t\,N. \label{esti:extra-jump term}
 \end{align}
Thus, using \eqref {esti:extra-jump term} in \eqref{esti:discrete-bv-1} and then sending the limit as $\eps \goto 0$, we get
 \begin{align}
\E\Big[ \sum_{j=-N}^{N}\big|u_{j+1}(t)-u_{j}(t)\big|\Big]& -\E\Big[ \sum_{j=-N}^{N} \big|u_{j+1}(0)-u_{j}(0)\big|\Big] \notag \\
& \le \frac{1}{\Delta x}  \E\Big[ \sum_{j=-N}^{N} \int_{0}^t
  \Big|\frac{ f(u_{j+2}(s))-f(u_{j+1}(s))}{u_{j+2}(s)-u_{j+1}(s)} \Big| \big|u_{j+2}(s)-u_{j+1}(s)\big|\,ds\Big] \notag \\
  & \hspace{1cm} -\frac{1}{\Delta x} \E\Big[ \sum_{j=-N}^{N} \int_{0}^t \Big|\frac{ f(u_{j+1}(s))-f(u_{j}(s))}{u_{j+1}(s)-u_{j}(s)}\Big|
  \big| u_{j+1}(s)-u_{j}(s)\big|\,ds\Big] \notag \\
  & = \frac{1}{\Delta x}  \E\Big[ \int_{0}^t
  \Big|\frac{ f(u_{N+2}(s))-f(u_{N+1}(s))}{u_{N+2}(s)-u_{N+1}(s)} \Big| \big|u_{N+2}(s)-u_{N+1}(s)\big|\,ds\Big] \notag \\
  & \hspace{1cm} -\frac{1}{\Delta x} \E\Big[ \int_{0}^t \Big|\frac{ f(u_{-N+1}(s))-f(u_{-N}(s))}{u_{-N+1}(s)-u_{-N}(s)}\Big|
  \big| u_{-N+1}(s)-u_{-N}(s)\big|\,ds\Big] \notag \\
  & \le \frac{1}{\Delta x}  \E\Big[ \int_{0}^t
  \Big|\frac{ f(u_{N+2}(s))-f(u_{N+1}(s))}{u_{N+2}(s)-u_{N+1}(s)} \Big| \big|u_{N+2}(s)-u_{N+1}(s)\big|\,ds\Big] \notag \\
  & \le \frac{C}{\Delta x}  \E\Big[ \int_{0}^t \big|u_{N+2}(s)-u_{N+1}(s)\big|\,ds\Big].\label{inq:discrete-bv}
 \end{align}
 By our definition of approximate finite volume solutions $u_{\Delta x}(t,x)$, we have 
 \begin{align*}
  \big|u_{N+2}(s)-u_{N+1}(s)\big| \le \frac{1}{\Delta x} \int_{\Delta x(N+\frac{1}{2})}^{\Delta x(N+\frac{3}{2})} 
  \big| u_{\Delta x}(s, \Delta x + x)-u_{\Delta x}(s,x)\big|\,.
 \end{align*}
 Again, since for almost all $\omega \in \Omega$, and a.e. $t\in [0,T]$, approximate solution $u_{\Delta x}(t,\cdot,\omega) \in L^1(\R)$, we see that 
 $ \big|u_{N+2}(s)-u_{N+1}(s)\big| \goto 0$ as $N \goto \infty$. Moreover, in view of boundedness property (cf. Lemma~\ref{imp}), one can use dominated convergence theorem to 
 conclude that right hand side of the inequality \eqref{inq:discrete-bv} goes to zero as $N\goto \infty$. Thus, passing to the limit as $N\goto + \infty$ in \eqref{inq:discrete-bv}, we obtain 
 \begin{align}
   \E\Big[ \sum_{j\in \Z}\big|u_{j+1}(t)-u_{j}(t)\big|\Big]\le \E\Big[ \sum_{j\in \Z}\big|u_{j+1}(0)-u_{j}(0)\big|\Big].\label{inq:discrete-bv-1}
 \end{align}
 Note that, in view of the lower semi-continuity property and the positivity of the total variation $TV_x$,  $\E[TV_x(u)]$ makes sense for 
 any $u\in L^1(\Omega \times \R)$.
 Since $u_0 \in BV(\R)$, we conclude that, for all $t>0$
 \begin{align*}
  \E\Big[ TV_x(u_{\Delta x}(t))\Big] \le  \E\Big[ TV_x(u_0)\Big]. 
 \end{align*}
Again, since $ \E\Big[\|u_{\Delta x}(t,\cdot)\|_{L^1(\R)}\Big] \le  C \,\E\Big[ \|u_0(\cdot)\|_{L^1(\R^d)}\Big] $, we arrive at the following conclusion that the approximate solution $u_{\Delta x}(t,x)$ lies in spatial BV class and satisfies \eqref{prop:bv}. This completes the proof.
 \end{proof}

%%%%%%%%%%%%%%%%%%%%%%%%%%%%%%%%%%%%%%%%%%%%%%%%%%%
\subsection{Convergence of Approximate Solutions}
Having secured all necessary a priori bounds on the approximate solutions $u_{\Dx}(t,x)$, we are
ready to prove the convergence of approximate solutions $u_{\Delta x}(t,x)$ to the entropy solutions of the underlying problem  \eqref{eq:stoc_con_laws}. We remark that, uniform moment estimates 
\eqref{uni:moment} and bounded variation estimate \eqref{prop:bv}
are not enough to ensure the compactness of the family $\{u_{\Delta x}(t,x)\}_{\Delta x >0}$ in the classical sense.
However, we use notion of Young measure in stochastic setting to prove the convergence of approximate solutions, and we have postpone the proof to Appendix~\ref{sec:app}. In fact, it is possible to prove the following theorem.

\begin{thm}
\label{thm:con}
Let the assumptions \ref{A1}-\ref{A6} hold. Then there exists a unique generalized entropy solution
$u(t,x,\alpha)$ to the initial value problem \eqref{eq:stoc_con_laws}. Moreover, $u(t,x,\alpha)$ is an independent function of variable $\alpha$ and $\bar{u}(t,x)=\int_0^1 u(t,x,\tau) d\tau = u(t,x,\alpha)$ (for all $\alpha$) is the unique stochastic entropy solution. 
\smallskip \\
More precisely, let $u_{\Delta x}(t,x)$ be a finite volume approximations as prescribed by \eqref{eq:app-solutions_1}. Then, 
 \begin{itemize}
 \item[$(a)$]  $u_{\Delta x}(t,x)\mapsto u(t,x,\alpha)$ in the sense of Young-measure.
 \item[$(b)$] $ u_{\Delta x}(t,x)\rightharpoonup \bar{u}(t,x)$ in $L^2(\Omega\times(0,T)\times\R)$ and $ u_{\Delta x}(t,x)\mapsto \bar{u}(t,x)$ in  $L_{\text{loc}}^p(\R; L^p (\Omega \times (0,T))$, for $1\le p < \infty$.
 \item[$(c)$] Furthermore, for a.e. $t \in (0,T)$, the following BV bound holds:
  $$\E\big[|\bar{u}(t,\cdot)|_{BV(\R)}\big] \le C \,\E\big[|u_0(\cdot)|_{BV(\R)}\big].$$
\end{itemize}
 \end{thm}

\begin{rem}
In the Appendix~\ref{sec:app}, we have shown that $u_{\Delta x}(t,x)$ converges 
to the unique entropy solution $\bar{u}(t,x)$ to the problem \eqref{eq:stoc_con_laws}. 
In view of the lower semi-continuity property of $TV_x$ and 
Fatou's lemma, we have, for a.e. $t>0$,
\begin{align*}
 \E\Big[ TV_x(\bar{u}(t,\cdot))\Big] \le \liminf_{\Delta x \goto 0} \E\Big[ TV_x(u_{\Delta x}(t,\cdot))\Big] \le \E\Big[ TV_x(u_0)\Big],
\end{align*}
where the last inequality follows from Lemma~\ref{lem:bv-estimate}. Thus, $\bar{u}(t,x)$ is a function of bounded variation in spatial variable, which proves part $(c)$ of the above Theorem~\ref{thm:con}.
\end{rem}

%%%%%%%%%%%%%%%%%%%%%%%%%%%%%%%%%%%%%%%%%%%%%%%%%%%%%

\subsection{Average Time Continuity}
We move on to establish the average time continuity of $u_{\Dx}$, independent of the discretization parameter
$\Dx>0$.
\begin{lem}
\label{continuity in time}%%%%%%%%%%%%%%%  Continuity in time
Let the assumptions \ref{A1}-\ref{A6} be true. Moreover, let $u_{\Delta x}(t,x)$ be the finite difference approximations prescribed by the finite difference scheme \eqref{eq:levy_stochconservation_ discrete_laws_1}. Then, for any compact subset $K \subset \R$, there exists a constant $C=C\big(|K|,E[\mbox{TV}(u_0)],T\big)$, such that for $s<t$ 
\begin{align*}
\E\Big[\int_{K} |u_{\Delta x}(t,x)-u_{\Delta x}(s,x)| dx\Big] \leq 
C\,|t-s|  + C \,\big(|K|+ \Delta x\big)  \sqrt{t-s}.
\end{align*}
Moreover, for any standard mollifier $\rho_{\delta_0}(r) = \frac{1}{\delta_0}\rho(\frac{r}{\delta_0})$ on $\R$ with $\supp(\rho) \subset [-1,0)$, we have the following estimates.
\begin{itemize}
\item [$(a)$] For the approximate solutions $u_{\Dx}(t,x)$ $(\text{cf.}\, \eqref{eq:app-solutions_1})$
\begin{align*}
 \int_0^T\int_0^T \E\Big[\int_{K} |u_{\Delta x}(t,x)-u_{\Delta x}(s,x)| dx\Big]  \rho_{\delta_0}(t-s)\,ds\,dt
 \leq  C  \sqrt{\delta_0}.
 \end{align*}
\item [$(b)$] For the unique entropy solution $u(t,x)$ of \eqref{eq:stoc_con_laws}
\begin{align*}
\int_0^T\int_0^T \E\Big[\int_{K} |u(t,x)-u(s,x)| dx\Big]  \rho_{\delta_0}(t-s)\,ds\,dt
 \leq  C  \sqrt{\delta_0}.
\end{align*}
\end{itemize}
\end{lem}
\begin{proof}
Note that, since $f$ is Lipschitz, a simple calculation using the scheme \eqref{eq:levy_stochconservation_ discrete_laws_1} reveals that for any $s<t$
\begin{align*}
& \Delta x |u_j(t)-u_j(s)| \\
\quad &\leq \int_s^t|f(u_{j+1}(\tau))-f(u_{j}(\tau))|d\tau + \Delta x \Big|\int_s^t \sigma(u_j(\tau))dW(\tau)\Big|
+ \Delta x \Big|\int_s^t \int_{|z|>0} \eta(u_j(\tau);z) \widetilde{N}(dz,d\tau)\Big|
\\
\qquad &\leq C\int_s^t|u_{j+1}(\tau)-u_{j}(\tau)|d\tau + \Delta x \Big|\int_s^t \sigma(u_j(\tau))dW(\tau)\Big|
 + \Delta x \Big|\int_s^t \int_{|z|>0} \eta(u_j(\tau);z) \widetilde{N}(dz,d\tau)\Big|.
\end{align*}
Thus, for any given compact subset $K \subset \R$ with $K \subset [-N,N]$, by \eqref{inq:discrete-bv-1} and the boundedness of $\sigma$ and $\eta$ along
with assumption \ref{A4}, 
one has that
\begin{align*}
& \E\Big[\int_{K} |u_{\Delta x}(t,x)-u_{\Delta x}(s,x)| dx\Big] \\
\leq & \E\Big[\int_{-N-\frac12}^{N+\frac12} |u_{\Delta x}(t,x)-u_{\Delta x}(s,x)| dx\Big] 
=
\E\Big[\sum_{|j|\leq N}\Delta x |u_j(t)-u_j(s)| \Big]
\\ \leq& 
C\,\E\Big[\int_s^t \sum_{j \in \N}|u_{j+1}(0)-u_{j}(0)|d\tau\Big]
+ \E\Big[\sum_{|j|\leq N} \Delta x \Big|\int_s^t \sigma(u_j(\tau))dW(\tau)\Big| \Big]  \\
& \hspace{3cm}+ 
\E\Big[\sum_{|j|\leq N} \Delta x \Big|\int_s^t \int_{|z|>0} \eta(u_j(\tau);z)\, \widetilde{N}(dz,d\tau)\Big|\Big] 
\\ \leq& 
C\,|t-s|  + \E\Big[\sum_{|j|\leq N} \Delta x  \Big]^{1/2} \E\Bigg[\sum_{|j|\leq N} \Delta x \Big|\int_s^t \sigma(u_j(\tau))\,dW(\tau)\Big|^2 \Bigg]^{1/2}
\notag \\
& \hspace{3cm} +   \E\Big[\sum_{|j|\leq N} \Delta x  \Big]^{1/2} \E\Bigg[\sum_{|j|\leq N} \Delta x \Big|\int_s^t \int_{|z|>0}\eta(u_j(\tau);z)\, \widetilde{N}(dz,d\tau)\Big|^2 \Bigg]^{1/2}
\\ \leq& 
C\,|t-s|  + \Big[(2N+1) \Delta x \Big]^{1/2}
\E\Bigg[\sum_{|j|\leq N} \Delta x \int_s^t |\sigma(u_j(\tau))|^2d\tau \Bigg]^{1/2} \\
& \hspace{3cm}+ \Big[(2N+1) \Delta x \Big]^{1/2}
\E\Bigg[\sum_{|j|\leq N} \Delta x \int_s^t \int_{|z|>0} |\eta(u_j(\tau);z)|^2 \,m(dz)\,d\tau \Bigg]^{1/2} \\ \leq& 
C\,|t-s|  + C (2N+1) \Delta x  \sqrt{t-s} \\ 
\leq&  C\,|t-s|  + C (|K|+ \Delta x)  \sqrt{t-s}.
\end{align*}
Considering a non-negative continuous function $h(t,s)$ on $[0,T]^2$, one has
\begin{align*}
\int_0^T\int_0^T \, & \E\Big[\int_{K} |u_{\Delta x}(t,x)-u_{\Delta x}(s,x)| dx\Big]  h(t,s)\,ds\,dt
\\ \leq &
C\int_0^T\int_0^T |t-s|\,h(t,s)\,ds\,dt  + C (|K|+ \Delta x) \int_0^T\int_0^T \sqrt{|t-s|}\,h(t,s)\,ds\,dt,
\end{align*}
and passing to the limit in $\Delta x$ yields
\begin{align*}
\int_0^T\int_0^T \, & \E\Big[\int_{K} |u(t,x)-u_(s,x)| dx\Big]  h(t,s)\,ds\,dt\\ 
\qquad \leq & C\int_0^T\int_0^T |t-s|\,h(t,s)\,ds\,dt  + C (|K|+ \Delta x) \int_0^T\int_0^T \sqrt{|t-s|}\,h(t,s)\,ds\,dt.
\end{align*}
The conclusion holds by setting $h(t,s)=\rho_{\delta_0}(t-s)$.
\end{proof}

%%%%%%%%%%%%%%%%%%%%%%%%%%%%%%%%%%%%%%%%%%%%%%%%%%%%%%%%%%%%%%%%%%%%%%%%%%%%%%%
 %%%%%%%%%%%%%%%%%%%%%%%%%%%%%%%%%%%%%%%%%%%%%%%%%%%%%%%%%%%%%%%%%%%%%%%%%%%%%%%%%%%%%%%%%%%%%%%%%%%%%%%%%%%%
\section{Proof of the Main Theorem}
\label{sec:main-thm}
In the previous section, we have mentioned that under the assumptions \ref{A1}-\ref{A6}, the approximate solutions $u_{\Delta x}(t,x)$ converges to the unique BV entropy solution $u(t,x)$ in $L_{\text{loc}}^p(\R; L^p (\Omega \times (0,T))$, for $1\le p < \infty$. In this section, our main aim is to estimate the $L^1$ difference between the approximate solution $u_{\Delta x}(t,x)$ and the unique entropy solution of \eqref{eq:stoc_con_laws}. To do this, we begin by introducing a class of entropy functions which will play a crucial role in the sequel.
Let $\beta:\R \rightarrow \R$ be a $C^\infty$ function satisfying 
\begin{align*}
      \beta(0) = 0,\quad \beta(-r)= \beta(r),\quad 
      \beta^\prime(-r) = -\beta^\prime(r),\quad \beta^{\prime\prime} \ge 0,
\end{align*} 
and 
\begin{align*}
	\beta^\prime(r)=\begin{cases} -1, &\quad \text{when} ~ r\le -1,\\
                               \in [-1,1], & \quad\text{when}~ |r|<1,\\
                               +1,& \quad \text{when} ~ r\ge 1.
                 \end{cases}
\end{align*} 
For any $\xi> 0$, define $\beta_\xi:\R \rightarrow \R$ by 
$\beta_\xi(r) = \xi \beta(\frac{r}{\xi})$. 
Then
\begin{align}\label{eq:approx to abosx}
 |r|-M_1\xi \le \beta_\xi(r) \le |r|, \quad 
 \text{and} \,\, \,\,|\beta_\xi^{\prime\prime}(r)| 
 \le \frac{M_2}{\xi} {\bf 1}_{|r|\le \xi},
\end{align} 
where $M_1 = \sup_{|r|\le 1}\big | |r|-\beta(r)\big |$ and 
$M_2 = \sup_{|r|\le 1}|\beta^{\prime\prime} (r)|$.
By simply dropping $\xi$, for $\beta= \beta_\xi$, ~ we define 
\begin{align*}
f^\beta(a,b)=\int_{b}^a \beta^\prime(r-b)f^\prime(r)\,dr,\quad \text{and}\quad 
{\bf f}(a,b)= \text{sign}(a-b)\big(f(a)-f(b)\big).
\end{align*} 

We also introduce a special class of functions, which plays a pivotal 
role in our analysis. To that context, let us define the set $\mathcal{K}$ consisting of non-zero 
$\zeta \in C^2(\R)\cap L^1(\R)\cap L^{\infty}(\R)$ for which there is a constant $C$ such that 
$| \zeta_x(x)| \le C \zeta(x)$, and $|\zeta_{xx}(x)| \le C \zeta(x)$. Then we have the following Lemma:

\begin{lem}
\label{imp}
Let $\zeta \in \mathcal{K}$ be any element. 
Then there exists ${\lbrace \zeta_R \rbrace}_{R>1} \subset C_c^{\infty}(\R)$ such that
\begin{align*}
\zeta_R \mapsto \zeta,  (\zeta_R)_x \mapsto  \zeta_x, \,\,\text{and}\,\,  
(\zeta_R)_{xx} \mapsto \zeta_{xx}\,\,\text{pointwise in}\,\, \R^d, \,\, \text{as}\,\, R \mapsto \infty
\end{align*}
\end{lem}

\begin{proof}
Note that, modulo a mollification step, we can assume that $\zeta \in C^{\infty}(\R)$. 
Let $\eta \in C_c^{\infty}(\R)$ be such that $0\le \eta \le 1$, and $\eta(0)=1$. 
Let us define $\zeta_R(x)= \zeta(x)\eta(x/R)$. Then a straightforward computation yields
\begin{align*}
(\zeta_R)_x(x) &=  \zeta_x(x) \eta(x/R) + \frac1R \zeta(x)  \eta_x(x/R), \\
(\zeta_R)_{xx}(x) &= \zeta_{xx}(x) \eta(x/R) + \frac{1}{R^2} \zeta(x) \eta_{xx}(x/R)
+  \frac2R \zeta_x(x) \eta_x(x/R).
\end{align*}
Taking limit as $R \mapsto \infty$ concludes the proof.
\end{proof}

To proceed further, let $\rho$ and $\varrho$ be the standard mollifiers on $\R$ such that 
$\supp(\rho) \subset [-1,0)$ and $\supp(\varrho) = [-1,1]$. For $\delta > 0$ and $\delta_0 > 0$, 
let $\rho_{\delta_0}(r) = \frac{1}{\delta_0}\rho(\frac{r}{\delta_0})$ and
$\varrho_{\delta}(x) = \frac{1}{\delta}\varrho(\frac{x}{\delta})$. Let 
$\psi(x)\in \mathcal{K}:= \{\psi: | \psi_x| \le C \psi(x), |\psi_{xx}| \le C \psi(x)\}$ be any function. 
Then by Lemma~\ref{imp}, there exists a sequence of functions 
$\lbrace \psi_R \rbrace \subset C_c^{\infty}(\R)$ such that $\psi_R \mapsto \psi$ pointwise.
In the sequel, with a slight abuse of notations, we donte $\psi=\psi_R$. In what follows, 
we will use the following test function: For two positive constants $\delta, \delta_0 $, we define            
 \begin{align}
\label{eq:doubled-variable} \phi_{\delta,\delta_0}(t,x, s,y) = \rho_{\delta_0}(t-s) \,\varrho_{\delta}(x-y)\, \psi(s,y). 
\end{align} 

Furthermore, let $J$ be the standard symmetric 
nonnegative mollifier on $\R$ with support in $[-1,1]$ and $J_l(r)= \frac{1}{l}J(\frac{r}{l})$ 
for $l > 0$. Let $u(t,x)$ be a BV entropy solution of the problem \eqref{eq:stoc_con_laws}. We write down the entropy inequality for $u(t,x)$, based on the entropy pair $\big(\beta(\cdot-k), f^\beta(\cdot, k)\big)$, and then multiply by $J_l(u_{\Delta y}(s,y)-k)$, integrate with respect to $ s, y, k$ and take the expectation. The result is
\begin{align}
0\le  & \E \Big[\int_{\Pi_T}\int_{\R^2} \beta \big(u_0(x)-k\big)
\phi_{\delta,\delta_0}(0,x,s,y) J_l(u_{\Delta y}(s,y)-k)\,dk \,dx\,dy\,ds\Big] \notag \\
 & + \E \Big[\int_{\Pi_T} \int_{\Pi_T} \int_{\R} \beta(u(t,x)-k)\partial_t \phi_{\delta,\delta_0}(t,x,s,y)
J_l(u_{\Delta y}(s,y)-k)\,dk \,dx\,dt\,dy\,ds \Big]\notag \\ 
 & +  \E \Big[ \int_{\Pi_T} \int_{\R}\int_{\Pi_T}  \sigma(u(t,x)) \beta^{\prime}(u(t,x)-k)\phi_{\delta,\delta_0}(t,x,s,y)
 \,J_l(u_{\Delta y}(s,y)-k) \,dW(t) \,dx \,dk \,dy\,ds \Big] \notag\\
& + \frac{1}{2}  \E \Big[\int_{\Pi_T} \int_{\Pi_T}
\int_{\R}  \sigma^2(u(t,x)) \beta^{\prime \prime}(u(t,x)-k)\phi_{\delta,\delta_0}(t,x,s,y)
 \,J_l(u_{\Delta y}(s,y)-k) \,dt \,dx \,dk \,dy\,ds \Big] \notag
 \\& +  
 \E \Big[\int_{\Pi_T}\int_{\Pi_T} \int_{\R} 
 f^\beta(u(t,x),k)\partial_x \phi_{\delta,\delta_0}(t,x,s,y) J_l(u_{\Delta y}(s,y)-k)\,dk\,dx\,dt\,dy\,ds\Big] \notag \\
 & +  \E \Big[ \int_{\Pi_T} \int_{\R}\int_{\Pi_T} \int_{|z|>0} \int_0^1 \eta(u(t,x);z) \beta^{\prime}\big(u(t,x)-k + \lambda \eta(u(t,x);z)\big) 
 \phi_{\delta,\delta_0}(t,x,s,y)  \notag \\
 & \hspace{4cm}  \times J_l(u_{\Delta y}(s,y)-k)\,d\lambda\, \widetilde{N}(dz,dt) \,dx \,dk \,dy\,ds \Big] \notag\\
 & +  \E \Big[ \int_{\Pi_T} \int_{\R}\int_{\Pi_T} \int_{|z|>0} \int_0^1 (1-\lambda)\eta^2(u(t,x);z) \beta^{\prime\prime}\big(u(t,x)-k + \lambda \eta(u(t,x);z)\big) 
 \phi_{\delta,\delta_0}(t,x,s,y) \notag \\
 & \hspace{4cm} \times J_l(u_{\Delta y}(s,y)-k)\,d\lambda\,m(dz)\,dt \,dx \,dk \,dy\,ds \Big] \notag\\
& :=  I_1 + I_2 + I_3 +I_4 + I_5 + I_6 + I_7. \label{stoc_entropy_1}
\end{align}

By notation, for any $y \in (y_{j-\frac12},y_{j+\frac12})$, 
\begin{align*}
 du_{\Delta y}(s,y) + \frac{1}{\Delta y}  \Big[f(u_{\Delta y}(s,y+\Delta y)) & -f(u_{\Delta y}(s,y))\Big] \notag \\
&= \sigma(u_{\Delta y}(s,y))\,dW(s) + \int_{|z|>0}
\eta(u_{\Delta y}(s,y);z)\,\widetilde{N}(dz,ds),
\end{align*}
and applying It\^{o}-L\'{e}vy product rule to $\beta(u_{\Delta y}(s,y))\phi_{\delta,\delta_0}(t,x,s,y)$, integrating with respect to $y,t,x,k$ and taking the expectation yield
\begin{align}
0\le  & \E \Big[ \int_{\Pi_T} \int_{\R} \int_\R \beta \big(u_{\Delta y}(0)-k\big)
\phi_{\delta,\delta_0}(t,x,0,y) J_l(u(t,x)-k)\,dk\,dy \,dx\,ds\Big] \notag 
\\ & + 
\E \Big[\int_{\Pi_T}\int_{\Pi_T}  \int_{\R}  \beta(u_{\Delta y}(s,y)-k)\partial_s \phi_{\delta,\delta_0}(t,x,s,y)
J_l(u(t,x)-k)\,dk \,dx\,dt\,dy\,ds \Big]\notag 
\\  & +  
\E \Big[ \int_{\Pi_T}\int_{\Pi_T} \int_{\R}  \sigma(u_{\Delta y}(s,y)) \beta^{\prime}(u_{\Delta y}(s,y)-k)\phi_{\delta,\delta_0}(t,x,s,y)
 \,J_l(u(t,x)-k) \,dk \,dt \,dx \,dy\,dW(s) \Big] \notag
 \\& + 
 \frac{1}{2}  \E \Big[\int_{\Pi_T}\int_{\Pi_T} \int_{\R}  \sigma^2(u_{\Delta y}(s,y)) \beta^{\prime\prime}(u_{\Delta y}(s,y)-k)
\phi_{\delta,\delta_0}(t,x,s,y)\,J_l(u(t,x)-k) \,dk\,dt \,dx \,dy\,ds \Big] \notag
\\& -  
\E \Big[\int_{\Pi_T} \int_{\Pi_T}\int_{\R}  \beta^{\prime}(u_{\Delta y}(s,y)-k) \frac{1}{\Delta y}\Big\{ f(u_{\Delta y}(s,y+\Delta y))-f(u_{\Delta y}(s,y))\Big\}
  \phi_{\delta,\delta_0}(t,x,s,y) \notag
  \\ &\hspace{9 cm} \times J_l(u(t,x)-k)\,dk\,ds\,dy\,dx\,dt\Big] \notag \\
  & + \E \Big[\int_{\Pi_T} \int_{\Pi_T}\int_{|z|>0} \int_{\R} \int_{0}^1 \eta(u_{\Delta y}(s,y);z) \beta^{\prime}\big(u_{\Delta y}(s,y)-k + \lambda 
  \eta(u_{\Delta y}(s,y);z)\big)\phi_{\delta,\delta_0}(t,x,s,y) \notag \\
  & \hspace{5cm}  \times \,J_l(u(t,x)-k)\,d\lambda \,dk \,\widetilde{N}(dz,ds)\,dt \,dx \,dy \Big] \notag \\
   &+\E \Big[\int_{\Pi_T} \int_{\Pi_T}\int_{|z|>0} \int_{\R} \int_{0}^1 (1-\lambda)\eta^2(u_{\Delta y}(s,y);z) \beta^{\prime\prime}\big(u_{\Delta y}(s,y)-k + \lambda 
  \eta(u_{\Delta y}(s,y);z)\big)\phi_{\delta,\delta_0}(t,x,s,y) \notag \\
  & \hspace{5cm}  \times \,J_l(u(t,x)-k)\,d\lambda \,dk \,m(dz)\,ds\,dt \,dx \,dy \Big] \notag \\
& := J_1 + J_2 + J_3 + J_4 + J_5 + J_6 + J_7. \label{stoc_entropy_2}
\end{align}
Our goal is to estimate the expected value of the $L^1$ difference between $u_{\Delta y}$ and $u$ in terms of the small parameters $\xi,\delta_0,\delta, l, \Delta y$, which are sufficiently small but fixed. 

Let us first consider the terms due to initial data. Note that, since $\supp(\rho_{\delta_0}) \subset [-\delta_0,0]$, $J_1=0$. Now, in view of the uniform moment estimate 
\eqref{uni:moment}, we infer that
\begin{align}
I_1 & = \E \Big[\int_{\Pi_T}\int_{\R^2} \beta \big(u_0(x)-k\big)
\phi_{\delta,\delta_0}(0,x,s,y) J_l(u_{\Delta y}(s,y)-k)\,dk \,dx\,dy\,ds\Big]
\\&= 
\E\Big[\int_{\R^3} \beta(u_0(x)-k) \varrho_{\delta}(x-y)\psi(0,y) J_l(u_{\Delta y}(0,y)-k)\,dk\,dy\,dx\Big] 
\notag \\ &\quad + \E\Big[\int_{\Pi_T}\int_{\R^2} \beta \big(u_0(x)-u_{\Delta y}(s,y)+k\big) \varrho_{\delta}(x-y)\Big(\psi(s,y)-\psi(0,y)\Big)\rho_{\delta_0}(-s) J_l(k)\,dk\,dy\,dx\,ds\Big] \notag \\
&\qquad  + \E\Big[\int_{\Pi_T}\int_{\R^2}  \Big\{\beta \big(u_0(x)-u_{\Delta y}(s,y)+k\big)- \beta \big(u_0(x)-u_{\Delta y}(0,y)+k\big)\Big\} \varrho_{\delta}(x-y)\psi(0,y) \notag \\
& \hspace{10cm} \times \rho_{\delta_0}(-s) J_l(k)\,dk\,dy\,dx\,ds\Big]. \notag 
\end{align}
Since $\supp(\psi(s,\cdot))\subset K$, where $K$ is a compact subset of $\R$, we have 
\begin{align}
 I_1 & \le  \E\Big[\int_{\R^3} \beta(u_0(x)-k) \varrho_{\delta}(x-y)\psi(0,y) J_l(u_{\Delta y}(0,y)-k)\,dk\,dy\,dx\Big] \notag \\
& \quad + C\, \E\Big[\int_{\Pi_T}\int_{\R^2} \mathds{1}_{K}(y)\big|u_0(x)-u_{\Delta y}(s,y)+k\big| \varrho_{\delta}(x-y)s\rho_{\delta_0}(-s) J_l(k)\,dk\,dy\,dx\,ds\Big] \notag \\
& \qquad +  \E\Big[\int_{\Pi_T}\int_{\R}  \big|u_{\Delta y}(s,y) - u_{\Delta y}(0,y)\big| \varrho_{\delta}(x-y)\psi(0,y) \rho_{\delta_0}(-s)\,dy\,dx\,ds\Big] \notag \\
& \le  \E\Big[\int_{\R^3} \beta(u_0(x)-k) \varrho_{\delta}(x-y)\psi(0,y) J_l(u_{\Delta y}(0,y)-k)\,dk\,dy\,dx\Big] \notag \\
& \qquad + C\, \E\Big[\int_{\Pi_T}\int_{\R}\mathds{1}_{K}(y)\big|u_0(x)-u_{\Delta y}(s,y)\big| \varrho_{\delta}(x-y)s\rho_{\delta_0}(-s)\,dy\,dx\,ds\Big] \notag \\
& \qquad \quad  + C\, l\, \E\Big[\int_{\Pi_T}\int_{\R} \mathds{1}_{K}(y) \varrho_{\delta}(x-y)s\rho_{\delta_0}(-s)\,dy\,dx\,ds\Big] \notag \\
& \qquad \qquad  +  C\, ||\psi(0,\cdot)||_{L^\infty(\R)}\, \E\Big[\int_{0}^T \int_{\supp(\psi(0,\cdot))} \big|u_{\Delta y}(s,y) - u_{\Delta y}(0,y)\big| \rho_{\delta_0}(-s)\,dy\,dx\,ds\Big] \notag \\
& \le  \E\Big[\int_{\R^3} \beta(u_0(x)-k) \varrho_{\delta}(x-y)\psi(0,y) J_l(u_{\Delta y}(0,y)-k)\,dk\,dy\,dx\Big]  + C\delta_0(1+l)\notag \\
& \hspace{3cm}+ C\,  \E\Big[\int_{0}^T \int_{\supp(\psi(0,\cdot))} \big|u_{\Delta y}(s,y) - u_{\Delta y}(0,y)\big| \rho_{\delta_0}(-s)\,dy\,ds\Big] \notag \\
& := I_{1,1} +  C\,\delta_0\,(1+l) + C\,\sqrt{\delta_0},
\label{esti:i1-half}
\end{align}
thanks to Lemma~\ref{continuity in time}. Moreover, since
\begin{align}
 I_{1,1}& \le \Big|  \E\Big[\int_{\R^3} \Big\{\beta \big(u_0(x)-u_{\Delta y}(0,y) + k\big)- \beta \big(u_0(x)-u_{\Delta y}(0,y)\big)\Big\} \varrho_{\delta}(x-y)
 \psi(0,y) J_l(k)\,dk\,dy\,dx\Big]\Big|\notag \\
 & \hspace{2cm} + \Big[\int_{\R^2} \beta\big(u_0(x)-u_{\Delta y}(0,y)\big) \varrho_{\delta}(x-y)\psi(0,y)\,dy\,dx\Big] \notag \\
& \le  \E\Big[\int_{\R^2} \beta\big(u_0(x)-u_{\Delta y}(0,y)\big) \varrho_{\delta}(x-y)\psi(0,y)\,dy\,dx\Big]  + C\,l,\label{esti:i11}
\end{align}
combining \eqref{esti:i11} and \eqref{esti:i1-half}, keeping in mind that $J_1=0$, we obtain 
\begin{align}
I_1 +J_1 \le  \E\Big[\int_{\R^2} \beta\big(u_0(x)-u_{\Delta y}(0,y)\big) \varrho_{\delta}(x-y)\psi(0,y)\,dy\,dx\Big] +  C\big(\delta_0 + l\big) + C \sqrt{\delta_0}. \label{esti:i1+j1-final}
\end{align}
Collecting the results above, we have proved the following Lemma.
\begin{lem}\label{I1+J1}
\begin{align*}
\limsup_{l \rightarrow 0}\,\limsup_{\delta_0 \rightarrow 0} \,(I_1 +J_1) \le  E\Big[\int_{\R^2} \beta\big(u_0(x)-u_{\Delta y}(0,y)\big) \varrho_{\delta}(x-y)\psi(0,y)\,dy\,dx\Big].
\end{align*}
\end{lem}
%%%%%%%%%%%%%%%%%%%%%%%%%%%%%%%%%%%%%%%%%%%%%%%%%%%%%%%%%%%%%%%%%%%%%%%%%%%%%%%%%%
%%%%%%%%%%%%%%%%%%%%%%%%%%%%%%%%%%%%%%%%%%%%%%%%%%%%%%%%%%%%%%%%%%%%%%%%%%%%%%%%%%%%%
\medskip

Let us tern our attention to the term $(I_2 +J_2)$. 
First note that, since $\beta$ and $J_l$ are even functions, we have 
\begin{align}
I_2 + J_{2}&= \E\Big[\int_{\Pi_T}\int_{\R}\int_{\Pi_T}  \beta(u_{\Delta y}(s,y)-k) \partial_s \psi(s,y) \rho_{\delta_0}(t-s)
\varrho_{\delta}(x-y)J_l(u(t,x)-k) \,ds\,dy\,dk \,dx\,dt \Big].\notag \\
&=  \E\Big[\int_{\Pi_T}\int_{\R^2}  \beta(u_{\Delta y}(s,y)-k) \partial_s \psi(s,y)
\varrho_{\delta}(x-y)J_l(u(s,x)-k) \,ds\,dy\,dk \,dx \Big]\notag \\
& -  \E\Big[\int_{\Pi_T}\int_{\R^2} \beta\big(u_{\Delta y}(s,y)-u(s,x)+k\big) \partial_s \psi(s,y) \big(1-\int_{0}^T \rho_{\delta_0}(t-s)\,dt\big)
\varrho_{\delta}(x-y) J_l(k) \,ds\,dy\,dk \,dx\Big]\notag \\
&\quad  +  \E\Big[\int_{\Pi_T}\int_{\R}\int_{\Pi_T} \Big\{\beta\big(u_{\Delta y}(s,y)-u(t,x)+k\big) 
- \beta\big(u_{\Delta y}(s,y)-u(s,x)+k\big)\Big\} \notag \\
& \hspace{7cm} \times \partial_s \psi(s,y) \rho_{\delta_0}(t-s) \varrho_{\delta}(x-y)J_l(k) \,ds\,dy\,dk \,dx\,dt \Big]\notag \\
& \le  \E\Big[\int_{\Pi_T}\int_{\R^2} \beta(u_{\Delta y}(s,y)-k) \partial_s \psi(s,y)
\varrho_{\delta}(x-y)J_l(u(s,x)-k) \,ds\,dy\,dk \,dx \Big]\notag \\ 
& \quad +  \E\Big[\int_{\Pi_T}\int_{\Pi_T}\big|u(t,x)-u(s,x)\big| |\partial_s \psi(s,y)| \rho_{\delta_0}(t-s)
\varrho_{\delta}(x-y)\,ds\,dy \,dx\,dt \Big]\notag \\
& \qquad +  \E\Big[\int_{s=0}^{\delta_0}\int_{\R^3} \big| u_{\Delta y}(s,y)-u(s,x)+k \big| |\partial_s \psi(s,y)|
\varrho_{\delta}(x-y) J_l(k)\,dk\,dx\,dy\,ds\Big] \notag \\
& \le  \E\Big[\int_{\Pi_T}\int_{\R^2}\beta(u_{\Delta y}(s,y)-k) \partial_s \psi(s,y)
\varrho_{\delta}(x-y)J_l(u(s,x)-k) \,ds\,dy\,dk \,dx \Big] + C\,\delta_0\big( 1+ l\big)\notag \\ 
& \quad +  \E\Big[\int_{\Pi_T}\int_{\Pi_T}\big|u(t,x)-u(s,x)\big| |\partial_s \psi(s,y)| \rho_{\delta_0}(t-s)
\varrho_{\delta}(x-y)\,ds\,dy \,dx\,dt \Big]\notag \\
& := J_{2,1}^{1,I_2} +  C\,\delta_0\big( 1+ l\big) + J_{2,1}^{2,I_2}. \label{i2+j21-half} 
\end{align}
Let us consider the term $J_{2,1}^{1,I_2}$. We have 
\begin{align}
J_{2,1}^{1,I_2} & \le   \E\Big[\int_{\Pi_T}\int_{\R} \beta \big(u_{\Delta y}(s,y)-u(s,x)\big) \partial_s \psi(s,y)
\varrho_{\delta}(x-y)\,dx\,dy\,ds \Big]\notag \\ 
& \quad + \Bigg| \E\Big[\int_{\Pi_T}\int_{\R^2} \Big\{  \beta \big(u_{\Delta y}(s,y)-u(s,x) +k \big)- \beta \big(u_{\Delta y}(s,y)-u(s,x)\big)\Big\}
\partial_s \psi(s,y)\notag \\
& \hspace{8cm} \times\varrho_{\delta}(x-y) J_l(k) \,dk\,dx\,dy \,ds \Big]\Bigg|\notag \\ 
& \le  \E\Big[\int_{\Pi_T}\int_{\R} \beta \big(u_{\Delta y}(s,y)-u(s,x)\big) \partial_s \psi(s,y)
\varrho_{\delta}(x-y)\,dx\,dy\,ds \Big] +  C\, l. \label{esti:j21-1i2}
\end{align}
Let us move on to estimate the term $J_{2,1}^{2,I_2}$: 
\begin{align}
 J_{2,1}^{2,I_2} & \le \|\partial_s \psi\|_{L^\infty(\Pi_T)} \E \Big[ \int_{\Pi_T} \int_0^T \int_{\supp(\partial_s \psi(s,\cdot))} 
 \big|u(t,x)-u(s,x)| \rho_{\delta_0}(t-s) \varrho_{\delta}(x-y)\,ds\,dy \,dx\,dt \Big]\notag \\
 & \le C\, \E \Big[ \int_{0}^T \int_0^T \int_{\supp(\partial_s \psi(s,\cdot))} 
 \big|u(t,x)-u(s,x)|dx \rho_{\delta_0}(t-s)\,ds\,dt \Big]  \le C \sqrt{\delta_0}. \label{esti:j21-2i2}
\end{align}
Thus, combining \eqref{esti:j21-1i2} and \eqref{esti:j21-2i2} in \eqref{i2+j21-half}, we obtain
\begin{align}
    I_2 + J_2 & \le  E\Big[\int_{\Pi_T}\int_{\R} \beta \big(u_{\Delta y}(s,y)-u(s,x)\big) \partial_s \psi(s,y)
\varrho_{\delta}(x-y)\,dx\,dy\,ds \Big] +  C\,\delta_0\big( 1+ l\big) + C \sqrt{\delta_0}.\label{esti:i2+j21}
\end{align}
In fact, we have proved the following Lemma.
\begin{lem}\label{I2+J2}
\begin{align*}
\limsup_{l \rightarrow 0}\,\limsup_{\delta_0 \rightarrow 0} \, (I_2 + J_2) \leq \E\Big[\int_{\Pi_T}\int_{\R} \beta \big(u_{\Delta y}(s,y)-u(s,x)\big) \partial_s \psi(s,y)
\varrho_{\delta}(x-y)\,dx\,dy\,ds \Big].
\end{align*}
\end{lem}

%%%%%%%%%%%%%%%%%%%%%%%%%%%%%%%%%%%%%%%%%%%%%%%%%%%%%%%%%%%%%%%%%%%%%%%%%%%%%%
%%%%%%%%%%%%%%%%%%%%%%%%%%%%%%%%%%%%%%%%%%%%%%%%%%%%%%%%%%%%%%%%%%%%%%%%%%%%%%%%%%%

Next, we move on to estimate the terms coming from the associated flux function. To do this, we first consider the term $J_5$. 
Since $f^\prime \le 0$ and $\beta$ is a convex function, one has
\begin{align}
 0\ge   \int_{u_j}^{u_{j+1}} &\beta^{\prime\prime}(r-k)\big( f(u_{j+1})-f(r)\big)\,dr  \notag \\
  & = - \beta^{\prime}(u_j-k) \big(f(u_{j+1})- f(u_{j})\big) + \int_{u_j}^{u_{j+1}} \beta^{\prime}(r-k)f^\prime(r)\,dr.\label{inq:properties-flux}
\end{align}
Therefore, in view of \eqref{inq:properties-flux}, we have
\begin{align}
 J_5 & = -  
\E \Big[\int_{\Pi_T} \int_{\Pi_T}\int_{\R}  \beta^{\prime}(u_{\Delta y}(s,y)-k) \frac{1}{\Delta y}\Big\{ f(u_{\Delta y}(s,y+\Delta y))-f(u_{\Delta y}(s,y))\Big\}
  \phi_{\delta,\delta_0}(t,x,s,y) \notag
  \\ &\hspace{10cm} \times
  J_l(u(t,x)-k)\,dk\,ds\,dy\,dx\,dt\Big] \notag
 \\& \le 
 - \E \Big[\int_{\Pi_T}\int_{\Pi_T} \int_{\R}  \frac1{\Delta y}\Big(\int_{u_{\Delta y}(s,y)}^{u_{\Delta y}(s,y+\Delta y)} \beta^{\prime}(r-k)f^\prime(r)\,dr\Big)
  \phi_{\delta,\delta_0}(t,x,s,y) J_l(u(t,x)-k)\,dk\,ds\,dy\,dx\,dt\Big] \notag 
  \\&= 
  \E \Big[\int_{\Pi_T} \int_{\Pi_T} \int_{\R}\frac1{\Delta y} \Big\{ f^\beta \big(u_{\Delta y}(s,y),k\big)-f^\beta \big(u_{\Delta y}(s,y+\Delta y),k\big)\Big\}  \notag
  \\ &\hspace{8cm} \times
  \phi_{\delta,\delta_0}(t,x,s,y) J_l(u(t,x)-k)\,dk\,ds\,dy\,dx\,dt\Big] \notag 
  \\& =  
  \E \Big[\int_{\Pi_T}\int_{\Pi_T} \int_{\R}  f^\beta \big(u_{\Delta y}(s,y),k\big)\frac1{\Delta y} \Big\{\phi_{\delta,\delta_0}(t,x,s,y)-
  \phi_{\delta,\delta_0}(t,x,s,y-\Delta y)\Big\}  \notag
  \\ &\hspace{10cm} \times
  J_l(u(t,x)-k)\,dk\,ds\,dy\,dx\,dt\Big]\notag 
  \\&  =  
   \E \Big[\int_{\Pi_T}\int_{\Pi_T} \int_{\R} f^\beta \big(u_{\Delta y}(s,y),k\big) \partial_y \phi_{\delta,\delta_0}(t,x,s,y)
  J_l(u(t,x)-k) \,dk\,ds\,dy\,dx\,dt\Big]\notag \\
  & \quad +  \E \Bigg[\int_{\Pi_T} \int_{\Pi_T}\int_{\R}  f^\beta \big(u_{\Delta y}(s,y),k\big)
   \Bigg\{ \displaystyle \frac{\phi_{\delta,\delta_0}(t,x,s,y)-\phi_{\delta,\delta_0}(t,x,s,y-\Delta y)}{\Delta y} 
   -  \partial_y \phi_{\delta,\delta_0}(t,x,s,y) \Bigg\} \notag \\
   & \hspace{10cm} \times J_l(u(t,x)-k)\,dy \,ds\,dk\,dx\,dt\Bigg]\notag 
  \\ &
  := J_{5,1} + J_{5,2}.\label{eq:j5}
\end{align}
Note that
\begin{align*}
\displaystyle \frac{\phi_{\delta,\delta_0}(t,x,s,y)-\phi_{\delta,\delta_0}(t,x,s,y-\Delta y)}{\Delta y} 
   -  \partial_y \phi_{\delta,\delta_0}(t,x,s,y) = \frac{-1}{\Delta y}\int_{-\Delta y}^0\int_\sigma^0 \partial_{yy}\phi_{\delta,\delta_0}(t,x,s,\tau+y)\,d\tau\,d\sigma.
\end{align*}
Thus, 
\begin{align*}
&J_{5,2}:=\E \Big[\int_{\Pi_T} \int_{\Pi_T} \int_{\R}  f^\beta \big(u_{\Delta y}(s,y),k\big)
   \Bigg\{ \displaystyle \frac{\phi_{\delta,\delta_0}(t,x,s,y)-\phi_{\delta,\delta_0}(t,x,s,y-\Delta y)}{\Delta y} 
   -  \partial_y \phi_{\delta,\delta_0}(t,x,s,y) \Bigg\} \notag \\
   & \hspace{9cm} \times J_l(u(t,x)-k)\,dy \,ds\,dk\,dx\,dt\Big]
   \\ = 
&\E \Big[\int_{\Pi_T}\int_{\Pi_T} \sum_{j}  f^\beta \big(u_{j}(s),k\big)\int_{y_{j-\frac12}}^{y_{j+\frac12}}
   \Bigg\{ \frac{-1}{\Delta y}\int_{-\Delta y}^0\int_\sigma^0 \partial_{yy}\phi_{\delta,\delta_0}(t,x,s,\tau+y)\,d\tau\,d\sigma \Bigg\}\,dy \notag \\
   & \hspace{9cm} \times J_l(u(t,x)-k)\,ds\,dk\,dx\,dt\Big]   
   \\ = 
&\E \Big[\int_{\Pi_T} \int_{\Pi_T} \sum_{j}  f^\beta \big(u_{j}(s),k\big)
   \Bigg\{ \frac{-1}{\Delta y}\int_{-\Delta y}^0\int_\sigma^0 \bigg\{\partial_{y}\phi_{\delta,\delta_0}(t,x,s,\tau+y_{j+\frac12}) - \partial_{y}\phi_{\delta,\delta_0}(t,x,s,\tau+y_{j-\frac12})\bigg\}\,d\tau\,d\sigma \Bigg\} \notag \\
   & \hspace{9cm} \times J_l(u(t,x)-k)\,ds\,dk\,dx\,dt\Big]   
   \\ = 
&\E \Big[\int_{\Pi_T}\int_{\Pi_T} \sum_{j}  \Big\{f^\beta \big(u_{j+1}(s),k\big) - f^\beta \big(u_{j}(s),k\big)\Big\}
   \Bigg\{ \frac{1}{\Delta y}\int_{-\Delta y}^0\int_\sigma^0 \partial_{y}\phi_{\delta,\delta_0}(t,x,s,\tau+y_{j+\frac12}) \,d\tau\,d\sigma \Bigg\} \notag \\
   & \hspace{9cm} \times J_l(u(t,x)-k)\,ds\,dk\,dx\,dt\Big]   
   \\ \leq
&C\, \E \Big[\int_{\Pi_T} \int_{\Pi_T} \sum_{j}  \Big|f^\beta \big(u_{j+1}(s),k\big) - f^\beta \big(u_{j}(s),k\big)\Big|
   \frac{1}{\Delta y}\int_{-\Delta y}^0\int_\sigma^0 \Big|\partial_{y}\phi_{\delta,\delta_0}(t,x,s,\tau+y_{j+\frac12})\Big| \,d\tau\,d\sigma \notag \\
   & \hspace{9cm} \times J_l(u(t,x)-k)\,ds\,dk\,dx\,dt\Big]   
   \\ \leq
&C\, \E \Big[\int_{\R} \int_0^T \int_0^T\sum_{j}  \big|u_{j+1}(s) - u_{j}(s)\big|
   \frac{1}{\Delta y}\int_{-\Delta y}^0\int_\sigma^0 \Big|\partial_{y}\phi_{\delta,\delta_0}(t,x,s,\tau+y_{j+\frac12})\Big| \,d\tau\,d\sigma  \,ds\,dx\,dt\Big]   
   \\ \leq
& C\, \E \Big[\int_{\R}\int_0^T  \sum_{j}  \big|u_{j+1}(s) - u_{j}(s)\big|
   \frac{1}{\Delta y}\int_{-\Delta y}^0\int_\sigma^0 \frac{\psi(s,\tau+y_{j+\frac12})}{\delta^2}\mathds{1}_{\{|x-y_{j+\frac12}|<\delta\}} \,d\tau\,d\sigma \,ds\,dx\,\Big]  
   \\ \leq
& C\, \E \Big[\int_0^T  \sum_{j}  \big|u_{j+1}(s) - u_{j}(s)\big|
   \frac{1}{\Delta y}\int_{-\Delta y}^0\int_\sigma^0 \frac{\psi(s, \tau+ y_{j+\frac12})}{\delta} \,d\tau\,d\sigma \,ds\Big]   
   \\ \leq
& C\, \frac{\Delta y}{\delta}\, \E \Big[\int_0^T  \sum_{j}  \big|u_{j+1}(s) - u_{j}(s)\big| \,ds\Big]   \leq C \frac{\Delta y}{\delta}
\end{align*}
where, in particular, \eqref{inq:discrete-bv-1} have been used.

%%%%%%%%%%%%%%%%%%%%%%%%%%
Let us first consider the term $J_{5,1}$. Note that for any $a,b \in \R$, $f^\beta(a,b)$ is Lipschitz continuous in both the variables.
Therefore, we have 
\begin{align}
J_{5,1}& = \E\Big[\int_{\Pi_T}\int_{\R^2} f^\beta \big(u_{\Delta y}(s,y),k\big)\partial_y \varrho_{\delta}(x-y)\psi(s,y)
J_l(u(s,x)-k) \,ds\,dy \,dk\,dx\Big]\notag \\
&\, - \E\Big[\int_{\Pi_T}\int_{\R^2} f^\beta \big(u_{\Delta y}(s,y),u(s,x)-k\big) \partial_y\varrho_{\delta}(x-y) \psi(s,y) \Big(1-\int_{0}^T \rho_{\delta_0}(t-s)\,dt\Big)
 J_l(k) \,ds\,dy\,dx\,dk\Big] \notag \\
 &\quad + \E\Big[\int_{\Pi_T}\int_{\R} \int_{\Pi_T} \Big\{ f^\beta \big(u_{\Delta y}(s,y),u(t,x)-k\big)-  f^\beta \big(u_{\Delta y}(s,y),u(s,x)-k\big)\Big\} \partial_y\varrho_{\delta}(x-y) \notag \\
 & \hspace{8cm} \times \rho_{\delta_0}(t-s)\psi(s,y) J_l(k) \,ds\,dy \,dk\,dx\,dt\Big]\notag \\
& \le  \E\Big[\int_{\Pi_T}\int_{\R^2} f^\beta \big(u_{\Delta y}(s,y),k\big)\partial_y \varrho_{\delta}(x-y)\psi(s,y) J_l(u(s,x)-k) \,ds\,dy \,dk\,dx\Big]\notag \\
&\quad  +\E\Big[\int_{\Pi_T}\int_{\Pi_T} \big|u(s,x)-u(t,x)\big| \psi(s,y)\rho_{\delta_0}(t-s) |\partial_y \varrho_{\delta}(x-y)|\,dx\,dt\,dy\,ds\Big]\notag \\
&\qquad \qquad   + \E\Big[\int_{s=0}^{\delta_0}\int_{\R^3}\big|u_{\Delta y}(s,y)-u(s,x)+k\big||\partial_y \varrho_{\delta}(x-y)|\psi(s,y)J_l(k)\,dk,dx\,dy\,ds\Big] \notag \\
& \le  \E\Big[\int_{\Pi_T}\int_{\R^2} f^\beta \big(u_{\Delta y}(s,y),k\big)\partial_y \varrho_{\delta}(x-y)\psi(s,y) J_l(u(s,x)-k) \,ds\,dy \,dk\,dx\Big]
+ C \frac{\delta_0}{\delta} \notag \\
& \hspace{2cm}  + \E\Big[\int_{\Pi_T}\int_{\Pi_T} \big|u(s,x)-u(t,x)\big| \psi(s,y)\rho_{\delta_0}(t-s) |\partial_y \varrho_{\delta}(x-y)|\,dx\,dt\,dy\,ds\Big] \notag \\
& := J_{5,1}^1 + C \frac{\delta_0}{\delta} + C\frac{\sqrt{\delta_0}}{\delta}, \label{esti:j51-half}
\end{align}
thanks to the Lemma~\ref{continuity in time}. Then, we estimate $J_{5,1}^1$ as follows:
\begin{align}
J_{5,1}^1 & \le  \E\Big[\int_{\Pi_T}\int_{\R} f^\beta \big(u_{\Delta y}(s,y),u(s,x)\big)\partial_y \varrho_{\delta}(x-y)\psi(s,y)\,dx\,dy\,ds\Big] \notag \\
& \qquad + \Big|\E\Big[\int_{\Pi_T}\int_{\R^2}\Big\{  f^\beta \big(u_{\Delta y}(s,y),u(s,x)-k\big)-  f^\beta \big(u_{\Delta y}(s,y),u(s,x)\big)\Big\} \notag \\
& \hspace{3cm} \times \partial_y \varrho_{\delta}(x-y) \psi(s,y)J_l(k) \,ds\,dy \,dk\,dx\Big]\Big|\notag \\
 & \le  \E\Big[\int_{\Pi_T}\int_{\R} f^\beta \big(u_{\Delta y}(s,y),u(s,x)\big)\partial_y \varrho_{\delta}(x-y)\psi(s,y)\,dy \,dx\,ds\Big] + C \frac{l}{\delta}.\label{esti:j51-1}
\end{align}

Therefore, we have 
\begin{align}
 J_{5,1} & \le  \E\Big[\int_{\Pi_T}\int_{\R} f^\beta \big(u_{\Delta y}(s,y),u(s,x)\big)\partial_y \varrho_{\delta}(x-y)\psi(s,y)\,dy \,dx\,ds\Big] 
 +  C\frac{\sqrt{\delta_0}}{\delta} + C \frac{l}{\delta}.\label{esti:j51}
\end{align}

Using same calculations as we have done for the term $J_{5,1}$, we can show that 
\begin{align}
I_5 &= \E \Big[\int_{\Pi_T}\int_{\Pi_T} \int_{\R} 
 f^\beta(u(t,x),k)\partial_x \phi_{\delta,\delta_0}(t,x,s,y) J_l(u_{\Delta y}(s,y)-k)\,dk\,dx\,dt\,dy\,ds\Big] \notag
 \\ &= 
 \E \Big[\int_{\Pi_T}\int_{\Pi_T} \int_{\R} 
 f^\beta(u(s,x),k)\partial_x \phi_{\delta,\delta_0}(t,x,s,y) J_l(u_{\Delta y}(s,y)-k)\,dk\,dx\,dt\,dy\,ds\Big] \notag
 \\&+ 
 \E \Big[\int_{\Pi_T}\int_{\Pi_T} \int_{\R} 
 [f^\beta(u(t,x),k)-f^\beta(u(s,x),k)]\partial_x \phi_{\delta,\delta_0}(t,x,s,y) J_l(u_{\Delta y}(s,y)-k)\,dk\,dx\,dt\,dy\,ds\Big] \notag
 \\ &\leq 
 \E \Big[\int_{\Pi_T}\int_{\Pi_T} \int_{\R} 
 f^\beta(u(s,x),u_{\Delta y}(s,y)-k)\partial_x \phi_{\delta,\delta_0}(t,x,s,y) J_l(k)\,dk\,dx\,dt\,dy\,ds\Big] \notag
 \\& \quad +
\frac{C}{\delta} \E \Big[\int_{K}\int_{(0,T)^2} 
 \Big[|u(t,x)-u(s,x)|+l\Big]\rho_{\delta_0}(t-s) \,dx\,dt\,ds\Big] \notag
 \\ &\leq 
 \E \Big[\int_{\Pi_T}\int_{\Pi_T} \int_{\R} 
 f^\beta(u(s,x),u_{\Delta y}(s,y))\partial_x \phi_{\delta,\delta_0}(t,x,s,y) J_l(k)\,dk\,dx\,dt\,dy\,ds\Big] 
 + C\frac{l}{\delta} + \frac{C\sqrt{\delta_0}}{\delta}\notag
\\&\le  
\E\Big[\int_{\Pi_T}\int_{\R} f^\beta \big(u(s,x),u_{\Delta y}(s,y)\big)\partial_x \varrho_{\delta}(x-y)\psi(s,y)\,dy \,dx\,ds\Big]
+ C\,\frac{\sqrt{\delta_0}}{\delta} + C \frac{l}{\delta}.\label{esti:i5}
\end{align}
Combining \eqref{esti:j51} in \eqref{eq:j5} along with the estimate \eqref{esti:i5}, we have 
\begin{align}
I_5 +J_5 & \le \E\Big[\int_{\Pi_T}\int_{\R} \Big\{f^\beta \big(u_{\Delta y}(s,y),u(s,x)\big)- f^\beta \big( u(s,x),u_{\Delta y}(s,y)\big) \Big\} \partial_y \varrho_{\delta}(x-y) \psi(s,y)\,dx\,dy\,ds\Big] \notag \\
& \hspace{2cm} + \E\Big[\int_{\Pi_T}\int_{\R} f^\beta \big(u_{\Delta y}(s,y),u(s,x)\big)\varrho_{\delta}(x-y)\partial_y \psi(s,y)\,dx\,dy\,ds\Big]  \notag \\
& \hspace{4cm}+  C\frac{\Delta y + \sqrt{\delta_0} + l}{\delta}.
\end{align}
Thus, we have proved the following Lemma.
\begin{lem}
\label{I5+J5}
\begin{align}
&\limsup_{l \rightarrow 0}\,\limsup_{\delta_0 \rightarrow 0} \, (I_5 +J_5) 
\\ & \le 
\E\Big[\int_{\Pi_T}\int_{\R} \Big\{f^\beta \big(u_{\Delta y}(s,y),u(s,x)\big)- f^\beta \big( u(s,x),u_{\Delta y}(s,y)\big) \Big\} \partial_y \varrho_{\delta}(x-y) \psi(s,y)\,dx\,dy\,ds\Big] \notag
\\& \hspace{2cm} + \E\Big[\int_{\Pi_T}\int_{\R} f^\beta \big(u(s,x),u_{\Delta y}(s,y)\big)\varrho_{\delta}(x-y)\partial_y \psi(s,y)\,dx\,dy\,ds\Big]
+ C\frac{\Delta y}{\delta}.  \notag
\end{align}
\end{lem}
Let us consider the additional term $J_4$. In view of the assumption \ref{A3} and the bound $\|\beta_{\xi}^{\prime\prime\prime}\|_{L^\infty(\R)}\le \frac{C}{\xi^2}$, we have
\begin{align}
 J_4  &= \frac{1}{2}  \E \Big[\int_{\Pi_T} \int_{\Pi_T}\int_{\R}  \sigma^2(u_{\Delta y}(s,y)) \beta^{\prime\prime}(u_{\Delta y}(s,y)-k)
\phi_{\delta,\delta_0}(t,x,s,y)\,J_l(u(t,x)-k) \,dk\,dt \,dx \,dy\,ds \Big] \notag
\\ & = 
\frac{1}{2}  \E \Big[\int_{\Pi_T} \int_{\R}  \int_{\Pi_T} \sigma^2(u_{\Delta y}(s,y)) \Big\{ \beta^{\prime\prime}\big(
 u_{\Delta y}(s,y)-u(t,x) + k\big)- \beta^{\prime\prime}\big(u_{\Delta y}(s,y)-u(s,x) + k\big) \Big\} \notag \\
 & \hspace{5cm} \times \psi(s,y) \varrho_{\delta}(x-y) \rho_{\delta_0}(t-s) J_l(k) \,dy \,ds \,dk\,dx\,dt \Big] \notag\\ 
 & \quad -\frac{1}{2}  \E \Big[\int_{\Pi_T} \int_{\R^2}\sigma^2(u_{\Delta y}(s,y))  \beta^{\prime\prime}\big(u_{\Delta y}(s,y)-u(s,x) + k\big) 
 \psi(s,y) \Big(1-\int_0^T \rho_{\delta_0}(t-s)\,dt\Big) \notag \\
 & \hspace{8cm} \times \varrho_{\delta}(x-y) J_l(k) \,dy \,ds \,dk\,dx \Big] \notag
 \\  & \qquad + 
 \frac{1}{2}  \E \Big[\int_{\Pi_T} \int_{\R^2} \sigma^2(u_{\Delta y}(s,y))  \beta^{\prime\prime}\big(u_{\Delta y}(s,y)-u(s,x) + k\big) 
 \psi(s,y)\varrho_{\delta}(x-y) J_l(k) \,dy \,ds \,dk\,dx \Big] \notag
 \\ & \le 
 \frac{1}{2}  \E \Big[\int_{\Pi_T} \int_{\R^2} \sigma^2(u_{\Delta y}(s,y))  \beta^{\prime\prime}\big(u_{\Delta y}(s,y)-u(s,x) + k\big) 
 \psi(s,y)\varrho_{\delta}(x-y) J_l(k) \,dy \,ds \,dk\,dx \Big] \notag
 \\  & \qquad + 
 \frac{C}{\xi^2} \E\Big[\int_{\Pi_T}\int_{\Pi_T} \sigma^2(u_{\Delta y}(s,y))\big|u(t,x)-u(s,x)\big| 
 \psi(s,y) \varrho_{\delta}(x-y) \rho_{\delta_0}(t-s)\,dy \,ds\,dx\,dt \Big] \notag
 \\  &  \hspace{6cm} + 
 \frac{C}{\xi} \E\Big[ \int_{s=0}^{\delta_0} \int_{\R} \sigma^2(u_{\Delta y}(s,y))\psi(s,y)\,dy\,ds\Big] \notag 
 \\  & \le 
 \frac{1}{2}  \E \Big[\int_{\Pi_T} \int_{\R^2} \sigma^2(u_{\Delta y}(s,y))  \beta^{\prime\prime}\big(u_{\Delta y}(s,y)-u(s,x) + k\big) 
  \psi(s,y)\varrho_{\delta}(x-y) J_l(k) \,dy \,ds \,dk\,dx \Big] \notag\\ 
  & \qquad + \frac{C}{\xi^2} \E\Big[\int_{K}\int_0^T \int_{0}^T \big|u(t,x)-u(s,x)\big|  \rho_{\delta_0}(t-s) \,ds\,dt\,dx \Big] + C \frac{\delta_0}{\xi} \notag\\ 
  &:= J_{4,1} + C \frac{\sqrt{\delta_0}}{\xi^2} +  C \frac{\delta_0}{\xi},\label{esti:j41}
  \end{align}
thanks to Lemma~\ref{continuity in time}. Let us focus on the term $J_{4,1}^1$.
 \begin{align}
 J_{4,1} & \le \frac{1}{2} \E \Big[\int_{\Pi_T} \int_{\R} \sigma^2(u_{\Delta y}(s,y))  \beta^{\prime\prime}\big(u_{\Delta y}(s,y)-u(s,x)\big) 
 \psi(s,y)\varrho_{\delta}(x-y) \,dy \,ds\,dx \Big] \notag\\ 
 & \qquad  + \frac{1}{2}\Big| \E \Big[\int_{\Pi_T} \int_{\R^2} \sigma^2(u_{\Delta y}(s,y))\Big\{   \beta^{\prime\prime}\big(u_{\Delta y}(s,y)-u(s,x) + k\big) 
 - \beta^{\prime\prime}\big(u_{\Delta y}(s,y)-u(s,x)\big) \Big\} \notag \\
 & \hspace{4cm} \times  \psi(s,y)\varrho_{\delta}(x-y) J_l(k)\,dk \,dy \,ds\,dx \Big]\Big| \notag\\ 
 & \le \frac{1}{2}  \E \Big[\int_{\Pi_T} \int_{\R} \sigma^2(u_{\Delta y}(s,y))  \beta^{\prime\prime}\big(u_{\Delta y}(s,y)-u(s,x)\big) 
 \psi(s,y)\varrho_{\delta}(x-y) \,dy \,ds\,dx \Big] \notag\\  
 & \qquad  \quad + \frac{C}{\xi^2} \E\Big[ \int_{\Pi_T}\int_{\R^2} \sigma^2(u_{\Delta y}(s,y)) |k|J_l(k) 
 \psi(s,y) \varrho_{\delta}(x-y) \,dk\,dy\,ds\,dx\Big]  \notag \\
 & \le \frac{1}{2}  \E \Big[\int_{\Pi_T} \int_{\R} \sigma^2(u_{\Delta y}(s,y))  \beta^{\prime\prime}\big(u_{\Delta y}(s,y)-u(s,x)\big) 
 \psi(s,y)\varrho_{\delta}(x-y) \,dy \,ds\,dx \Big] \notag \\
 & \quad \qquad + C  \|\psi\|_{L^\infty(\Pi_T)} \frac{l}{\xi^2} \sup_{0\le s\le T} \E\Big[ \|u_{\Delta y}(s,\cdot)\|_2^2\Big]  \notag \\
 &\le  \frac{1}{2}  \E \Big[\int_{\Pi_T} \int_{\R} \sigma^2(u_{\Delta y}(s,y))  \beta^{\prime\prime}\big(u_{\Delta y}(s,y)-u(s,x)\big) 
 \psi(s,y)\varrho_{\delta}(x-y) \,dy \,ds\,dx \Big] + C \frac{l}{\xi^2}.\label{esti:j41-1}
\end{align}
Combining the above terms together leads to
\begin{align}
 J_4  & \le \frac{1}{2}  \E \Big[\int_{\Pi_T} \int_{\R} \sigma^2(u_{\Delta y}(s,y))  \beta^{\prime\prime}\big(u_{\Delta y}(s,y)-u(s,x)\big) 
 \psi(s,y)\varrho_{\delta}(x-y) \,dy \,ds\,dx \Big]   \notag \\
 & \hspace{2cm}  + C \Big( \frac{\delta_0}{\xi}   + \frac{l}{\xi^2} + \frac{\sqrt{\delta_0}}{\xi^2}\Big).\label{esti:j4-final}
\end{align}
Again, we see that
\begin{align}
 I_4 & =\frac{1}{2}  \E \Big[\int_{\Pi_T} \int_{\Pi_T}
\int_{\R}  \sigma^2(u(t,x)) \beta^{\prime \prime}(u(t,x)-k)\phi_{\delta,\delta_0}(t,x,s,y)
 \,J_l(u_{\Delta y}(s,y)-k) \,dt \,dx \,dk \,dy\,ds \Big] \notag
 \\ &= 
 \frac{1}{2}  \E \Big[\int_{\Pi_T} \int_{\Pi_T}
 \sigma^2(u(t,x)) \beta^{\prime \prime}(u(t,x)-u_{\Delta y}(t,y))\phi_{\delta,\delta_0}(t,x,s,y) \,dt \,dx \,dy\,ds \Big] \notag
 \\&+
 \frac{1}{2}  \E \Big[\int_{\Pi_T} \int_{\Pi_T}
\int_{\R}  \sigma^2(u(t,x)) \Big[\beta^{\prime \prime}(u(t,x)-u_{\Delta y}(s,y) + k) - \beta^{\prime \prime}(u(t,x)-u_{\Delta y}(t,y))\Big] \notag
\\& \hspace{8cm}\times\phi_{\delta,\delta_0}(t,x,s,y)
 \,J_l(k) \,dt \,dx \,dk \,dy\,ds \Big] \notag
 \\ & \leq
 \frac{1}{2}  \E \Big[\int_{\R^2} \int_0^T
 \sigma^2(u(t,x)) \beta^{\prime \prime}(u(t,x)-u_{\Delta y}(t,y))\psi(t,y)\rho_{\delta}(x-y) \,dt \,dx \,dy \Big] \notag
\\ & +
 \frac{1}{2}  \E \Big[\int_{\Pi_T} \int_{\Pi_T}
 \sigma^2(u(t,x)) \beta^{\prime \prime}(u(t,x)-u_{\Delta y}(t,y))\Big[\psi(s,y)-\psi(t,y)\Big]\rho_{\delta}(x-y)\rho_{\delta_0}(t-s) \,dt \,dx \,dy\,ds \Big] \notag
 \\&+
 \frac{C}{\xi^2}  \E \Big[\int_{\Pi_T} \int_{\Pi_T}
\int_{\R}  \sigma^2(u(t,x)) \Big[|u_{\Delta y}(s,y)  - u_{\Delta y}(t,y)| + l\Big] \phi_{\delta,\delta_0}(t,x,s,y)
 \,J_l(k) \,dt \,dx \,dk \,dy\,ds \Big] \notag
 \\ & \leq
 \frac{1}{2}  \E \Big[\int_{\R^2} \int_0^T
 \sigma^2(u(s,x)) \beta^{\prime \prime}(u(s,x)-u_{\Delta y}(s,y))\psi(s,y)\rho_{\delta}(x-y) \,ds \,dx \,dy \Big] 
+ C\frac{\delta_0}{\xi^2} + 
 \frac{C}{\xi^2}  (\sqrt{\delta_0}+l).\label{esti:i4-final}
\end{align}
In view of the estimates \eqref{esti:j4-final} and \eqref{esti:i4-final}, we obtain 
\begin{align}
I_4 + J_4 & \le \frac{1}{2}  \E \Big[\int_{\Pi_T} \int_{\R} \beta^{\prime\prime}\big(u_{\Delta y}(s,y)-u(s,x)\big) \Big\{\sigma^2(u(s,x)) + \sigma^2(u_{\Delta y}(s,y)) \Big\} 
\psi(s,y)\varrho_{\delta}(x-y) \,dy \,ds\,dx \Big] \notag \\
 & \hspace{3cm} + C (\frac{\delta_0}{\xi} + \frac{\sqrt{\delta_0}}{\xi^2}  + \frac{l}{\xi^2}). \label{esti:i4+j4-final}
\end{align}
%%%%%%%%%%%%%%%%%%%%%%%%%%%%%%%%%%%%%%%%%%%%%%%%%%%%%%%%%%%%%%%%%%%%%%%%%%%%%%%%%%%%%%%%%%%%%%%%%%%%%%%%%%%%%%%%%%%%%%%%%%%%%
%%%%%%%%%%%%%%%%%%%%%%%%%%%%%%%%%%%%%%%%%%%%%%%%%%%%%%%
Next we consider It\^{o} integral terms. Note that, for any two constants $t_1, t_2 \ge 0$ with $t_1 < t_2$, 
\begin{align}\label{independency}
\begin{cases}
\E \Big[ X_{t_1} \int_{t_1}^{t_2} J(t)\,dW(t)\Big]=0, \vspace{0.1cm}\\
\E\Big[X_{t_1}\int_{t_1}^{t_2} \int_{|z|>0} \zeta(t,z)\,\widetilde{N}(dz,dt)\Big] = 0,
\end{cases}
\end{align}
where $ J,\zeta$ are predictable integrand with 
$\E \Big[\int_0^T \int_{|z|>0}\zeta^2(t,z)\, m(dz)\, dt\Big] < +\infty$ 
and $X(\cdot)$ is an adapted process. We claim that $J_3=0$. Indeed, in view of \eqref{independency}
\begin{align*}
 J_3 & = \int_{\Pi_T} \int_{\R^2}   \E\Big[ J_l(u(t,x)-k) \int_{s=t}^{s=t+ \delta_0} 
  \sigma(u_{\Delta y}(s,y)) \beta^{\prime}(u_{\Delta y}(s,y)-k)\phi_{\delta,\delta_0}(t,x,s,y)\,dW(s)\Big] dk\,dy\,dt\,dx \notag\\
  &= 0. 
\end{align*}
To proceed further, we define
\begin{align*}
 \mathcal{M}[\beta,\phi_{\delta,\delta_0}](s,y,k)= \int_{\Pi_T} \sigma(u(t,x))\beta\big(u(t,x)-k\big) \phi_{\delta,\delta_0}(t,x,s,y)\,dx\, dW(t).
\end{align*}
Regarding $\mathcal{M}[\beta,\phi_{\delta,\delta_0}](s,y,k)$, we have the following lemma whose proof could be
found in \cite{BisMaj,BisKarlMaj}.
\begin{lem} The following identities hold:
 \begin{align*}
    &\partial_k  \mathcal{M}[\beta, \phi_{\delta,\delta_0}](s, y, k) =  J[-\beta^\prime, \phi](s, y, k)\\
    & \partial_{y}  \mathcal{M}[\beta, \phi_{\delta,\delta_0}](s, y, k) =  J[\beta, \partial_{y}\phi](s, y, k).
   \end{align*}
Moreover, let $ \beta=\beta_\xi$ be a function defined previously. Then there exists a constant $C=C(\psi )$ such that
 \begin{align}
 &\sup_{0\le s\le T} \E\Big[|| \mathcal{M}[\beta^{\prime\prime},\phi_{\delta,\delta_0}](s,\cdot,\cdot)||_{L^\infty(\R\times \R)}^2\Big] \le
  \frac{C(\psi )}{ \xi^\frac{3}{2}\,\delta_0^\frac{3}{2}},\label{eq:l-infinity bound-1} \\
&\sup_{0\le s\le T} \E \Big[|| \mathcal{M}[\beta^{\prime\prime\prime},\phi_{\delta,\delta_0}](s,\cdot,\cdot)||_{L^\infty(\R\times \R)}^2\Big] \le
  \frac{C(\psi )}{ \xi^2\,\delta_0^\frac{3}{2}}.\label{eq:l-infinity bound-2}
 \end{align}
\end{lem}
Note that, in view of the Fubini's theorem and \eqref{independency},
\begin{align*}
 \E\Big[\int_{\Pi_T}\int_{\R} \mathcal{M}[\beta^\prime,\phi_{\delta,\delta_0}](s,y,k) J_l \big(u_{\Delta y}(s-\delta_0,y)-k\big) \,dk\,dy\,ds\Big]=0,
\end{align*}
and hence we rewrite $I_3$  as 
\begin{align}
I_3 &= \E\Big[ \int_{\R}\int_{\Pi_T} \mathcal{M}[\beta^\prime,\phi_{\delta,\delta_0}](s,y,k)\Big\{ J_l\big(u_{\Delta y}(s,y)-k\big) -J_l\big(u_{\Delta y}(s-\delta_0,y)-k\big)\Big\}\,dy\,ds\,dk\Big] \notag \\
& = \E\Big[ \int_{\R}\int_{0}^T \sum_{j\in \mathbb{Z}} \int_{y_{j-\frac{1}{2}}}^{y_{j+\frac{1}{2}}}
\mathcal{M}[\beta^\prime,\phi_{\delta,\delta_0}](s,y,k)\Big\{ J_l\big(u_j(s)-k\big) -J_l\big(u_j(s-\delta_0)-k\big)\Big\}
\,dy\,ds\,dk\Big].\notag
\end{align}
To proceed further, we apply It\^{o} formula to $J_l(u_j(\cdot)-k)$ along with It\^{o} product rule and have 
\begin{align}
 I_3&=  \E\Big[ \int_{\Pi_T} \sum_{j\in \mathbb{Z}} \int_{y_{j-\frac{1}{2}}}^{y_{j+\frac{1}{2}}} 
 \mathcal{M}[\beta^{\prime\prime},\phi_{\delta,\delta_0}](s,y,k) \Bigg( \int_{r=s-\delta_0}^{s} J_l(u_j(r)-k)
 \frac{f(u_j(r))-f(u_{j+1}(r))}{\Delta y}\,dr\Bigg)\,dy\,ds\,dk\Big] \notag \\
& \quad  +  \E\Big[ \int_{\Pi_T} \sum_{j\in \mathbb{Z}} \int_{y_{j-\frac{1}{2}}}^{y_{j+\frac{1}{2}}} \int_{r=s-\delta_0}^s \int_{\R} 
 \beta^\prime(u(r,x)-k) J_l^\prime\big(u_j(r)-k\big) \sigma(u_j(r))\sigma(u(r,x)) \notag \\
 & \hspace{7cm} \times \phi_{\delta,\delta_0}(r,x,s,y)\,dx\,dr\,dy\,ds\,dk\Big] \notag \\
 &  \qquad + \frac{1}{2}  \E\Big[ \int_{\Pi_T} \sum_{j\in \mathbb{Z}} \int_{y_{j-\frac{1}{2}}}^{y_{j+\frac{1}{2}}} 
 \mathcal{M}[\beta^{\prime\prime\prime},\phi_{\delta,\delta_0}](s,y,k) \Bigg( \int_{r=s-\delta_0}^{s} \sigma^2(u_j(r))
 J_l(u_j(r)-k)\,dr\Bigg)\,dy\,ds\,dk\Big] \notag \\
 &= \E\Big[ \int_{\Pi_T} \int_{\R}
 \mathcal{M}[\beta^{\prime\prime},\phi_{\delta,\delta_0}](s,y,k) \Bigg( \int_{r=s-\delta_0}^{s} 
 \frac{f(u_{\Delta y}(r,y))-f(u_{\Delta y}(r,y+\Delta y))}{\Delta y} \notag \\
 & \hspace{7cm} \times  J_l\big(u_{\Delta y}(r,y)-k\big)\,dr\Bigg)\,dy\,ds\,dk\Big] \notag \\
& \quad  -  \E\Big[ \int_{\Pi_T} \int_{\R} \int_{r=s-\delta_0}^s \int_{\R} 
 \beta^{\prime\prime}(u(r,x)-k) J_l\big(u_{\Delta y}(r,y)-k\big) \sigma(u_{\Delta y}(r,y))\sigma(u(r,x)) \notag \\
 & \hspace{9cm} \times \phi_{\delta,\delta_0}(r,x,s,y)\,dx\,dr\,dy\,ds\,dk\Big] \notag \\
 &  \qquad + \frac{1}{2}  \E\Big[ \int_{\Pi_T} \int_{\R}
 \mathcal{M}[\beta^{\prime\prime\prime},\phi_{\delta,\delta_0}](s,y,k) \Bigg( \int_{r=s-\delta_0}^{s} \sigma^2(u_{\Delta y}(r,y))
 J_l\big(u_{\Delta y}(r,y)-k\big)\,dr\Bigg)\,dy\,ds\,dk\Big]\notag \\
 & := I_{3,1} + I_{3,2} + I_{3,3}.\label{esti:i3}
\end{align}
Let us first consider the expression $I_{3,1}$. 
\begin{align}
 I_{3,1}& \le \frac{C ||f^\prime||_{L^\infty}}{\Delta y}  \E\Big[ \int_{\Pi_T} \int_{\R} \int_{r=s-\delta_0}^s \big| \mathcal{M}[\beta^{\prime\prime},\phi_{\delta,\delta_0}](s,y,k) \big|
 \big|u_{\Delta y}(r,y)-u_{\Delta y}(r,y+\Delta y)\big| \notag \\
 & \hspace{5cm} \times J_l\big(u_{\Delta y}(r,y)-k\big)\,dr\,dy\,ds\,dk\Big] \notag \\
 & \le \frac{C ||f^\prime||_{L^\infty}}{\Delta y}  \E\Big[ \int_{K}\int_{0}^T \int_{r=s-\delta_0}^s 
 \| \mathcal{M}[\beta^{\prime\prime},\phi_{\delta,\delta_0}](s,\cdot,\cdot) \|_{L^\infty(\R\times\R)}\big|u_{\Delta y}(r,y)-u_{\Delta y}(r,y+\Delta y)\big|\,dr\,dy\,ds\Big] \notag \\
 & \le \frac{C ||f^\prime||_{L^\infty}}{\Delta y}  \Big( \int_{K}\int_{0}^T \int_{r=s-\delta_0}^s  \E \Big[  \| \mathcal{M}[\beta^{\prime\prime},\phi_{\delta,\delta_0}](s,\cdot,\cdot) \|_{L^\infty(\R\times\R)}^2\Big]
 \,dr\,ds\,dy\Big)^\frac{1}{2}  \notag \\
 & \hspace{4cm} \times \Big( \E \Big[\int_{K}\int_{0}^T \int_{r=s-\delta_0}^s \big|u_{\Delta y}(r,y)-u_{\Delta y}(r,y+\Delta y)\big|^2 \,dr\,ds\,dy \Big]\Big)^\frac{1}{2}\notag \\
 & \le C ||f^\prime||_{L^\infty} \delta_0 \frac{1}{\xi^\frac{3}{4}\delta_0^\frac{3}{4}} =C \frac{\delta_0^\frac{1}{4}}{\xi^\frac{3}{4}}.\label{esti:i31}
\end{align}
In the above, we have used the fact that $\E\Big[ |u_{\Delta y}(r,\cdot)|_{BV(\R)}\Big] \le  C \, \E\big[|u_0(\cdot)|_{BV(\R)}\big]$ 
and $u_{\Delta y}(r,\cdot) \in L^\infty(\R)$ for all $r\in [0,T]$ almost surely $\omega \in \Omega$ along with the estimates \eqref{eq:l-infinity bound-1}.
\vspace{.2cm}

Let us turn our attention to estimate the term $I_{3,2}$. Following a similar argument as in \cite[Lemma 5.6]{BisKarlMaj}, we see that
\begin{align}
I_{3,2} &\le - \E\Big[ \int_{\Pi_T}\int_{\R^2}\beta^{\prime\prime}(u(r,x)-k) \sigma(u_{\Delta y}(r,y))\sigma(u(r,x)) J_l\big(u_{\Delta y}(r,y)-k\big)
  \psi(r,y)\varrho_{\delta}(x-y)\,dk\,dx\,dy\,dr\Big] \notag \\
  & \hspace{6cm} + C\frac{\delta_0}{\xi} \notag \\ 
 & \le - \E\Big[ \int_{\Pi_T}\int_{\R} \beta^{\prime\prime}\big(u(r,x)-u_{\Delta y}(r,y)\big) \sigma(u_{\Delta y}(r,y))\sigma(u(r,x))
  \psi(r,y)\varrho_{\delta}(x-y)\,dx\,dy\,dr\Big] + C\frac{\delta_0}{\xi} \notag \\
  & \qquad  + \Big| \E\Big[ \int_{\Pi_T}\int_{\R}\int_{\R} \Big\{\beta^{\prime\prime}\big(u(r,x)-u_{\Delta y}(r,y) + k\big) - 
   \beta^{\prime\prime}\big(u(r,x)-u_{\Delta y}(r,y)\big)\Big\} \notag \\
   & \hspace{4cm} \times \sigma(u_{\Delta y}(r,y))\sigma(u(r,x)) J_l(k) \psi(r,y)\varrho_{\delta}(x-y)\,dk\,dx\,dy\,dr\Big]\Big| \notag \\
   & \le - \E\Big[ \int_{\Pi_T}\int_{\R} \beta^{\prime\prime}\big(u(r,x)-u_{\Delta y}(r,y)\big) \sigma(u_{\Delta y}(r,y))\sigma(u(r,x))
  \psi(r,y)\varrho_{\delta}(x-y)\,dx\,dy\,dr\Big]  \notag \\
  & \hspace{6cm} + C \big(\frac{\delta_0}{\xi} + \frac{l}{\xi^2}\big).\label{esti:i32}
\end{align}
Next we want to estimate the term $I_{3,3}$. In view of the uniform moment estimate \eqref{uni:moment} and the estimate \eqref{eq:l-infinity bound-2}, we have 
\begin{align}
 I_{3,3} &= \frac{1}{2}  \E\Big[ \int_{\Pi_T} \int_{\R}
 \mathcal{M}[\beta^{\prime\prime\prime},\phi_{\delta,\delta_0}](s,y,k) \Big( \int_{r=s-\delta_0}^{s} \sigma^2(u_{\Delta y}(r,y))
 J_l\big(u_{\Delta y}(r,y)-k\big)\,dr\Big)\,dy\,ds\,dk\Big] \notag
 \\ & \le 
 \E\Big[ \int_{|y|<c_\psi} \int_{0}^T  \int_{r=s-\delta_0}^s   \| \mathcal{M}[\beta^{\prime\prime\prime},\phi_{\delta,\delta_0}](s,\cdot,\cdot) \|_{L^\infty(\R\times\R)}
  \sigma^2(u_{\Delta y}(r,y))\,dr\,ds\,dy\Big] \notag \\
  & \le C\, \E\Big[ \int_{|y|<c_\psi} \int_{0}^T  \int_{r=s-\delta_0}^s   \| \mathcal{M}[\beta^{\prime\prime\prime},\phi_{\delta,\delta_0}](s,\cdot,\cdot) \|_{L^\infty(\R\times\R)}
  u_{\Delta y}^2(r,y)\,dr\,ds\,dy\Big] \notag \\
  & \le C\, \int_0^T \int_{r=s-\delta_0}^s  \Big( \E\Big[  \| \mathcal{M}[\beta^{\prime\prime\prime},\phi_{\delta,\delta_0}]
  (s,\cdot,\cdot) \|_{L^\infty(\R\times\R)}^2 \Big]\Big)^\frac{1}{2} \notag \\
  & \hspace{5cm} \times \Big( \E\Big[  \int_{|y|<c_\psi}  |u_{\Delta y}(r,y)|^4\,dy \Big] \Big)^\frac{1}{2}\,dr\,ds\notag \\
  & \le \frac{C}{\xi} \delta_0^\frac{1}{4}T \Big( \sup_{\Delta y >0}\sup_{0\le t\le T}
  \E \Big[ \|u_{\Delta y}(t,\cdot)\|_4^4 \Big]\Big)^\frac{1}{2} \le \frac{C}{\xi}\delta_0^\frac{1}{4}.\label{esti:i33}
\end{align}
We combine \eqref{esti:i31}-\eqref{esti:i33} in \eqref{esti:i3} and have 
\begin{align}
 I_3 +J_3  &\le - \E\Big[ \int_{\Pi_T}\int_{\R} \beta^{\prime\prime}\big(u(r,x)-u_{\Delta y}(r,y)\big) \sigma(u_{\Delta y}(r,y))\sigma(u(r,x))
  \psi(r,y)\varrho_{\delta}(x-y)\,dx\,dy\,dr\Big]\notag \\
  & \hspace{3cm} + C\big( \frac{\delta_0}{\xi} + \frac{l}{\xi^2}\big) + C\frac{\delta_0^\frac{1}{4}}{\xi},\label{esti:i3+j3-final}
\end{align}
and combining \eqref{esti:i4+j4-final} and \eqref{esti:i3+j3-final} gives
\begin{align}
& I_3 +J_3 +  I_4 +J_4 
\\& \leq 
 \frac{1}{2}  \E \Big[\int_{\Pi_T} \int_{\R} \beta^{\prime\prime}\big(u_{\Delta y}(s,y)-u(s,x)\big) \Big[\sigma(u(s,x)) - \sigma(u_{\Delta y}(s,y))\Big]^2  
\psi(s,y)\varrho_{\delta}(x-y) \,dy \,ds\,dx \Big] \notag \\
 & \hspace{3cm} + C (\frac{\delta_0+\delta_0^\frac{1}{4}}{\xi} + \frac{\sqrt{\delta_0}}{\xi^2}  + \frac{l}{\xi^2}). 
 \end{align}
Collecting all the above results, we have proved the following Lemma.
\begin{lem}\label{I3+J3+I4+J4}
\begin{align*}
& \limsup_{l \rightarrow 0}\,\limsup_{\delta_0 \rightarrow 0} \, (I_3 +J_3 +  I_4 +J_4)
\\& \leq 
 \frac{1}{2}  \E \Big[\int_{\Pi_T} \int_{\R} \beta^{\prime\prime}\big(u_{\Delta y}(s,y)-u(s,x)\big) \Big[\sigma(u(s,x)) - \sigma(u_{\Delta y}(s,y))\Big]^2  
\psi(s,y)\varrho_{\delta}(x-y) \,dy \,ds\,dx \Big]. 
 \end{align*}
\end{lem}

Let us turn our attention to estimate the stochastic integral terms corresponding to jump noise.  Note that, thanks to \eqref{independency}
\begin{align*}
 J_6= & \int_{\Pi_T}\int_{\R^2}\int_0^1 \E \Big[ J_l(u(t,x)-k) \int_{s=t}^{s=t+ \delta_0} \int_{|z|>0} \eta(u_{\Delta y}(s,y);z) 
 \beta^{\prime}\big(u_{\Delta y}(s,y)-k + \lambda \eta(u_{\Delta y}(s,y);z) \big) \notag \\
 & \hspace{6cm} \times \phi_{\delta,\delta_0}(t,x,s,y) \widetilde{N}(dz,ds) \Big] \,d\lambda\,dk\,dx\,dy\,dt \notag \\
 &=0.
\end{align*}
 We  define 
\begin{align*}
\mathcal{J}[\beta, \phi_{\delta,\delta_0}](s; y,k) :=&
\int_{\Pi_T}\int_{|z|>0} \Big(\beta\big(u(r,x) +\eta(u(r,x); z)-k\big)
 -\beta\big(u(r,x)-k\big)\Big) \\
&\hspace{3cm}\times\phi_{\delta,\delta_0}(r,x,s,y)\widetilde{N}(dz,dr)\,dx. 
\end{align*} 
In view of Lemmas $5.4$ and $5.5$ of \cite{BisKarlMaj}, one can deduce the following: 
\begin{align*}
    &\partial_k  \mathcal{J}[\beta, \phi_{\delta,\delta_0}](s, y, k) =  \mathcal{J}[-\beta^\prime, \phi](s, y, k)\\
    & \partial_{y}  \mathcal{J}[\beta, \phi_{\delta,\delta_0}](s, y, k) =  \mathcal{J}[\beta, \partial_{y}\phi](s, y, k).
   \end{align*}
Moreover, since $ \beta=\beta_\xi$, there exists a constant $C=C(\psi )$ such that
 \begin{align}
 &\sup_{0\le s\le T} \E\Big[|| \mathcal{J}[\beta^{\prime},\phi_{\delta,\delta_0}](s,\cdot,\cdot)||_{L^\infty(\R\times \R)}^2\Big] \le
  \frac{C(\psi )}{ \xi^\frac{3}{2}\,\delta_0^\frac{3}{2}},\label{eq:l-infinity bound-1-1} \\
&\sup_{0\le s\le T} \E \Big[|| \mathcal{J}[\beta^{\prime\prime},\phi_{\delta,\delta_0}](s,\cdot,\cdot)||_{L^\infty(\R\times \R)}^2\Big] \le
  \frac{C(\psi )}{ \xi^2\,\delta_0^\frac{3}{2}}.\label{eq:l-infinity bound-2-2}
 \end{align}
 Arguing similarly as we have done for the term $I_3$, we arrive at 
 
 \begin{align}
  I_6=& \E\Big[ \int_{\Pi_T} \int_{\R}
 \mathcal{J}[\beta^{\prime},\phi_{\delta,\delta_0}](s,y,k) \Big( \int_{r=s-\delta_0}^{s} 
 \frac{f(u_{\Delta y}(r,y))-f(u_{\Delta y}(r,y+\Delta y))}{\Delta y} \notag \\
 & \hspace{7cm} \times  J_l\big(u_{\Delta y}(r,y)-k\big)\,dr\Big)\,dy\,ds\,dk\Big] \notag \\
 & +  \E\Big[ \int_{\Pi_T} \int_{\R}\mathcal{J}[\beta^{\prime\prime},\phi_{\delta,\delta_0}](s,y,k) \Big\{ \int_{r=s-\delta_0}^{s} \int_{|z|>0} \int_0^1 
 J_l\big( u_{\Delta y}(r,y) -k + \lambda \eta(u_{\Delta y}(r,y);z)\big) \notag \\
 & \hspace{6cm} \times (1-\lambda) \eta^2(u_{\Delta y}(r,y);z)\,\,d\lambda\,m(dz)\,dr \Big\}\,dk\,dy\,ds \Big] \notag \\
 & +  \E\Big[ \int_{\Pi_T}\int_{\R^2} \int_{r=s-\delta_0}^{s} \int_{|z|>0} \Big\{ \beta\big( u(r,x)-k + \eta(u(r,x);z)\big) - \beta(u(r,x)-k)\Big\} \notag \\
 & \hspace{4cm} \times \Big( J_l \big( u_{\Delta y}(r,y) -k +\eta(u_{\Delta y}(r,y);z)\big)- J_l( u_{\Delta y}(r,y)-k)\Big)  \notag \\
 & \hspace{5.5cm} \times \rho_{\delta_0}(r-s)\psi(s,y)\,\varrho_{\delta}(x-y)\,m(dz)\,dr\,dx\,dk\,dy\,ds\Big] \notag \\
 &:=  I_{6,1}+  I_{6,2} +  I_{6,3}. \label{sum:i61+i62+i63}
 \end{align}
Observe that, in view of \eqref{eq:l-infinity bound-1-1}, \eqref{eq:l-infinity bound-2-2} and uniform moment estimate \eqref{uni:moment} along with the assumptions \ref{A4} and \ref{A5}, 
 \begin{align}
 \label{esti:i61+i62}
 \begin{aligned}
  I_{6,1} \le C \|f^\prime\|_{\infty} \frac{\delta_0^\frac{1}{4}}{\xi^\frac{3}{4}}, \quad \text{and} \,\,\,\,
   I_{6,3} \le C \frac{C}{\xi}\delta_0^\frac{1}{4}.
   \end{aligned} 
 \end{align}
 Next we want to estimate the term $I_{6,3}$.  A similar argument as in \cite[Lemma 5.6]{BisKarlMaj} leads to 
 \begin{align}
  I_{6,3} & \le \E \Big[ \int_{r=0}^T \int_{\R^2} \int_{s=r}^{r+ \delta_0} \int_{\R}\int_{|z|>0} \int_0^1 \int_0^1 \beta^{\prime\prime}\big( u(r,x)-k + \lambda \eta(u(r,x);z) \big) 
  \eta(u(r,x);z)\eta(u_{\Delta y}(r,y);z) \notag \\
  & \hspace{5.5cm} \times  J_l\big( u_{\Delta y}(r,y)-k + \theta \eta(u_{\Delta y}(r,y);z)\big) \rho_{\delta_0}(r-s)\psi(s,y) \notag \\
  & \hspace{6cm} \times \,\varrho_{\delta}(x-y)
  \,d\theta\,d\lambda\,m(dz)\,dx\,ds\,dk\,dy\,dr\Big] + C \frac{\delta_0}{\xi} \notag \\
  & \le \E \Big[ \int_{0}^T \int_{\R^3}\int_{|z|>0} \int_0^1 \int_0^1 \beta^{\prime\prime}\big( u(r,x)-k + \lambda \eta(u(r,x);z) \big) 
  \eta(u(r,x);z)\eta(u_{\Delta y}(r,y);z) \notag \\
  & \hspace{5.5cm} \times  J_l\big( u_{\Delta y}(r,y)-k + \theta \eta(u_{\Delta y}(r,y);z)\big)\psi(r,y) \notag \\
  & \hspace{6cm} \times \,\varrho_{\delta}(x-y)
  \,d\theta\,d\lambda\,m(dz)\,dx\,dk\,dy\,dr\Big] + C \frac{\delta_0}{\xi} \notag \\
  & = \E \Big[ \int_{0}^T \int_{\R^3}\int_{|z|>0} \Big\{ \beta\big( u(r,x)-k + \eta(u(r,x);z)\big) - \beta(u(r,x)-k)\Big\} \notag \\
 & \hspace{4cm} \times \Big( J_l \big( u_{\Delta y}(r,y) -k +\eta(u_{\Delta y}(r,y);z)\big)- J_l( u_{\Delta y}(r,y)-k)\Big)  \notag \\
 & \hspace{5.5cm} \times \psi(r,y)\,\varrho_{\delta}(x-y)\,m(dz)\,dx\,dk\,dy\,dr\Big] + C \frac{\delta_0}{\xi} \notag \\
 & \le \E\Big[\int_{\Pi_{T}}\int_{\R} \int_{|z|>0}  \Big\{ \beta\big(u(r,x)+ \eta(u(r,x);z)
  - u_{\Delta y}(r,y)-\eta(u_{\Delta y}(r,y);z)\big) \notag \\
  & \hspace{2.5cm}-\beta\big(u(r,x)-u_{\Delta y}(r,y)-\eta(u_{\Delta y}(r,y);z)\big)-\beta\big(u(r,x)+ \eta(u(r,x);z)-u_{\Delta y}(r,y)\big)
   \notag \\
  &  \hspace{3cm} + \beta \big(u(r,x)-u_{\Delta y}(r,y)\big) \Big\}\psi(r,y)\,\varrho_{\delta}(x-y)\,m(dz)\,dx\,dy\,dr\Big] + C\big(l +  \frac{\delta_0}{\xi}\big). 
  \label{esti:i63}
 \end{align}
 Thus, combining \eqref{esti:i61+i62} and \eqref{esti:i63} in \eqref{sum:i61+i62+i63} and keeping in mind that $J_6=0$, we have 
 \begin{align}
  I_6 + J_6 & \le \E\Big[\int_{\Pi_{T}}\int_{\R} \int_{|z|>0}  \Big\{ \beta\big(u(r,x)+ \eta(u(r,x);z)
  - u_{\Delta y}(r,y)-\eta(u_{\Delta y}(r,y);z)\big) \notag \\
  & \hspace{2.5cm}-\beta\big(u(r,x)-u_{\Delta y}(r,y)-\eta(u_{\Delta y}(r,y);z)\big)-\beta\big(u(r,x)+ \eta(u(r,x);z)-u_{\Delta y}(r,y)\big)
   \notag \\
  &  \hspace{2.5cm} + \beta \big(u(r,x)-u_{\Delta y}(r,y)\big) \Big\}\psi(r,y)\,\varrho_{\delta}(x-y)\,m(dz)\,dx\,dy\,dr\Big] + C \Big( \frac{\delta_0^\frac{1}{4}}{\xi^\frac{3}{4}}
  + \frac{\delta_0^\frac{1}{4}}{\xi} + l\Big).\label{esti:i6+j6-final}
 \end{align}
%%%%%%%%%%%%%%%%%%%%%%%%%%%%%%%%%%%%%%%%%%%%%%%%%%%%
Now we focus on the term $J_7$. Thanks to uniform moment estimate \eqref{uni:moment}, Lemma~\ref{continuity in time}, assumptions \ref{A4} and \ref{A5}, and the bound of $\beta^{\prime\prime}$ and $\beta^{\prime\prime\prime}$, we obtain
\begin{align}
 J_7= & \E\Big[ \int_{\Pi_T}\int_{\R}\int_{\Pi_T} \int_{|z|>0} \int_0^1 (1-\lambda) \eta^2(u_{\Delta y}(s,y);z) \Big\{ \beta^{\prime\prime} \big( u_{\Delta y}(s,y)
 + \lambda \eta(u_{\Delta y}(s,y);z)-u(t,x) + k \big) \notag \\
 & \hspace{4 cm}- \beta^{\prime\prime} \big( u_{\Delta y}(s,y)
 + \lambda \eta(u_{\Delta y}(s,y);z)-u(s,x) + k \big)\Big\} \phi_{\delta,\delta_0}(t,x,s,y) \notag \\
 & \hspace{5cm}\times J_l(k)\,d\lambda\,m(dz)\, ds\,dy\,dk\,dx\,dt \Big] \notag \\
 & -  \E\Big[ \int_{\Pi_T}\int_{\R^2} \int_{|z|>0} \int_0^1 (1-\lambda) \eta^2(u_{\Delta y}(s,y);z) \beta^{\prime\prime} \big( u_{\Delta y}(s,y)
 + \lambda \eta(u_{\Delta y}(s,y);z)-u(s,x) + k \big) \notag \\
 & \hspace{4cm} \times \psi(s,y) \big( 1- \int_0^T \rho_{\delta_0}(t-s)\,dt \big) \rho_{\delta}(x-y) J_l(k)\,d\lambda\,m(dz)\, ds\,dy\,dk\,dx \Big] \notag \\
 & + \E\Big[ \int_{\Pi_T}\int_{\R^2} \int_{|z|>0} \int_0^1 (1-\lambda) \eta^2(u_{\Delta y}(s,y);z) \beta^{\prime\prime} \big( u_{\Delta y}(s,y)
 + \lambda \eta(u_{\Delta y}(s,y);z)-u(s,x) + k \big) \notag \\
 & \hspace{4cm} \times \psi(s,y) \rho_{\delta}(x-y) J_l(k)\,d\lambda\,m(dz)\, ds\,dy\,dk\,dx \Big] \notag \\
 & \le \E\Big[ \int_{\Pi_T}\int_{\R^2} \int_{|z|>0} \int_0^1 (1-\lambda) \eta^2(u_{\Delta y}(s,y);z) \beta^{\prime\prime} \big( u_{\Delta y}(s,y)
 + \lambda \eta(u_{\Delta y}(s,y);z)-u(s,x) + k \big) \notag \\
 & \hspace{4cm} \times \psi(s,y) \rho_{\delta}(x-y) J_l(k)\,d\lambda\,m(dz)\, ds\,dy\,dk\,dx \Big] \notag \\ 
 & + \frac{C}{\xi^2} \E\Big[\int_{\Pi_T}\int_{\Pi_T} \int_{|z|>0} \eta^2(u_{\Delta y}(s,y);z)\big|u(t,x)-u(s,x)\big| 
 \psi(s,y) \varrho_{\delta}(x-y) \rho_{\delta_0}(t-s)\,m(dz)\,dy \,ds\,dx\,dt \Big] \notag \\
 & + \frac{C}{\xi} \E\Big[\int_{s=0}^{\delta_0}\int_{R^2} \int_{|z|>0} \eta^2(u_{\Delta y}(s,y);z) 
 \psi(s,y) \varrho_{\delta}(x-y)\,m(dz)\,dy\,dx\,ds \Big] \notag \\
 &  \le \E\Big[ \int_{\Pi_T}\int_{\R^2} \int_{|z|>0} \int_0^1 (1-\lambda) \eta^2(u_{\Delta y}(s,y);z) \beta^{\prime\prime} \big( u_{\Delta y}(s,y)
 + \lambda \eta(u_{\Delta y}(s,y);z)-u(s,x) + k \big) \notag \\
 & \hspace{4cm} \times \psi(s,y) \rho_{\delta}(x-y) J_l(k)\,d\lambda\,m(dz)\, ds\,dy\,dk\,dx \Big] \notag \\ 
 & + \frac{C}{\xi^2} \E\Big[\int_{\Pi_T}\int_{\Pi_T} \int_{|z|>0} \big( 1+ u^2_{\Delta y}(s,y)\big) (1\wedge |z|^2)\big|u(t,x)-u(s,x)\big| 
 \psi(s,y) \varrho_{\delta}(x-y) \notag \\
 & \hspace{5cm} \times \rho_{\delta_0}(t-s)\,m(dz)\,dy \,ds\,dx\,dt \Big] \notag \\
 & + \frac{C}{\xi} \E\Big[\int_{s=0}^{\delta_0}\int_{R^2} \int_{|z|>0} \big( 1+ u^2_{\Delta y}(s,y)\big) (1\wedge |z|^2)
 \psi(s,y) \varrho_{\delta}(x-y)\,m(dz)\,dy\,dx\,ds \Big] \notag \\
 & \le  \E\Big[ \int_{\Pi_T}\int_{\R^2} \int_{|z|>0} \int_0^1 (1-\lambda) \eta^2(u_{\Delta y}(s,y);z) \beta^{\prime\prime} \big( u_{\Delta y}(s,y)
 + \lambda \eta(u_{\Delta y}(s,y);z)-u(s,x) + k \big) \notag \\
 & \hspace{4cm} \times \psi(s,y) \rho_{\delta}(x-y) J_l(k)\,d\lambda\,m(dz)\, ds\,dy\,dk\,dx \Big]
  +  C \frac{\sqrt{\delta_0}}{\xi^2} +  C \frac{\delta_0}{\xi} \notag \\
& \le \E\Big[ \int_{\Pi_T}\int_{\R} \int_{|z|>0} \int_0^1 (1-\lambda) \eta^2(u_{\Delta y}(s,y);z) \beta^{\prime\prime} \big( u_{\Delta y}(s,y)
 + \lambda \eta(u_{\Delta y}(s,y);z)-u(s,x) \big) \notag \\
 & \hspace{4cm} \times \psi(s,y) \rho_{\delta}(x-y)\,d\lambda\,m(dz)\,ds\,dy\,dx \Big]
  +  C \frac{\sqrt{\delta_0}}{\xi^2} +  C \frac{\delta_0}{\xi} + C \frac{l}{\xi^2}. \label{esti:j7-final}
\end{align}
Again, one can show that (similar to the estimation of $I_4$ term)
\begin{align}
 I_7 & \le  \E\Big[ \int_{\Pi_T}\int_{\R} \int_{|z|>0} \int_0^1 (1-\lambda) \eta^2(u(s,x);z) \beta^{\prime\prime} \big( u(s,x)- u_{\Delta y}(s,y)
 + \lambda \eta(u(s,x);z) \big) \notag \\
 & \hspace{4cm} \times \psi(s,y) \rho_{\delta}(x-y)\,d\lambda\,m(dz)\,dy\,dx\,ds \Big] + C \big(\frac{\delta_0}{\xi}
 + \frac{\sqrt{\delta_0}}{\xi^2}  + \frac{l}{\xi^2}\big). \label{esti:i7-final}
\end{align}
Thus, from \eqref{esti:j7-final} and \eqref{esti:i7-final}, we have 
\begin{align}
 I_7 + J_7 & \le  \E\Big[ \int_{\Pi_T}\int_{\R} \int_{|z|>0} \int_0^1 (1-\lambda) \Big\{ \eta^2(u_{\Delta y}(s,y);z) \beta^{\prime\prime} \big( u_{\Delta y}(s,y)
 + \lambda \eta(u_{\Delta y}(s,y);z)-u(s,x) \big) \notag \\
 & \hspace{3.5cm} + \eta^2(u(s,x);z) \beta^{\prime\prime} \big( u(s,x)- u_{\Delta y}(s,y)
 + \lambda \eta(u(s,x);z) \big) \Big\} \psi(s,y) \notag \\
 & \hspace{4cm} \times \rho_{\delta}(x-y)\,d\lambda\,m(dz)\,dy\,dx\,ds \Big]  + C \big(\frac{\delta_0}{\xi}
 + \frac{\sqrt{\delta_0}}{\xi^2}  + \frac{l}{\xi^2}\big), \label{esti:i7+j7-final}
\end{align}
%%%%%%%%%%%%%%%%%%%%%%%%%%%%%%%%%%%%%%%% 
and hence adding two estimates \eqref{esti:i6+j6-final} and \eqref{esti:i7+j7-final} yields
 \begin{align}
  & I_6 + I_7 + J_6 + J_7 \notag \\
  & \le  \E\Big[ \int_{\Pi_T} \int_{\R}\int_{|z|>0} \Big\{ \beta\big(u(r,x)+ \eta(u(r,x);z) - u_{\Delta y}(r,y)-\eta(u_{\Delta y}(r,y);z)\big)
  -\beta \big(u(r,x)-u_{\Delta y}(r,y)\big) \notag \\
  &\hspace{3cm} -\big( \eta(u(r,x);z)-\eta(u_{\Delta y}(r,y);z)\big) \beta^\prime \big(u(r,x)-u_{\Delta y}(r,y)\big) \Big\}\psi(r,y)\notag \\
  & \hspace{3.5cm} \times \varrho_{\delta}(x-y)\,m(dz)\,dx\,dy\,dr\Big] + C \Big( \frac{\delta_0^\frac{1}{4}}{\xi^\frac{3}{4}}
  + \frac{\delta_0^\frac{1}{4}}{\xi} + l\Big) + C \Big(\frac{\delta_0}{\xi}
 + \frac{\sqrt{\delta_0}}{\xi^2}  + \frac{l}{\xi^2}\Big).
 \end{align}
Thus, we have proved the following Lemma: 
\begin{lem}\label{I6+J6+I7+J7}
\begin{align*}
& \limsup_{l \rightarrow 0}\,\limsup_{\delta_0 \rightarrow 0} \,\big(I_6 +J_6 +  I_7 +J_7\big) \\ 
& \leq \E\Big[ \int_{\Pi_T} \int_{\R}\int_{|z|>0} \Big\{ \beta\big(u(r,x)+ \eta(u(r,x);z) - u_{\Delta y}(r,y)-\eta(u_{\Delta y}(r,y);z)\big)
  -\beta \big(u(r,x)-u_{\Delta y}(r,y)\big) \notag \\
  &\hspace{2cm} -\big( \eta(u(r,x);z)-\eta(u_{\Delta y}(r,y);z)\big) \beta^\prime \big(u(r,x)-u_{\Delta y}(r,y)\big) \Big\}\psi(r,y)
  \varrho_{\delta}(x-y)\,m(dz)\,dx\,dy\,dr\Big]. 
 \end{align*}
\end{lem}
%%%%%%%%%%%%%%%%%%%%%%%%%%%%%%%%%%%%%%%%%%%%%%%%%%%%%%
\noindent Now we are in a position to combine the Lemma's \ref{I1+J1}, \ref{I2+J2}, \ref{I5+J5}, \ref{I3+J3+I4+J4} and \ref{I6+J6+I7+J7}. The result is 
\begin{align}
& 0 \le  \E\Big[\int_{\R}\int_{\R} \beta\big(u_0(x)-u_{\Delta y}(0,y)\big) \varrho_{\delta}(x-y)\psi(0,y)\,dy\,dx\Big] \notag \\
& +  \E\Big[\int_{\Pi_T}\int_{\R} \beta \big(u_{\Delta y}(s,y)-u(s,x)\big) \partial_s \psi(s,y)
\varrho_{\delta}(x-y)\,dx\,dy\,ds \Big] \notag 
\\& +
\E\Big[\int_{\Pi_T}\int_{\R} \Big\{f^\beta \big(u_{\Delta y}(s,y),u(s,x)\big)- f^\beta \big( u(s,x),u_{\Delta y}(s,y)\big) \Big\} \partial_y \varrho_{\delta}(x-y) \psi(s,y)\,dx\,dy\,ds\Big] \notag
\\& \quad + \E\Big[\int_{\Pi_T}\int_{\R} f^\beta \big(u(s,x),u_{\Delta y}(s,y)\big)\varrho_{\delta}(x-y)\partial_y \psi(s,y)\,dx\,dy\,ds\Big] \notag \\
& \quad + \frac{1}{2} \E\Big[\int_{\Pi_T}\int_{\R} \beta^{\prime\prime}\big(u_{\Delta y}(s,y)-u(s,x)\big) \big(\sigma(u_{\Delta y}(s,y)-\sigma(u(s,x))\big)^2 
\varrho_{\delta}(x-y)\psi(s,y)\,dx\,dy\,ds\Big] \notag \\ 
&  + \E\Big[ \int_{\Pi_T} \int_{\R}\int_{|z|>0} \Big\{ \beta\big(u(r,x)+ \eta(u(r,x);z) - u_{\Delta y}(r,y)-\eta(u_{\Delta y}(r,y);z)\big)
  -\beta \big(u(r,x)-u_{\Delta y}(r,y)\big)\notag \\
  &  \hspace{2cm} -\big( \eta(u(r,x);z)-\eta(u_{\Delta y}(r,y);z)\big) \beta^\prime \big(u(r,x)-u_{\Delta y}(r,y)\big) \Big\}\psi(r,y) \varrho_{\delta}(x-y)\,m(dz)\,dx\,dy\,dr\Big] \notag \\
  &\hspace{8cm} + C \frac{\Delta y}{\delta} \notag \\
&  \hspace{2cm}\notag \\
& := \mathcal{A}_1 + \mathcal{A}_2 + \mathcal{A}_3 + \mathcal{A}_4 + \mathcal{A}_5 +  \mathcal{A}_6 + C \frac{\Delta y}{\delta}. \label{esti:0}
\end{align}
Note that since $\Big| f^\beta(x,y)-|y-x|\Big| \leq C\,\xi$, $\lim_{\xi}\mathcal{A}_3 =0$.
Let us consider the term $\mathcal{A}_6$. In fact, by rewriting $\mathcal{A}_6$, we have 
\begin{align}
 \mathcal{A}_6 =\E \Big[\int_{\Pi_T}\int_{\R}\Big( \int_{|z|>0} \int_{\theta=0}^1 b^2
(1-\theta)\beta^{\prime\prime}(a+\theta\,b) \,d\theta\,m(dz)\Big)\,\psi(t,y)\varrho_{\delta}(x-y)\,dx\,dy\,dt\Big],\label{eq:endgame}
\end{align}
where $ a=u(t,x)-u_{\Delta y}(t,y)$ and $b=\eta(u(t,x);z)-\eta(u_{\Delta y}(t,y);z)$. Note that, in view of assumption \ref{A5},
\begin{align}
b^2\beta^{\prime\prime}(a+\theta\,b)
&=\big(\eta(u(t,x);z)-\eta(u_{\Delta}(t,y);z)\big)^2\,
\beta^{\prime\prime}\Big(a+\theta\,\big(\eta(u(t,x);z)-\eta(u_{\Delta y}(t,y);z)\big)\Big)\notag \\
& \leq \big|u(t,x)-u_{\Delta y}(t,y)\big|^2 (1\wedge | z|^2)\,\beta^{\prime\prime}(a+\theta\,b)
= a^2  \,\beta^{\prime\prime}(a+\theta\,b)\, (1\wedge | z|^2), \label{eq:nonlocal-estim}
\end{align}
 and hence we need to find a suitable upper bound on $a^2\,\beta^{\prime\prime}(a+\theta\,b)$. Since $\beta^{\prime\prime}$ is nonnegative and symmetric around zero, 
 we can assume without loss of generality that $a \ge 0$. Thus, thanks to assumption  \ref{A5}, we obtain 
 \begin{align}
0 \le a \le (1-\lambda^*)^{-1}(a+ \theta b),\,\,\,\text{for all}\,\,\,\theta\in [0,1]. \label{eq:nonlocal-estim-1}
\end{align}
We substitute $\beta= \beta_{\xi}$ in 
\eqref{eq:nonlocal-estim}, and 
use \eqref{eq:nonlocal-estim-1} to obtain
\begin{align}
b^2 \beta''_\xi (a+\theta\, b) 
&  \le (1-\lambda^*)^{-2}(a+\theta\,b)^2 
\,\beta^{\prime\prime}_{\xi} (a+\theta b)
\,(|z|^2\wedge 1) \notag \\
& \le C\,\xi\,(|z|^2\wedge 1 ),\label{eq:endgame2}
\end{align}
as $\sup_{r\in \R}\, r^2 \beta^{\prime\prime}_\xi(r) 
\le C \xi$ by \eqref{eq:approx to abosx}. Therefore, by using \ref{A4}, we have from \eqref{eq:endgame}
\begin{align*}
 \mathcal{A}_6 & \le C \xi\, \E \Big[\int_{\Pi_T}\int_{\R} \int_{|z|>0} (|z|^2\wedge 1 ) \psi(t,y)\varrho_{\delta}(x-y)\,dx\,dy\,dt\Big] \le C\,\xi,
\end{align*}
and hence  $\lim_{\xi}\mathcal{A}_6 =0$. 
Note also that
\begin{align*}
0 \leq \beta^{\prime\prime}\big(u_{\Delta y}(s,y)-u(s,x)\big) \big(\sigma(u_{\Delta y}(s,y)-\sigma(u(s,x))\big)^2 \leq M_2 \xi 
\end{align*}
so that, passing to the limit when $\xi \to 0^+$ in  \eqref{esti:0} with Lebesgue Theorem yields
\begin{align}
0 \le & \E\Big[\int_{\R}\int_{\R} \big|u_0(x)-u_{\Delta y}(0,y)\big| \varrho_{\delta}(x-y)\psi(0,y)\,dy\,dx\Big] \notag \\
& +  \E\Big[\int_{\Pi_T}\int_{\R}  \big|u_{\Delta y}(s,y)-u(s,x)\big| \partial_s \psi(s,y)
\varrho_{\delta}(x-y)\,dx\,dy\,ds \Big] \notag \\
& \quad + \E\Big[\int_{\Pi_T}\int_{\R} \mathbf{f}\big(u(s,x),u_{\Delta y}(s,y)\big)
\varrho_{\delta}(x-y)\partial_y \psi(s,y)\,dx\,dy\,ds\Big]  + C \frac{\Delta y}{\delta}. \label{esti:0BIS}
\end{align}

Our aim is to estimate each of the above terms suitably. 
To do this, we follow the ideas from \cite{BisKoleyMaj,Chen:2012fk}. At this point we first let $R \mapsto \infty$ in 
\eqref{esti:0BIS}. Moreover, 
since $|\mathbf{f}^\beta(a,b)|\le \|f^\prime\|_{L^\infty}|a-b|$, for any $a,b\in \R$ 
and $|\partial_y \psi(s,y)|\le C \psi(s,y)$, we conclude that 
\begin{align}
\E\Big[\int_{\Pi_T}\int_{\R} & \mathbf{f}\big(u(s,x),u_{\Delta y}(s,y)\big)\varrho_{\delta}(x-y)\partial_y \psi(s,y)\,dx\,dy\,ds\Big] \notag
\\ \leq& 
C\|f^\prime\|_{L^\infty} \E \Big[\int_{\Pi_T}\int_{\R} \big|u_{\Delta y}(s,y)-u(s,x)\big| \psi(s,y)\varrho_{\delta}(x-y)\,dx\,dy\,ds\Big].\label{esti:a4}
\end{align}

 To proceed further, we make a special choice for the function $\psi(s,y)$. 
 To this end, for each $h>0$ and fixed $t\ge 0$, we define
 \begin{align}
 \psi_h^t(s)=\begin{cases} 1, &\quad \text{if}~ s\le t, \notag \\
 1-\frac{s-t}{h}, &\quad \text{if}~~t\le s\le t+h,\notag \\
 0, & \quad \text{if} ~ s \ge t+h.
 \end{cases}
 \end{align}
Furthermore, let $\phi \in C^2(\R)$ be a cut-off function such that $|\phi'(x)| \le C \phi(x),~~ | \phi^{\prime\prime}(x)|\le C \phi(x)$.
 Clearly, $\psi(s,y)=\psi_h^t(s)\phi(y)$ is an admissible test-function and 
\begin{align}
 & \frac{1}{h}\int_{s=t}^{t+h}  \E \Big[\int_{\R^2} \big| u_{\Delta y}(s,y) -u(s,x)\big|\phi(y)\varrho_\delta(x-y)\,dx\,dy \Big]\,ds \notag \\
 &\qquad \le \E\Big[\int_{\R}\int_{\R} \big|u_0(x)-u_{\Delta y}(0,y)\big| \varrho_{\delta}(x-y)\phi(y)\,dy\,dx\Big] \notag \\
& \quad + C\, \E\Big[\int_0^{t+h}\int_{\R^2} \big|u(s,x)-u_{\Delta y}(s,y)\big| \phi(y) \varrho_{\delta}(x-y) \psi_h^t(s)\,dx\,dy\,ds\Big]  + C \frac{\Delta y}{\delta}.
\end{align}
Let $\mathbb{T}$ be the set all points $t$ in $[0, \infty)$ such that $t$ is  right Lebesgue point of 
 $$B(s)= \E \Big[\int_{\R^2} \big|u_{\Delta y}(s,y)-u(s,x)\big|\phi(y) \varrho_\delta(x-y)\,dx\,dy\Big].$$
 Clearly, $\mathbb{T}^{\complement}$ has zero Lebesgue measure. Fix  $t\in \mathbb{T}$. Then, passing to the limit as $h\goto 0$ yields
\begin{align*}
 & \E \Big[\int_{\R^2} \big| u_{\Delta y}(t,y) -u(t,x)\big|\phi(y)\varrho_\delta(x-y)\,dx\,dy \,ds\Big] \notag \\
&\qquad \qquad  \le \E\Big[\int_{\R}\int_{\R} \big|u_0(x)-u_{\Delta y}(0,y)\big| \varrho_{\delta}(x-y)\phi(y)\,dy\,dx\Big] \notag \\
& \quad \qquad \qquad + C\, \E\Big[\int_0^{t}\int_{\R^2} \big|u(s,x)-u_{\Delta y}(s,y)\big| \phi(y) \varrho_{\delta}(x-y) \,dx\,dy\,ds\Big]  + C\frac{\Delta y}{\delta}.
\end{align*}
 An application of Gronwall's lemma gives, for a.e. $t>0$
 \begin{align}
  \E\Big[\int_{\R^2} \big| u_{\Delta y}(t,y) & - u(t,x)\big|\phi(y)\varrho_\delta(x-y)\,dx\,dy \Big] \notag \\
  & \le   
 \Bigg( \E\Big[\int_{\R}\int_{\R} \big|u_0(x)-u_{\Delta y}(0,y)\big| \varrho_{\delta}(x-y)\phi(y)\,dy\,dx\Big] + C \frac{\Delta y}{\delta}\Bigg)e^{Ct}
.\label{esti:2}
 \end{align}
 Sending $\phi$ to $\mathds{1}_{\R}$ in \eqref{esti:2} along with the fact
 \begin{align*}
 \E\Big[\int_{\R} \big|u_{\Delta y}(t,y)-u(t,y)\big|\,dy\Big] \le  \E\Big[\int_{\R^2} \big| u_{\Delta y}(t,y) -u(t,x)\big|\varrho_\delta(x-y)\,dx\,dy \Big] + \delta \E\big[|u_0|_{BV(\R)}\big],\\
  \E\Big[\int_{\R^2} \big| u_{\Delta y}(0,y) -u_0(x)\big|\varrho_\delta(x-y)\,dx\,dy \Big] \le  \E\Big[\int_{\R} \big|u_{\Delta y}(0,y)-u_0(y)\big|\,dy\Big] +  \delta \E\big[|u_0|_{BV(\R)}\big],
 \end{align*}
 and the initial error bound \eqref{error:initial-data},  we obtain 
 \begin{align}
\E\Big[\int_{\R} \big|u_{\Delta y}(t,y)-u(t,y)\big|\,dy\Big] 
    \le C_T \big(\Delta y + \delta + C \frac{\Delta y}{\delta} \Big),\label{esti:3} 
 \end{align}
where $C_T$ is a positive constant independent of the small parameters. Minimizing the above inequality \eqref{esti:3} with respect to $\delta$ yields the estimate 
 \begin{align*}
 \E\Big[\int_{\R} \big|u_{\Delta y}(t,y)-u(t,y)\big|\,dy\Big]  \le C\sqrt{\Delta y}.
 \end{align*}
 This finishes the proof of the main theorem.
%%%%%%%%%%%%%%%%%%%%%%%%%%%%%%%%%%%%%%%%%%%%%%%%%%%%%%%%%%%%%%%%%%%%%%%%%%%%%%%%

\section{Appendix}
\label{sec:app}
Here we study the semi-discrete finite difference scheme \eqref{eq:levy_stochconservation_ discrete_laws_1}  approximating \eqref{eq:stoc_con_laws} and show the convergence of approximate solutions. As we have mentioned earlier, Rohde et al. \cite{kroker} have studied semi-discrete finite volume scheme for the underlying problem \eqref{eq:stoc_con_laws} with $\eta=0$ and invoked stochastic 
compensated compactness method to show the convergence of approximate solutions to the unique entropy solution. Here, we consider the problem \eqref{eq:stoc_con_laws} and study convergence of approximate finite difference solutions via Young measure technique. 
\subsection{Proof of Theorem~\ref{thm:con}}
To begin with, let $ 0\le \phi(t,x) \in C_c^\infty([0,T)\times \R)$ and $(\beta,\zeta)$ be a given convex entropy-entropy 
flux pair with $\beta,\,\beta^\prime,\,\beta^{\prime \prime}$ and  $ \beta^{\prime \prime \prime}$ having at most polynomial growth. Let us also denote $\phi_j(t)= \phi(t,x_j)$.
In view of Lemma \eqref{lem:cellentropyinequality_finitevolumescheme}, a simple application of It\^{o}-L\'{e}vy product rule applied to $\beta(u_j(t))\phi_j(t)$ gives
\begin{align}
\label{eq:1}
 0\le & \beta(u_j(0))\phi_j(0) + \int_0^T \partial_t \phi_j(t)\beta(u_j(t))\,dt - \frac{1}{\Delta x} \int_0^T \phi_j(t)
 \Big(\zeta(u_{j+1}(t)) -\zeta(u_{j}(t))\Big)\,dt  \notag \\
 & \quad +  \int_0^T \sigma(u_j(t))\beta^{\prime}(u_j(t))\phi_j(t)\,dW(t) + \frac{1}{2}  \int_0^T \sigma^2(u_j(t))\beta^{\prime\prime}(u_j(t))\phi_j(t)\,dt \notag \\
 & \quad + \int_0^T \int_{|z|>0} \Big\{ \beta \big(u_j(t) + \eta(u_j(t);z)\big)-\beta(u_j(t))\Big\} \, \phi_j(t)\, \widetilde{N}(dz,dt) \notag \\
 & \qquad + \int_0^T \int_{|z|>0} \int_{0}^1 (1-\lambda)  \, \eta^2(u_j(t);z) \, \beta^{\prime\prime}\big(u_j(t) +
 \lambda\, \eta(u_j(t);z)\big) \, \phi_j(t)\,d\lambda\,m(dz)\,dt.
\end{align} 
Next, multiplying \eqref{eq:1} by $\Delta x$ and summing the resulting inequality over all $j\in \mathbb{Z}$, we obtain
\begin{align}
 0\le & {\Delta x} \sum_{j\in \mathbb{Z}}\beta(u_j(0))\phi_j(0) + {\Delta x} \sum_{j\in \mathbb{Z}}\int_0^T \partial_t \phi_j(t)\beta(u_j(t)) \,dt
 -{\Delta x} \sum_{j\in \mathbb{Z}}  \int_0^T \phi_j(t)
\frac{\zeta(u_{j+1}(t)) -\zeta(u_{j}(t))}{\Delta x} \,dt  \notag  \\ 
& \quad + {\Delta x} \sum_{j\in \mathbb{Z}} \int_0^T \sigma(u_j(t))\beta^{\prime}(u_j(t))\phi_j(t)\,dW(t) +  
{\Delta x} \sum_{j\in \mathbb{Z}}\frac{1}{2}  \int_0^T \sigma^2(u_j(t))\beta^{\prime\prime}(u_j(t))\phi_j(t)\,dt \notag \\
 & \qquad \quad  + {\Delta x} \sum_{j\in \mathbb{Z}}\int_0^T \int_{|z|>0} \Big\{ \beta \big(u_j(t) + \eta(u_j(t);z)\big)-\beta(u_j(t))\Big\} \, \phi_j(t)\, \widetilde{N}(dz,dt)  \label{eq:2} \\
 & + {\Delta x} \sum_{j\in \mathbb{Z}}\int_0^T \int_{|z|>0} \int_{0}^1 (1-\lambda) \, \eta^2(u_j(t);z)\,\beta^{\prime\prime}\big(u_j(t) + \lambda\, \eta(u_j(t);z)\big)
 \, \phi_j(t)\,d\lambda\,m(dz)\,dt. \notag
\end{align}
Let $B$ be an arbitrary set from $\mathcal{F}_T$. Taking expectation in \eqref{eq:2} gives us
\begin{align*}
& 0\le  \E\Big[\mathds{1}_{B}~{\Delta x} \sum_{j\in \mathbb{Z}}\beta(u_j(0))\phi_j(0)\Big]
 + \E\Big[ \mathds{1}_{B}~{\Delta x} \sum_{j\in \mathbb{Z}}\int_0^T \partial_t \phi_j(t)\beta(u_j(t)) \,dt\Big] \notag \\
 & \hspace{7cm} - \E\Big[\mathds{1}_{B}~{\Delta x} \sum_{j\in \mathbb{Z}}  \int_0^T \phi_j(t)
\frac{\zeta(u_{j+1}(t))-\zeta(u_{j}(t))}{\Delta x} \,dt\Big]  \notag \\
 & + \E\Big[\mathds{1}_{B}~{\Delta x} \sum_{j\in \mathbb{Z}}\int_0^T \int_{|z|>0} \Big\{ \beta \big(u_j(t) + \eta(u_j(t);z)\big)-\beta(u_j(t))\Big\} \, \phi_j(t)\,\widetilde{N}(dz,dt)\Big] \notag \\
 & + \E\Big[\mathds{1}_{B}~{\Delta x} \sum_{j\in \mathbb{Z}}\int_0^T \int_{|z|>0} \int_{0}^1 (1-\lambda) \,
 \eta^2(u_j(t);z) \, \beta^{\prime\prime} \big(u_j(t) + \lambda\, \eta(u_j(t);z)\big) \, \phi_j(t)\,d\lambda\,m(dz)\,dt \Big] \\
 & + \E\Big[\mathds{1}_{B}~{\Delta x} \sum_{j\in \mathbb{Z}}\int_0^T  \sigma(u_j(t))\beta^{\prime}(u_j(t))\phi_j(t)\,dW(t) \Big] + \frac{1}{2} \E\Big[\mathds{1}_{B}~{\Delta x} \sum_{j\in \mathbb{Z}}\int_0^T  \sigma^2(u_j(t))\beta^{\prime\prime}(u_j(t))\phi_j(t)\,dt \Big] \notag \\
 & := \E\Big[\mathds{1}_{B}~{\Delta x} \sum_{j\in \mathbb{Z}}\beta(u_j(0))\phi_j(0)\Big]
 + \E\Big[ \mathds{1}_{B}~{\Delta x} \sum_{j\in \mathbb{Z}}\int_0^T \partial_t \phi_j(t)\beta(u_j(t)) \,dt\Big] \\
 & \hspace{7cm}+ \mathcal{E}^1_{\Delta x} + \mathcal{E}^2_{\Delta x}+ \mathcal{E}^3_{\Delta x} + \mathcal{E}^4_{\Delta x} + \mathcal{E}^5_{\Delta x}. 
\end{align*}
We estimate each of the above terms. To begin with, we first estimate the first term. Since $\beta$ is convex, in view of Jensen's inequality one has
\begin{align}
  {\Delta x} \sum_{j\in \mathbb{Z}} & \beta(u_j(0))\phi_j(0)  = {\Delta x} \sum_{j\in \mathbb{Z}}\beta\Big( \frac{1}{\Delta x} \int_{x_{j-\frac{1}{2}}}^{x_{j+\frac{1}{2}}} u_0(x)\,dx\Big)\phi(0,x_j) \notag \\
 & \le \sum_{j\in \mathbb{Z}} \int_{x_{j-\frac{1}{2}}}^{x_{j+\frac{1}{2}}}\beta( u_0(x))\,\phi(0,x_j)\,dx \notag \\
 & = \sum_{j\in \mathbb{Z}} \int_{x_{j-\frac{1}{2}}}^{x_{j+\frac{1}{2}}}\beta( u_0(x))\,\phi(0,x)\,dx + 
 \sum_{j\in \mathbb{Z}} \int_{x_{j-\frac{1}{2}}}^{x_{j+\frac{1}{2}}}\beta( u_0(x))\,\Big(\phi(0,x_j)-\phi(0,x)\Big)\,dx  \notag \\
 & \le \int_{\R}\beta( u_0(x))\,\phi(0,x)\,dx + 
  {\Delta x} ||\phi_x(0,\cdot)||_{\infty} \int_{K}\beta( u_0(x))\,dx, \notag
\end{align} 
where $K$ is the compact support of the test function $\phi$ in $\R$. Since $B \in \mathcal{F}_T$ and $\beta$ has at most polynomial growth, we obtain
\begin{align}
 \E\Big[\mathds{1}_{B}~{\Delta x} \sum_{j\in \mathbb{Z}}\beta(u_j(0))\phi_j(0)\Big] \le 
 \E\Big[\mathds{1}_{B} \int_{\R}\beta(u_0(x))\phi(0,x)\,dx\Big] + \mathcal{O}({\Delta x}). \label{estimate:ini-discreate}
\end{align}
Next, we move on to estimate the second term. Note that $\phi(t,x) \in C_c^\infty([0,\infty)\times \R)$. Therefore
\begin{align}
& \E\Big[ \mathds{1}_{B} {\Delta x} \sum_{j\in \mathbb{Z}}\int_0^T \partial_t \phi_j(t)\beta(u_j(t)) \,dt\Big] =  \E\Big[ \mathds{1}_{B}  \int_0^T \sum_{j\in \mathbb{Z}}\int_{x_{j-\frac{1}{2}}}^{x_{j+\frac{1}{2}}} 
 \partial_t \phi(t,x_j)\beta(u_j(t))\,dx \,dt\Big] \notag \\
 & =  \E\Big[ \mathds{1}_{B}  \int_0^T \sum_{j\in \mathbb{Z}}\int_{x_{j-\frac{1}{2}}}^{x_{j+\frac{1}{2}}} 
 \partial_t \phi(t,x_j)\beta(u_{\Delta x}(t,x))\,dx \,dt\Big] \notag \\
  & =  \E\Big[ \mathds{1}_{B}  \Big\{\int_0^T \sum_{j\in \mathbb{Z}}\int_{x_{j-\frac{1}{2}}}^{x_{j+\frac{1}{2}}} 
 \partial_t \phi(t,x)\beta(u_{\Delta x}(t,x))\,dx \,dt  \notag \\
   & \hspace{5cm} +   \int_0^T \sum_{j\in \mathbb{Z}}\int_{x_{j-\frac{1}{2}}}^{x_{j+\frac{1}{2}}} 
 \Big(\partial_t \phi(t,x_j)-\partial_t\phi(t,x)\Big)\beta(u_{\Delta x}(t,x))\,dx \,dt\Big\}\Big] \notag \\
 & \quad \le    \E\Big[ \mathds{1}_{B} \int_0^T \int_{\R}\partial_t \phi(t,x)\beta(u_{\Delta x}(t,x))\,dx \,dt \Big]
    + {\Delta x}\, ||\phi_{tx}(t,\cdot)||_{\infty}  \E\Big[  \int_0^T \int_{K}\beta(u_{\Delta x}(t,x))\,dx \,dt\Big], \notag 
\end{align}  
where $K$ is the compact support of the test function $\phi$ in $\R$. Now we can  invoke  uniform moment estimate \eqref{uni:moment} and
the polynomial growth condition of $\beta$ to conclude
\begin{align}
  \E\Big[ \mathds{1}_{B}~{\Delta x} \sum_{j\in \mathbb{Z}}\int_0^T \partial_t \phi_j(t)\beta(u_j(t)) \,dt\Big]  \le 
   \E\Big[ \mathds{1}_{B} \int_0^T \int_{\R}\partial_t \phi(t,x)\beta(u_{\Delta x}(t,x))\,dx \,dt \Big] + \mathcal{O}({\Delta x}).
 \label{estimate:time derivative-discreate}
\end{align}
Let us focus on the third term $\mathcal{E}^1_{\Delta x}$. Notice that
\begin{align}
 &-{\Delta x} \sum_{j\in \mathbb{Z}}  \int_0^T \phi_j(t) \frac{ \zeta(u_{j+1} (t))-\zeta(u_{j}(t))}{\Delta x} \,dt
 = {\Delta x} \sum_{j\in \mathbb{Z}}  \int_0^T \frac{\phi_{j}(t) -\phi_{j-1} (t)}{\Delta x} \zeta(u_{j}(t))\,dt.   \label{expansion:entropy monotone flux}
\end{align} 
Since $\phi \in C_c^\infty([0,\infty)\times \R)$ and $\beta$ has a polynomial growth condition,
one can argue along the same line of previous argument and conclude
\begin{align}
\E\Big[\mathds{1}_{B} {\Delta x} \sum_{j\in \mathbb{Z}}  \int_0^T \frac{\phi_{j}(t) -\phi_{j-1}(t)}{\Delta x} \zeta(u_j(t)) \,dt\Big]
\le \E\Big[\mathds{1}_{B}   \int_0^T  \int_{\R}\partial_x \phi(t,x) \zeta(u_{\Delta x}(t,x)) \,dx\,dt\Big] + \mathcal{O}({\Delta x}).
\label{estimate:entropy monotone flux}
\end{align}
Hence, we conclude that 
\begin{align}
   - \E\Big[\mathds{1}_{B}~{\Delta x} \sum_{j\in \mathbb{Z}} & \int_0^T  \phi_j(t)
\frac{\zeta(u_{j+1} (t))-\zeta(u_{j}(t))}{\Delta x} \,dt\Big] \notag \\
&\qquad \qquad  \le 
 \E\Big[\mathds{1}_{B}\int_0^T  \int_{\R}\partial_x \phi(t,x) \zeta(u_{\Delta x}(t,x)) \,dx\,dt\Big] + \mathcal{O}({\Delta x}). 
 \label{estimate:spatial derivative-discreate}
\end{align}
Next, we consider the stochastic term $\mathcal{E}^2_{\Delta x}$. Let $B \in \mathcal{F}_T$ be any measurable set, then
\begin{align}
 & \E\Big[\mathds{1}_{B}~{\Delta x} \sum_{j\in \mathbb{Z}}\int_0^T \int_{|z|>0} \Big\{ \beta \big(u_j(t) + \eta(u_j(t);z)\big)-\beta(u_j(t))\Big\}
 \phi_j(t) \widetilde{N}(dz,dt)\Big] \notag \\
 = & \E\Big[\mathds{1}_{B} \sum_{j\in \mathbb{Z}} \int_{x_{j-\frac{1}{2}}}^{x_{j+\frac{1}{2}}} \int_0^T \int_{|z|>0} \Big\{ \beta \big(u_j(t) + \eta(u_j(t);z)\big)-\beta(u_j(t))\Big\}
 \phi_j(t)\, \widetilde{N}(dz,dt)\,dx\Big] \notag \\
  =& \E\Big[\mathds{1}_{B} \sum_{j\in \mathbb{Z}} \int_{x_{j-\frac{1}{2}}}^{x_{j+\frac{1}{2}}} \int_0^T \int_{|z|>0} \Big\{ \beta \big(u_{\Delta x}(t,x) + \eta(u_{\Delta x}(t,x);z)\big)-\beta(u_{\Delta x}(t,x))\Big\}
 \phi_j(t)\, \widetilde{N}(dz,dt)\,dx\Big] \notag \\
  = & \E\Big[\mathds{1}_{B}  \int_0^T \int_{|z|>0} \int_{\R} \Big\{ \beta \big(u_{\Delta x}(t,x) + \eta(u_{\Delta x}(t,x);z)\big)-\beta(u_{\Delta x}(t,x))\Big\}
 \phi(t,x) \,dx\, \widetilde{N}(dz,dt)\Big] \notag \\
  &\quad +  \E\Big[\mathds{1}_{B} \sum_{j\in \mathbb{Z}} \int_{x_{j-\frac{1}{2}}}^{x_{j+\frac{1}{2}}} \int_0^T \int_{|z|>0}
 \Big\{ \beta \big(u_{\Delta x}(t,x) + \eta(u_{\Delta x}(t,x);z)\big)-\beta(u_{\Delta x}(t,x))\Big\} \notag \\
 &\hspace{9cm} \times \big( \phi(t,x_j)-\phi(t,x) \big)\,\widetilde{N}(dz,dt) \,dx\Big] \notag \\
    &:= \E\Big[\mathds{1}_{B}  \int_0^T \int_{|z|>0} \int_{\R} \Big\{ \beta \big(u_{\Delta x}(t,x) + \eta(u_{\Delta x}(t,x);z)\big)-\beta(u_{\Delta x}(t,x))\Big\}
 \phi(t,x) \,dx\, \widetilde{N}(dz,dt)\Big] \notag \\
 & \hspace{ 4cm} + \mathcal{E}^{2,1}_{\Delta x}.\label{expansion: stochastic term}
\end{align} 
To proceed further, we consider the term $\mathcal{E}^{2,1}_{\Delta x}$. In fact, in view of assumption \ref{A4} and \ref{A5}, Cauchy-Schwartz inequality, BDG inequality,
Jensen's inequality, and  polynomial growth condition on  $\beta^\prime$ along with uniform moment estimate \eqref{uni:moment}, 
we have
\begin{align}
\label{estimate:A_1}
 \mathcal{E}^{2,1}_{\Delta x}  \le &  \Bigg( \E\Big[\Big( \sum_{j\in \mathbb{Z}} \int_{x_{j-\frac{1}{2}}}^{x_{j+\frac{1}{2}}} \int_0^T \int_{|z|>0}
 \Big\{ \beta \big(u_{\Delta x}(t,x) + \eta(u_{\Delta x}(t,x);z)\big)-\beta(u_{\Delta x}(t,x))\Big\}  \\
 &\hspace{7cm} \times \big( \phi(t,x_j)-\phi(t,x) \big)\,\widetilde{N}(dz,dt) \,dx \Big)^2\Big]\Bigg)^\frac{1}{2} \notag \\
  \le & \Bigg( \E\Big[\int_0^T \int_{|z|>0} \Big\{ \sum_{j\in \mathbb{Z}} \int_{x_{j-\frac{1}{2}}}^{x_{j+\frac{1}{2}}}
 \int_{0}^1 \eta(u_{\Delta x}(t,x);z) \beta^\prime \big(u_{\Delta x}(t,x)+\lambda\,\eta(u_{\Delta x}(t,x);z)\big)  \notag \\
 & \hspace{7cm} \times \big( \phi(t,x_j)-\phi(t,x) \big)\,d\lambda
 \,dx \Big\}^2 \,m(dz)\,dt \Big]\Bigg)^\frac{1}{2} \notag \\
  \le & C {\Delta x} || \phi_x(t,\cdot)||_{\infty} \Bigg( \E\Big[\int_0^T \int_{|z|>0}\int_{K}
 \int_{0}^1 {\beta^\prime}^2 \big(u_{\Delta x}(t,x)+\lambda\,\eta(u_{\Delta x}(t,x);z)\big)  \notag \\
 & \hspace{8cm} \times  \eta^2(u_{\Delta x}(t,x);z) \,d\lambda\,dx  \,m(dz)\,dt \Big]\Bigg)^\frac{1}{2} \notag \\
   \le &  C {\Delta x} || \phi_x(t,\cdot)||_{\infty} \Bigg( \E\Big[\int_0^T \int_{|z|>0}  \int_{K}\big(1+ |u_{\Delta x}(t,x)|^2 \big)
  \big( |u_{\Delta x}(t,x)|^{2p} + |\eta(u_{\Delta x}(t,x);z)|^{2p}\big)    \notag \\
   & \hspace{4cm} \times (1 \wedge |z|^2) \,dx  \,m(dz)\,dt \Big]\Bigg)^\frac{1}{2} \notag \\
   \le & C {\Delta x} || \phi_x(t,\cdot)||_{\infty} \Bigg( \E\Big[\int_0^T \int_{|z|>0}  \int_{K} \Big( 1 + |u_{\Delta x}(t,x)|^{2p} + 
  |u_{\Delta x}(t,x)|^{2p+2} + |u_{\Delta x}(t,x)|^{2} \Big)  \notag \\
  &  \hspace{5cm}\times(1 \wedge |z|^2) \,dx  \,m(dz)\,dt \Big]\Bigg)^\frac{1}{2} \notag \\
  \le & C_1 {\Delta x}. \notag 
\end{align} 
Now we combine  \eqref{expansion: stochastic term} and \eqref{estimate:A_1} and conclude that for any measurable set
 $ B\in \mathcal{F}_T,$
\begin{align}
   \E\Big[&\mathds{1}_{B}~{\Delta x} \sum_{j\in \mathbb{Z}}\int_0^T \int_{|z|>0} \Big\{ \beta \big(u_j(t) + \eta(u_j(t);z)\big)-\beta(u_j(t))\Big\}
 \phi_j(t)\, \widetilde{N}(dz,dt)\Big] \notag \\
 \le &  \E\Big[\mathds{1}_{B}\int_0^T \int_{|z|>0} \int_{\R} \Big\{ \beta \big(u_{\Delta x}(t,x) + \eta(u_{\Delta x}(t,x);z)\big)-\beta(u_{\Delta x}(t,x))\Big\} 
  \phi(t,x) \,dx\, \widetilde{N}(dz,dt)\Big] + \mathcal{O}({\Delta x}).
 \label{estimate: stochastic term-discreate}
\end{align}
Let us consider the term $\mathcal{E}^3_{\Delta x}$. Notice that
\begin{align}
  &\E\Big[\mathds{1}_{B}~{\Delta x} \sum_{j\in \mathbb{Z}}\int_0^T \int_{|z|>0} \int_{0}^1 (1-\lambda) 
 \eta^2(u_j(t);z)\beta^{\prime\prime} \big(u_j(t) + \lambda\, \eta(u_j(t);z)\big) \phi_j(t)\,d\lambda\,m(dz)\,dt \Big] \notag \\
 =&  \E\Big[\mathds{1}_{B} \sum_{j\in \mathbb{Z}}\int_{x_{j-\frac{1}{2}}}^{x_{j+\frac{1}{2}}} \int_0^T \int_{|z|>0} \int_{0}^1 (1-\lambda) 
 \eta^2(u_j(t);z)\beta^{\prime\prime} \big(u_j(t) + \lambda\, \eta(u_j(t);z)\big)
 \phi(t,x_j)\,d\lambda\,m(dz)\,dt\,dx \Big] \notag \\
 = &  \E\Big[\mathds{1}_{B} \sum_{j\in \mathbb{Z}}\int_{x_{j-\frac{1}{2}}}^{x_{j+\frac{1}{2}}} \int_0^T \int_{|z|>0} \int_{0}^1 (1-\lambda) 
 \eta^2(u_{\Delta x}(t,x);z)\beta^{\prime\prime} \big(u_{\Delta x}(t,x) + \lambda\, \eta(u_{\Delta x}(t,x);z)\big) \notag \\
 & \hspace{8cm} \times 
 \phi(t,x_j)\,d\lambda\,m(dz)\,dt\,dx \Big] \notag \\
 =&  \E\Big[\mathds{1}_{B}  \int_0^T \int_{|z|>0} \int_{\R} \int_{0}^1 (1-\lambda) 
 \eta^2(u_{\Delta x}(t,x);z)\beta^{\prime\prime} \big(u_{\Delta x}(t,x) + \lambda\, \eta(u_{\Delta x}(t,x);z)\big) \notag \\
 & \hspace{8cm} \times 
 \phi(t,x)\,d\lambda\,m(dz)\,dt\,dx \Big] \notag \\
 & +  \E\Big[\mathds{1}_{B} \sum_{j\in \mathbb{Z}}\int_{x_{j-\frac{1}{2}}}^{x_{j+\frac{1}{2}}} \int_0^T \int_{|z|>0} \int_{0}^1 (1-\lambda) 
 \eta^2(u_{\Delta x}(t,x);z)\beta^{\prime\prime} \big(u_{\Delta x}(t,x) + \lambda\, \eta(u_{\Delta x}(t,x);z)\big)   \notag \\
& \hspace{5cm} \times \big( \phi(t,x_j)-\phi(t,x)\big)\,d\lambda\,m(dz)\,dt\,dx \Big] \notag \\
 & :=  \E\Big[\mathds{1}_{B}  \int_0^T \int_{|z|>0} \int_{\R} \int_{0}^1 (1-\lambda) 
 \eta^2(u_{\Delta x}(t,x);z)\beta^{\prime\prime} \big(u_{\Delta x}(t,x) + \lambda \eta(u_{\Delta x}(t,x);z)\big)  \notag \\
&  \hspace{8.5cm}  \times \phi(t,x)\,d\lambda\,dx\,m(dz)\,dt \Big] 
+\mathcal{E}^{3,1}_{\Delta x}.
\label{expansion:correction term}
\end{align}
We would like to estimate $\mathcal{E}^{3,1}_{\Delta x}$.
Thanks to the polynomial growth of $\beta^{\prime \prime}$ and the property of $\eta$, we have 
  \begin{align}
&\mathcal{E}^{3,1}_{\Delta x}  \le C \Bigg( \E\Bigg[ \Bigg( \sum_{j\in \mathbb{Z}}\int_{x_{j-\frac{1}{2}}}^{x_{j+\frac{1}{2}}} \int_0^T \int_{|z|>0}  
 \big(1+ |u_{\Delta x}(t,x)|^2 \big)\big(|u_{\Delta x}(t,x)|^p + (1+|u_{\Delta x}(t,x)|^p)\big) \notag \\
 & \hspace{6cm} \times (1\wedge |z|^2) |\phi_x(t,\xi)| |x-x_j|\,m(dz)\,dt\,dx \Bigg)^2\Bigg]\Bigg)^\frac{1}{2} \notag \\
 & \le C \Bigg( \E\Bigg[ \Bigg( \sum_{j\in \mathbb{Z}}\int_{x_{j-\frac{1}{2}}}^{x_{j+\frac{1}{2}}} \int_0^T  
  \Big(1+ |u_{\Delta x}(t,x)|^2 + |u_{\Delta x}(t,x)|^p + |u_{\Delta x}(t,x)|^{p+2}\Big) 
  |\phi_x(t,\xi)| |x-x_j|\,dt\,dx \Bigg)^2\Bigg]\Bigg)^\frac{1}{2}\notag \\
 & \le C\,{\Delta x}\, ||\phi_x||_{\infty}  \Bigg( \E\Bigg[ \int_{K} \int_0^T 
  \Big(1+ |u_{\Delta x}(t,x)|^{2p} +|u_{\Delta x}(t,x)|^{2(p+2)} +|u_{\Delta x}(t,x)|^4 \Big)\,dt\,dx \Bigg]\Bigg)^\frac{1}{2} \notag \\
  & \le  C_2 {\Delta x},  \label{estimate:B_1}
  \end{align}  
where $ \xi \in (x,x_j)$. In the above, we have used uniform moment estimate \eqref{uni:moment} along with the condition that $\phi$ has compact support.
We use the estimate \eqref{estimate:B_1} in \eqref{expansion:correction term} and  have
 \begin{align}
   &\E\Big[\mathds{1}_{B}~{\Delta x} \sum_{j\in \mathbb{Z}}\int_0^T \int_{|z|>0} \int_{0}^1 (1-\lambda) 
 \eta^2(u_j(t);z)\beta^{\prime\prime}\big(u_j(t) + \lambda\, \eta(u_j(t);z)\big) \phi_j(t)\,d\lambda\,m(dz)\,dt \Big] \notag \\
 \le &  \E\Big[\mathds{1}_{B}  \int_{\Pi_T}\int_{|z|>0}  \int_{0}^1 (1-\lambda) 
 \eta^2(u_{\Delta x}(t,x),z)\beta^{\prime\prime}\big(u_{\Delta x}(t,x) + \lambda\, \eta(u_{\Delta x}(t,x);z)\big)
  \phi(t,x)\,d\lambda\,m(dz)\,dt\,dx \Big] \notag \\
 & \hspace{5cm}+ \mathcal{O}({\Delta x}). 
 \label{estimate: correction term-discreate}
 \end{align}
 Again, observe that 
 \begin{align*}
  \mathcal{E}^{4}_{\Delta x} &= \E\Big[\mathds{1}_{B}\int_0^T  \int_{\R} \sigma(u_{\Delta x}(t,x)) \beta^{\prime}(u_{\Delta x}(t,x)) \phi(t,x)\,dx\,dW(t) \Big] \notag \\
  & +  \E\Big[\mathds{1}_{B}  \sum_{j\in \mathbb{Z}}\int_{x_{j-\frac{1}{2}}}^{x_{j+\frac{1}{2}}} \int_0^T  \sigma(u_{\Delta x}(t,x)) \beta^{\prime}(u_{\Delta x}(t,x))
  \big( \phi(t,x_j)-\phi(t,x)\big)\,dx\,dW(t)\Big] \notag \\
  & = \E\Big[\mathds{1}_{B}\int_0^T  \int_{\R} \sigma(u_{\Delta x}(t,x)) \beta^{\prime}(u_{\Delta x}(t,x)) \phi(t,x)\,dx\,dW(t) \Big] + \mathcal{E}^{4,1}_{\Delta x}.  
 \end{align*}
 One can argue similary as in the estimation of the term $\mathcal{E}^{2,1}_{\Delta x}$ and have $\mathcal{E}^{4,1}_{\Delta x}\le C \Delta x$. Thus, we conclude that 
 \begin{align}
  \mathcal{E}^{4}_{\Delta x} & \le  \E\Big[\mathds{1}_{B}\int_0^T  \int_{\R} \sigma(u_{\Delta x}(t,x)) \beta^{\prime}(u_{\Delta x}(t,x)) \phi(t,x)\,dx\,dW(t) \Big]
  + C\, \Delta x. \label{estimate:ito-brown-term-discreate}
 \end{align}
 Finally, we consider the expression  $\mathcal{E}^{5}_{\Delta x}$. We can re-wrire this term as 
 \begin{align}
  \mathcal{E}^{5}_{\Delta x} & = \frac{1}{2} \E\Big[\mathds{1}_{B}\int_0^T \int_{\R} \sigma^2(u_{\Delta x}(t,x)) \beta^{\prime\prime}(u_{\Delta x}(t,x)) \phi(t,x)\,dx\,dt \Big] \notag \\
  &  \qquad +  \frac{1}{2} \E\Big[\mathds{1}_{B}  \sum_{j\in \mathbb{Z}}\int_{x_{j-\frac{1}{2}}}^{x_{j+\frac{1}{2}}} \int_0^T  \sigma^2(u_{\Delta x}(t,x))
  \beta^{\prime\prime}(u_{\Delta x}(t,x))
  \big( \phi(t,x_j)-\phi(t,x)\big)\,dt\,dx\Big] \notag \\
  & := \frac{1}{2} \E\Big[\mathds{1}_{B}\int_0^T \int_{\R} \sigma^2(u_{\Delta x}(t,x)) \beta^{\prime\prime}(u_{\Delta x}(t,x)) \phi(t,x)\,dx\,dt \Big]
  + \mathcal{E}^{5,1}_{\Delta x}. \label{expression:e5}
 \end{align}
 In view of polynomial growth condition of $\beta^{\prime\prime}$, assumption \ref{A3} and uniform moment estimate \eqref{uni:moment}, we arrive at
 (similar to the term $ \mathcal{E}^{3,1}_{\Delta x}$)
\begin{align}
  \mathcal{E}^{5,1}_{\Delta x} \le C\,\Delta x. \label{esti:e51}
\end{align}
Thus, combining \eqref{expression:e5} and \eqref{esti:e51}, we obtain 
\begin{align}
  \mathcal{E}^{5}_{\Delta x} \le \frac{1}{2} \E\Big[\mathds{1}_{B}\int_0^T \int_{\R} \sigma^2(u_{\Delta x}(t,x)) \beta^{\prime\prime}(u_{\Delta x}(t,x)) \phi(t,x)\,dx\,dt \Big]
   + C\, \Delta x. \label{estimate: correction term-discreate-Brownian}
\end{align}
 Now we are in a position to combine the estimates \eqref{estimate:ini-discreate},\eqref{estimate:time derivative-discreate}, \eqref{estimate:spatial derivative-discreate},
 \eqref{estimate: stochastic term-discreate}, \eqref{estimate: correction term-discreate}, \eqref{estimate:ito-brown-term-discreate} and \eqref{estimate: correction term-discreate-Brownian}
 to conclude that
 \begin{align}
   0\le &  \E\Big[\mathds{1}_{B} \int_{\R}\beta(u_0(x))\phi(0,x)\,dx\Big] + 
    \E\Big[ \mathds{1}_{B} \int_0^T \int_{\R}\partial_t \phi(t,x)\beta(u_{\Delta x}(t,x))\,dx \,dt \Big]  \notag \\
     &+  \E\Big[\mathds{1}_{B}\int_0^T  \int_{\R}\partial_x \phi(t,x) \zeta(u_{\Delta x}(t,x)) \,dx\,dt\Big]  \notag \\
      & + \E\Big[\mathds{1}_{B}\int_0^T  \int_{\R} \sigma(u_{\Delta x}(t,x)) \beta^{\prime}(u_{\Delta x}(t,x)) \phi(t,x)\,dx\,dW(t) \Big]\notag \\
      & + \frac{1}{2} \E\Big[\mathds{1}_{B}\int_0^T \int_{\R} \sigma^2(u_{\Delta x}(t,x)) \beta^{\prime\prime}(u_{\Delta x}(t,x)) \phi(t,x)\,dx\,dt \Big]\notag \\
     & + \E\Big[\mathds{1}_{B}\int_0^T \int_{|z|>0} \int_{\R} \Big\{ \beta\big (u_{\Delta x}(t,x) + \eta(u_{\Delta x}(t,x);z)\big)-\beta(u_{\Delta x}(t,x))\Big\}
 \phi(t,x) \,dx \,\widetilde{N}(dz,dt)\Big] \notag \\
 & + \E\Big[\mathds{1}_{B}  \int_0^T \int_{|z|>0} \int_{\R} \int_{0}^1 (1-\lambda) 
 \eta^2(u_{\Delta x}(t,x);z)\beta^{\prime\prime} \big(u_{\Delta x}(t,x) + \lambda\, \eta(u_{\Delta x}(t,x);z)\big) \notag \\
 &\hspace{6cm} \times \phi(t,x)\,d\lambda\,dx\,m(dz)\,dt \Big] + \mathcal{O}({\Delta x}). \label{eq:young-discreate-approximation}
 \end{align}
Finally, we would like to pass to the limit as ${\Delta x} \goto 0$. To do so,  we use the 
  technique of Young measure theory, see \cite{Balder}.  
  Let the predictable $\sigma$-field of $\Omega\times(0,T)$ with the respect to $\{\mathcal{F}_t\}$ is
  denoted by $\mathcal{P}_T$, and we set 
\begin{align*}
  \Theta = \Omega\times (0,T)\times \R,\quad \Sigma = \mathcal{P}_T \times \mathcal{L}(\R)\quad \text{and} \quad \mu= P\otimes \lambda_t\otimes \lambda_x,
\end{align*}
where $\lambda_t$ and $\lambda_x$ are respectively the Lebesgue measures on $(0,T)$ and $\R$.
Moreover, for $M\in \mathbb{N}$, set
  $ \Theta_M = \Omega\times (0,T)\times B_M,$ where $B_M$ be the ball of radius $M$ around zero in $\R$. The set of all Young measures from $\Theta$ into $\R$ is denoted by $\mathcal{R}(\Theta, \Sigma, \mu).$
  Here we sum up the necessary results in the following lemma to carry over the subsequent analysis. For a proof of this lemma, consult \cite{BaVaWit, BisKarlMaj}.
\begin{lem} 
\label{prop:young-measure}
  Let  $\{u_{\Delta x}(t,x)\}_{{\Delta x}> 0}$ be a sequence of  $L^p(\R)$-valued predictable processes such that
  \eqref{uni:moment} holds.Then there exists a subsequence $\{{\Delta x}_n\}$ with ${\Delta x}_n\goto 0$ and a Young measure
  $\nu\in \mathcal{R}(\Theta, \Sigma, \mu) $ such that the following hold:
  \begin{itemize}
   \item [$(A)$] If  $h(\theta,\xi)$ is a Caratheodory function on $\Theta\times \R$ such that $\mbox{supp}(h)\subset \Theta_M\times \R$
 for some $M \in \mathbb{N}$ and $\{h(\theta, u_{{\Delta x}_n}(\theta)\}_n$ (where $\theta:= (\omega; t, x)$) is uniformly 
 integrable,  then 
 \begin{align}
    \lim_{{\Delta x}_n\rightarrow 0} \int_{\Theta}h(\theta, u_{{\Delta x}_n}(\theta))\,\mu(d\theta) =   \int_\Theta\Big[\int_{\R} h(\theta, \xi)\nu(\theta)(\,d\xi)\Big]\,\mu(d \theta). \label{yong-weak-limit}
    \end{align}
\item [$(B)$] Denoting a triplet $(\omega, t,x)\in \Theta$ by $\theta$, we define 
\begin{align}
  u(\theta,\alpha)=\inf\Big\{ c\in \R: \nu(\theta)\Big((-\infty,c)\Big)>\alpha \Big\},\,\,\,\, \text{for}\,\,\,\, \alpha \in (0,1)~\text{and}~\theta\in \Theta.\label{definition:young measure-solution}
\end{align} 
 Then, the function $  u(\theta,\alpha)$ is  non-decreasing, right continuous on $(0,1)$ and  $ \mathcal{P}_T\times \mathcal{L}(\R\times (0,1))$- measurable.
  Moreover, if $h(\theta,\xi)$ is a nonnegative Caratheodory function on $\Theta\times \R$, then 
  \begin{align}
 \int_{\Theta}\Big[\int_\R h(\theta,\xi)\nu(\theta)(\,d\xi)\Big]\, \mu \,(d\theta)= \int_\Theta\int_{\alpha=0}^1 h(\theta, u(\theta,\alpha))
 \,d\alpha\, \mu(d\theta). \label{yong-weak-limit-1}
 \end{align}
\end{itemize}
\end{lem}

In view of the Lemma \ref{prop:young-measure}, one can conclude (see \cite[Corollary $4.6$]{BisKarlMaj}, and \cite[Lemma $4.7$]{BisMaj}) that
   for every $B\in \mathcal{F}_T$ 
\begin{align}
  &  \lim_{{\Delta x}_n\rightarrow 0}  \E\Big[\mathds{1}_{B} \int_0^T \int_{|z|>0}\int_{\R} \Big( \beta\big( u_{{\Delta x}_n}(t,x) + \eta( u_{{\Delta x}_n}(t,x); z)\big) -  \beta( u_{{\Delta x}_n}(t,x))\Big)
  \phi(t,x)\, dx\,\widetilde{N}(dz,dt) \Big] \notag\\
    = & \E\Big[\mathds{1}_{B} \int_0^T \int_{|z|>0}\int_{\R}\int_{0}^1 \Big(\beta \big( u(t,x,\alpha) + \eta(u(t,x,\alpha); z)\big)
    -  \beta( u(t,x,\alpha)\Big)\phi(t,x)\,d\alpha\, dx\,\widetilde{N}(dz,dt) \Big],
    \label{limit-stochastic term-approximation}
\end{align}
and 
\begin{align}
  &  \lim_{{\Delta x}_n\rightarrow 0}  \E\Big[\mathds{1}_{B} \int_0^T \int_{\R} \sigma(u_{{\Delta x}_n}(t,x)) \beta^{\prime}( u_{{\Delta x}_n}(t,x))
  \phi(t,x)\, dx\,dW(t)\Big] \notag\\
    &= \E\Big[\mathds{1}_{B} \int_0^T\int_{\R}\int_{0}^1 \sigma(u(t,x,\alpha)) \beta^\prime( u(t,x,\alpha))\,\phi(t,x)\,d\alpha\, dx\,dW(t) \Big].
    \label{limit-stochastic term-approximation-Brown}
\end{align}
Note that for any $B\in \mathcal{F}_T$, the
function ${\bf 1}_B \in L^2\big(0,T; L^2({(\Omega, \mathcal{F}_T)}, L^2(\R))\big)$, and  $ L^2(\Theta, \Sigma, \mu)$
is a closed subspace of  $L^2\big(0,T; L^2({(\Omega, \mathcal{F}_T)}, L^2(\R))\big)$. Hence the weak convergence
in $ L^2\big(\Theta, \Sigma, \mu\big)$ would imply weak convergence
in $L^2\big(0,T; L^2({(\Omega, \mathcal{F}_T)}, L^2(\R))\big)$. Therefore, in view of Lemma \ref{prop:young-measure},
 we have 
 \begin{align}
  & \lim_{{\Delta x}_n\rightarrow 0} \E\Big[ \mathds{1}_{B} \int_0^T \int_{\R} \Big(\partial_t \phi(t,x)\beta(u_{{\Delta x}_n}(t,x))+\partial_x \phi(t,x)
   \zeta(u_{{\Delta x}_n}(t,x)) \Big) \,dx\,dt\Big] \notag \\
   &\qquad \qquad = \E\Big[ \mathds{1}_{B} \int_0^T \int_{\R} \int_{0}^1 \Big(\partial_t \phi(t,x)\beta(u(t,x,\alpha))+\partial_x \phi(t,x)
   \zeta(u(t,x,\alpha)) \Big)\, d\alpha\,dx\,dt\Big], \label{limit-partial-approximation} \\
   & \lim_{{\Delta x}_n\rightarrow 0} \E\Big[\mathds{1}_{B}  \int_0^T \int_{|z|>0} \int_{\R} \int_{0}^1 (1-\lambda) 
 \eta^2(u_{{\Delta x}_n}(t,x);z)\beta^{\prime\prime} \big(u_{{\Delta x}_n}(t,x) + \lambda\,\eta(u_{{\Delta x}_n}(t,x);z)\big) \notag \\
 &\hspace{8cm} \times \phi(t,x)\,d\lambda\,dx\,m(dz)\,dt \Big] \notag \\
  & \qquad \qquad= \E\Big[\mathds{1}_{B}  \int_0^T \int_{|z|>0} \int_{\R} \int_{0}^1 \int_{0}^1 (1-\lambda) 
 \eta^2(u(t,x,\alpha);z)\beta^{\prime\prime} \big(u(t,x,\alpha) + \lambda\,\eta(u(t,x,\alpha);z)\big) \notag \\
 &\hspace{8cm} \times \phi(t,x)\,d\alpha\,d\lambda\,dx\,m(dz)\,dt \Big], \label{limit-correction term-approximation} \\
 & \text{and} \notag  \\
   & \lim_{{\Delta x}_n\rightarrow 0} \frac{1}{2} \E\Big[\mathds{1}_{B}  \int_0^T \int_{\R}
 \sigma^2(u_{{\Delta x}_n}(t,x))\beta^{\prime\prime}(u_{{\Delta x}_n}(t,x)) \phi(t,x)\,dx\,dt \Big] \notag \\
  & \qquad \qquad= \frac{1}{2} \E\Big[\mathds{1}_{B}  \int_0^T \int_{\R} \int_{0}^1
 \sigma^2(u(t,x,\alpha))\beta^{\prime\prime}(u(t,x,\alpha)) \phi(t,x)\,d\alpha\,dx\,dt \Big]. \label{limit-correction term-approximation-brown}
 \end{align}
 
 Now,  one can use \eqref{limit-partial-approximation}, \eqref{limit-correction term-approximation} and \eqref{limit-correction term-approximation-brown} along with  
\eqref{limit-stochastic term-approximation} and \eqref{limit-stochastic term-approximation-Brown} and pass to the limit as ${\Delta x}_n \goto 0$ in \eqref{eq:young-discreate-approximation},
and conclude that for any $0\le \phi\in C_c^\infty([0,\infty)\times \R)$ and given any  convex entropy-entropy flux pair $(\beta,\zeta)$ with $\beta,\beta
^\prime$, and  $ \beta^{\prime \prime}$ having  at most polynomial growth and for any $B\in \mathcal{F}_T$, the 
following inequality holds:
 \begin{align}
   0 \le &  \E\Big[\mathds{1}_{B} \int_{\R}\beta(u_0(x))\phi(0,x)\,dx\Big] + 
    \E\Big[ \mathds{1}_{B} \int_0^T \int_{\R} \int_{0}^1 \partial_t \phi(t,x)\beta(u(t,x,\alpha))\,d\alpha \,dx \,dt \Big]  \notag \\
     &+  \E\Big[\mathds{1}_{B}\int_0^T  \int_{\R} \int_{0}^1 \partial_x \phi(t,x) \zeta(u(t,x,\alpha))\, d\alpha\,dx\,dt\Big] \notag \\
     & + \E\Big[\mathds{1}_{B} \int_0^T\int_{\R}\int_{0}^1 \sigma(u(t,x,\alpha)) \beta^\prime( u(t,x,\alpha))\,\phi(t,x)\,d\alpha\, dx\,dW(t) \Big] \notag \\
     & + \frac{1}{2} \E\Big[\mathds{1}_{B}  \int_0^T \int_{\R} \int_{0}^1
 \sigma^2(u(t,x,\alpha))\beta^{\prime\prime}(u(t,x,\alpha)) \phi(t,x)\,d\alpha\,dx\,dt \Big] \notag \\
     & + \E\Big[\mathds{1}_{B}\int_0^T \int_{|z|>0} \int_{\R} \int_{0}^1 \Big\{ \beta \big(u(t,x,\alpha) + \eta(u(t,x,\alpha);z)\big)
     -\beta(u(t,x,\alpha))\Big\}\phi(t,x)\,d\alpha \,dx \, \widetilde{N}(dz,dt)\Big] \notag \\
 & + \E\Big[\mathds{1}_{B}  \int_0^T \int_{|z|>0} \int_{\R} \int_{0}^1 \int_{0}^1 (1-\lambda) 
 \eta^2(u(t,x,\alpha);z)\beta^{\prime\prime}\big(u(t,x,\alpha) + \lambda\, \eta(u(t,x,\alpha);z)\big) \notag \\
 &\hspace{10cm} \times \phi(t,x)\,d\alpha\,d\lambda \,dx\,m(dz)\,dt \Big]. \notag 
 \end{align} 
 In other words,
\begin{align}
   &  \int_{\R}\beta(u_0(x))\phi(0,x)\,dx + 
      \int_{\Pi_T} \int_{0}^1 \Big(\partial_t \phi(t,x)\beta(u(t,x,\alpha))
      +  \partial_x \phi(t,x) \zeta(u(t,x,\alpha))\Big)\, d\alpha\,dx\,dt \notag \\
      & +  \int_0^T\int_{\R}\int_{0}^1 \sigma(u(t,x,\alpha)) \beta^\prime( u(t,x,\alpha))\,\phi(t,x)\,d\alpha\, dx\,dW(t) \notag \\
     & + \frac{1}{2}  \int_0^T \int_{\R} \int_{0}^1
 \sigma^2(u(t,x,\alpha))\beta^{\prime\prime}(u(t,x,\alpha)) \phi(t,x)\,d\alpha\,dx\,dt \notag \\
 & + \int_0^T \int_{|z|>0} \int_{\R} \int_{0}^1 \Big\{ \beta\big(u(t,x,\alpha) + \eta(u(t,x,\alpha);z)\big)
     -\beta(u(t,x,\alpha))\Big\} \phi(t,x)\,d\alpha \,dx \,\widetilde{N}(dz,dt) \notag \\
 & +  \int_0^T \int_{|z|>0} \int_{\R} \int_{0}^1 \int_{0}^1 (1-\lambda) 
 \eta^2(u(t,x,\alpha);z)\beta^{\prime\prime}\big(u(t,x,\alpha) + \lambda\, \eta(u(t,x,\alpha);z)\big) \notag \\
 &\hspace{10cm} \times \phi(t,x)\,d\alpha \,d\lambda\,dx\,m(dz)\,dt \notag \\
  & \ge 0\quad \mathbb{P}- \text{a.s.}\notag
\end{align}
Again, in view of the uniform moment estimate \eqref{uni:moment} along with \cite[Remark $2.4$]{BaVaWit}, we see that 
\[ \sup_{0\le t\le T} \E\Big[||u(t,\cdot,\cdot)||_2^2 \Big] < \infty.\] 
This implies that $u(t,x,\alpha)$ is a generalized entropy solution 
to the problem \eqref{eq:stoc_con_laws}. Again, thanks to the uniqueness result \cite{BaVaWit,BisKarlMaj}, we conclude that
 $u(t,x,\alpha)$ is an independent function of variable $\alpha$ and $\bar{u}(t,x)=\int_0^1 u(t,x,\tau) d\tau = u(t,x,\alpha)$ (for all $\alpha$) is the unique stochastic entropy solution. Moreover, since $u_{\Delta x}$ is bounded in $L^\infty(\Omega \times \Pi_T)$, we conclude that $u_{\Delta x}$  converges to $\bar{u}$ in 
 $L_{\text{loc}}^p(\R; L^p (\Omega \times (0,T))$, for $1\le p < \infty$. 
 
%%%%%%%%%%%%%%%%%%%%%%%%%%%%%%%%%%%%%%%%%%%%%%%%%%%%%%%%%%%%%

 \end{document}